\documentclass[a4paper]{article}
\usepackage[T1]{fontenc}
\usepackage{amsfonts,amsmath,amsthm,amssymb}
\usepackage{calc}
\usepackage{gensymb}
\usepackage[nottoc, notlof, notlot]{tocbibind}
\usepackage{fullpage}
\usepackage{mathrsfs}  
\usepackage{bigints}
\usepackage{indentfirst}
\usepackage{enumerate}
\usepackage{authblk}
\usepackage{hyperref}

\numberwithin{equation}{section}
\DeclareMathOperator{\tr}{Tr}
\DeclareMathOperator{\ts}{tr}
\DeclareMathOperator{\cov}{Cov}

\DeclareMathOperator{\supp}{Supp}
\DeclareMathOperator{\dep}{depth}
\DeclareMathOperator{\id}{id}
\let\limsup\relax
\DeclareMathOperator*{\limsup}{limsup}

\newcommand{\norm}[1]{\left\Vert #1\right\Vert}

\begin{document}

\newtheorem{theorem}{Theorem} [section]
\newtheorem{prop}[theorem]{Proposition} 
\newtheorem{defi}[theorem]{Definition} 
\newtheorem{exe}[theorem]{Example} 
\newtheorem{lemma}[theorem]{Lemma} 
\newtheorem{rem}[theorem]{Remark} 
\newtheorem{cor}[theorem]{Corollary} 
\newtheorem{conj}[theorem]{Conjecture}
\renewcommand\P{\mathbb{P}}
\newcommand\E{\mathbb{E}}
\newcommand\N{\mathbb{N}}
\newcommand\1{\mathbf{1}}
\newcommand\C{\mathbb{C}}
\newcommand\CC{\mathcal{C}}
\newcommand\M{\mathbb{M}}
\newcommand\R{\mathbb{R}}
\newcommand\U{\mathbb{U}}
\newcommand\A{\mathcal{A}}
\newcommand\B{\mathcal{B}}
\newcommand\F{\mathcal{F}}
\renewcommand\i{\mathbf{i}}
\renewcommand\S{\mathcal{S}}
\renewcommand\d{\partial_i}
\newcommand\PP{\mathcal{P}}

\def\etc{,\dots ,}

\title{Asymptotic expansion of smooth functions in polynomials in deterministic matrices and iid GUE matrices}

\date{}

\author[1,2]{F\'elix Parraud}
\affil[1]{\small Universit\'e de Lyon, ENSL, UMPA, 46 all\'ee d'Italie, 69007 Lyon.}
\affil[2]{\small Department of Mathematics, Graduate School of Science, Kyoto University, Kyoto 606-8502, Japan.}

\maketitle

\noindent E-mail of the corresponding author: \href{mailto:felix.parraud@ens-lyon.fr}{felix.parraud@ens-lyon.fr}

\noindent Data sharing not applicable to this article as no datasets were generated or analysed during the current study.

\noindent The authors have no relevant financial or non-financial interests to disclose. \\

\begin{abstract}
		
	Let $X^N$ be a family of $N\times N$ independent GUE random matrices, $Z^N$ a family of deterministic matrices, $P$ a self-adjoint noncommutative polynomial, that is for any $N$, $P(X^N,Z^N)$ is self-adjoint, $f$ a smooth function. We prove that for any $k$, if $f$ is smooth enough, there exist deterministic constants $\alpha_i^P(f,Z^N)$ such that
	$$ \mathbb{E}\left[\frac{1}{N}\text{Tr}\left( f(P(X^N,Z^N)) \right)\right]\ =\ \sum_{i=0}^k \frac{\alpha_i^P(f,Z^N)}{N^{2i}}\ +\ \mathcal{O}(N^{-2k-2}) .$$
	Besides, the constants $\alpha_i^P(f,Z^N)$ are built explicitly with the help of free probability. In particular, if $x$ is a free semicircular system, then when the support of $f$ and the spectrum of $P(x,Z^N)$ are disjoint, $\alpha_i^P(f,Z^N)=0$ for all $i\in\N$. As a corollary, we prove that given $\alpha<1/2$, for $N$ large enough, every eigenvalue of $P(X^N,Z^N)$ is $N^{-\alpha}$-close to the spectrum of $P(x,Z^N)$. 
	
\end{abstract}

\section{Introduction}

Asymptotic expansions in Random Matrix Theory created  connections between  different worlds, including  topology, statistical mechanics, and quantum field theory. In mathematics, a breakthrough  was made in 1986 in \cite{harerzag} by Harer and Zagier who used the large dimension expansion of the moments of Gaussian matrices to compute the Euler characteristic of the moduli space of curves. A good introduction to this topic is  given in the survey \cite{zvonski} by  Zvonkin.  In physics,  the seminal works of t'Hooft \cite{T_ouf} and Br\'ezin, Parisi, Itzykson and Zuber \cite{parisi} related matrix models with the enumeration of maps of any genus, hence providing a purely analytical tool to solve these hard combinatorial problems.  Considering matrices in interaction via a potential, the so-called matrix models, indeed allows us to consider the enumeration of maps with several vertices, including a possible coloring of the edges when the matrix model contains several matrices. This relation allowed to associate matrix models to statistical models on random graphs \cite{macl, cherbin, segala1, segala2,segala3}, as well as in \cite{segalaU1} and \cite{segalaU2} for the unitary case. This was also extended to the so-called $\beta$-ensembles in \cite{cherbin2,betaens1, borot1,borot2,borot4,borot5}. Among other objects, these works study correlation functions and the so-called free energy and show that they expand as power series in the inverse of the dimension, and the coefficients of these expansions enumerate maps sorted by their genus. To compute asymptotic expansions, often referred to in the literature as topological expansions, one of the most successful methods is the loop equations method, see \cite{debut} and \cite{debut2}. Depending on the model of random matrix, those are Tutte's equations, Schwinger-Dyson equations, Ward identities, Virasoro constraints, W-algebra or simply integration by parts. This method was refined and used repeatedly in physics, see for example the work of Eynard and his collaborators, \cite{eynard1,eynard2,betaens2,borot3}.  At first those equations were only solved for the first few orders, however in 2004, in \cite{eynard2} and later \cite{further} and \cite{further2}, this method was refined to push the expansion to any order recursively \cite{anci}. 

In this paper we want to generalize Harer-Zagier expansion for the moments of Gaussian matrices to more general smooth functions. Instead of a single GUE matrix, we will consider several independent matrices and deterministic matrices. We repeatedly use Schwinger-Dyson equations associated to GUE matrices to carry out our estimates. While we do not use the link between the coefficients of our expansion and map enumeration, as a corollary we get a new expression of these combinatorial objects. We show that the number of colored maps of genus $g$ with a single specific vertex can be expressed as an integral, see remark \ref{3map} for a precise statement.

Most papers quoted above have in common that they deal with polynomials or exponentials of polynomial evaluated in random matrices. With a few exceptions, such as \cite{macl} and \cite{precurso}, smooth functions have not been considered. However, being able to work with such functions is important for the applications. In particular we need to be able to work with functions with compact support to prove strong convergence results, that is proving the convergence of the spectrum for the Hausdorff distance. In this paper we establish a finite expansion of any order around the dimension of the random matrix for the trace of smooth functions evaluated in polynomials in independent GUE random matrices. We refer to Definition \ref{3GUEdef} for a definition of those objects. The link between maps and topological expansion is a good motivation to prove such kind of theorem. Another motivation is to study the spectrum of polynomials in these random matrices: because we consider general smooth functions, our expansion will for instance allow to study the spectrum outside of the limiting bulk. In the case of a single GUE matrix, we have an explicit formula for the distribution of the eigenvalues of those random matrices, see Theorem 2.5.2 of \cite{alice}. However, if we consider polynomials in independent GUE matrices, we have no such result. The first result in this direction dates back to 1991 when Voiculescu proved in \cite{Vo91} that the renormalized trace of such polynomials converges towards a deterministic limit $\alpha(P)$. In particular given $X_1^N,\dots,X_d^N$ independent GUE matrices, the following holds true almost surely:
\begin{equation}\label{3dv} \lim_{N\to \infty} \frac{1}{N}\tr_N\left( P(X_1^N,\dots,X_d^N) \right) = \alpha(P).\end{equation}

\noindent Voiculescu computed the limit $\alpha(P)$ with the help of free probability. Besides, if $A_N$ is a self-adjoint matrix of size $N$, then one can define the empirical measure of its (real) eigenvalues by 
$$ \mu_{A_N} = \frac{1}{N} \sum_{ i=1}^N \delta_{\lambda_i} \,$$

\noindent where $\delta_{\lambda}$ is the Dirac mass in $\lambda$ and $\lambda_1\etc \lambda_N$ are the eigenvalue of $A_N$. In particular, if $P$ is a self-adjoint polynomial, that is such that for any self adjoint matrices $A_1\etc A_d$, $P(A_1\etc A_d)$ is a self-adjoint matrix, then one can define the random measure $\mu_{P(X_1^N,\dots,X_d^N)}$. In this case, Voiculescu's result \eqref{3dv} implies that there exists a measure $\mu_P$ with compact support such that almost surely $\mu_{P(X_1^N,\dots,X_d^N)}$ converges weakly towards $\mu_P$ : it is given by $\mu_P(x^k)=\alpha(P^k)$ for all integer numbers $k$. Consequently, assuming we can apply the Portmanteau theorem, the proportion of eigenvalues of $A_N = P(X_1^N,\dots,X_d^N)$ in the interval $[a,b]$, that is $\mu_{A_N}([a,b])$, converges towards $\mu_P([a,b])$.

Therefore in order to study the eigenvalues of a random matrix one has to study the renormalized trace of its moments. However, if instead of studying the renormalized trace of polynomials in $A_N$, we study the non-renormalized trace of smooth function in $A_N$, then we can get precise information on the location of the eigenvalues. It all comes from the following remark, let $f$ be a non-negative function such that $f$ is equal to $1$ on the interval $[a,b]$, then if $\sigma(A_N)$ is the spectrum of $A_N$,
$$ \P\Big( \sigma(A_N)\cap [a,b] \neq \emptyset \Big) \leq \P\Big( \tr_{N}\left( f(A_N) \right)\geq 1 \Big) \leq \E\Big[ \tr_{N}\left( f(A_N) \right) \Big] .$$

\noindent Thus, if one can show that the right-hand side of this inequality converges towards zero when $N$ goes to infinity, then asymptotically there is no eigenvalue in the segment $[a,b]$. In the case of the random matrices that we study in this paper, that is polynomials in independent GUE matrices, a breakthrough was made in 2005 by Haagerup and Thorbj\o rnsen in \cite{HT}. They proved the almost sure convergence of the norm of those matrices. More precisely, they proved that for $P$ a self-adjoint polynomial, almost surely, for any $\varepsilon>0$ and $N$ large enough, 
\begin{equation}
\label{3spec}
\sigma\left( P(X_1^N,\dots,X_d^N) \right) \subset \supp \mu_P + (-\varepsilon,\varepsilon) \,
\end{equation}

\noindent where $\supp \mu_P$ is the support of the measure $\mu_P$. In order to do so, they showed that given a smooth function $f$, there is a constant $\alpha_0^P(f)$, which can be computed explicitly with the help of free probability, such that
$$ \E\Big[ \frac{1}{N} \tr_{N}\left( f(A_N) \right) \Big] = \alpha_0^P(f) + \mathcal{O}(N^{-2}) .$$

\noindent A similar equality was proved in \cite{un} with a better estimation of the dependency in the parameters such as $f$ and $Z^N$ in the $\mathcal{O}(N^{-2})$. Given the important consequences that studying the first two orders had, one can wonder what happens at the next order. More precisely, could we write this expectation as a finite order Taylor expansion, and what consequences would it have on the eigenvalues? That is, can we prove that for any $k$, if $f$ is smooth enough, there exist deterministic constants $\alpha_i^P(f)$ such that
$$ \mathbb{E}\left[\frac{1}{N}\tr_N\left( f(P(X_1^N,\dots,X_d^N)) \right)\right]\ =\ \sum_{i=0}^k \frac{\alpha_i^P(f)}{N^{2i}}\ +\ \mathcal{O}(N^{-2k-2}) ?$$

\noindent In 2002, by using Riemann-Hilbert techniques, Ercolani and McLaughlin gave in \cite{macl} a positive answer for the case of a single random matrix (but not necessarily a GUE random matrix), that is $d=1$. Haagerup and Thorbj\o rnsen gave a simplified proof in 2010 (see \cite{precurso}) for the specific case of a single GUE matrix. However, the method of the proof relied heavily on the explicit formula of the law of the eigenvalues of a GUE matrix and since there is no equivalent for polynomials in GUE matrices we cannot adapt this proof. Instead, we developed a proof whose main tool is free probability. The main idea of the proof is to interpolate independent GUE matrices and free semicirculars with free Ornstein-Uhlenbeck processes. However, thanks to a computation trick we only need to work with the marginals of this process at a given time $t$, and since it is well-known that the law of this random variable can be viewed as a clever interpolation between the process at time $0$ and its limit, we do not need to introduce notions of free stochastic calculus. For more details we refer to \cite[subsection 3.1]{un}. This means in particular that we do not need to define any notion of free stochastic calculus. The main result of this paper is the following theorem.

\begin{theorem}
	\label{3lessopti}
	We define,
	\begin{itemize}
		\item $X^N = (X_1^N,\dots,X_d^N)$ independent $GUE$ matrices of size $N$,
		\item $Z^N = (Z_1^N,\dots,Z_r^N, {Z_1^N}^*,\dots,{Z_r^N}^*)$ deterministic matrices whose norm is uniformly bounded over $N$,
		\item $P$ a self-adjoint polynomial which can be written as a linear combination of $\mathbf{m}$ monomials of degree at most $n$ and coefficients with an absolute value of at most $c_{\max}$,
		\item $f:\R\mapsto\R$ a function of class $\mathcal{C}^{4(k+1)+2}$.
	\end{itemize}
	
	\noindent We set $\norm{f}_{\infty}$ the supremum over $\R$ of $f$, and 
	$$ \norm{f}_{\CC^i} = \sum_{l=0}^i \norm{f^{(l)}}_{\infty}. $$
	Then there exist deterministic coefficients $(\alpha_i^P(f,Z^N))_{1\leq i\leq k}$ and constants $C,K$ and $c$ independent of $P,f,N$ and $k$, such that with $K_N = \max \{ \norm{Z^N_1}, \dots, \norm{Z^N_q}, K\}$, $C_{\max}(P) = \max \{1, c_{\max}\}$, for any $N$, if $k\leq cN n^{-1}$,
	\begin{align}
	\label{3mainresu0}
	&\left| \E\left[ \frac{1}{N}\tr_N\Big(f(P(X^N,Z^N))\Big)\right] - \sum_{0\leq i\leq k} \frac{1}{N^{2i}} \alpha_i^P(f,Z^N) \right| \\
	&\leq \frac{1}{N^{2(k+1)}} \norm{f}_{\mathcal{C}^{4(k+1) +2}} \times \Big(C\times n^2 K_N^{n} C_{\max} \mathbf{m} \Big)^{4(k+1)+1}\times k^{12k} . \nonumber
	\end{align}
	
	\noindent Besides, if we define $\widehat{K}_N$ like $K_N$ but with $2$ instead of $K$, then we have that for any $i$,
	\begin{equation}
	\label{3mainresu02}
	\left| \alpha_i^P(f,Z^N) \right| \leq \norm{f}_{\mathcal{C}^{4i+2}} \times \Big(C\times n^2 \widehat{K}_N^{n} C_{\max} \mathbf{m} \Big)^{4i+1}\times i^{12i} .
	\end{equation}
	Finally if $f$ and $g$ are functions of class $\mathcal{C}^{4(k+1)+2}$ equal on a neighborhood of the spectrum of $P(x,Z^N)$, where $x$ is a free semicircular system free from $\M_N(\C)$, then for any $i\leq k$, $\alpha_i^P(f,Z^N) = \alpha_i^P(g,Z^N)$. In particular if the support of $f$ and the spectrum of $P(x,Z^N)$ are disjoint, then for any $i$, $\alpha_i^P(f,Z^N)=0$.
	
\end{theorem}

This theorem is a consequence of the slightly sharper, but less explicit, Theorem \ref{3TTheo}. It is essentially the same statement, but instead of having the $\mathcal{C}^k$-norm of $f$, we work with the moments of the Fourier transform of $f$. We also give an explicit expression for the coefficients $\alpha_i^P$. The above theorem calls for a few remarks.

\begin{itemize}
	\item  In Theorem \ref{3lessopti}, we only considered a single function $f$ evaluated in a self-adjoint polynomial $P$. However, one could easily adapt the proof of Theorem \ref{3TTheo} to consider a product of functions $f_i$ evaluated in self-adjoint polynomials $P_i$ and get a similar result. The main difference would be that instead of $\norm{f}_{\mathcal{C}^{4(k+1) +2}}$ one would have $ \max_i \norm{f_i}_{\mathcal{C}^{4(k+1) +2}}$. One could also adapt the proof to deal with the case of a product of traces. We give more details about those two situations in Remark \ref{produf}.
	\item The coefficients $(\alpha_i^P(f,Z^N))_{1\leq i\leq k}$ are continuous with respect to all of their parameters, $f,Z^N$ and $P$. We give a precise statement in Corollary \ref{continucoeff}. In particular if $Z^N$ converges in distribution when $N$ goes to infinity (as defined in Definition \ref{3freeprob}) towards a family of noncommutative random variables $Z$, then for every $i$, $\alpha_i^P(f,Z^N)$ converges towards $\alpha_i^P(f,Z)$.
	\item We assumed that the matrices $Z^{N}$ are deterministic, but thanks to Fubini's theorem we can assume that they are random matrices as long as they are independent from $X^N$. In this situation though, $K_N^n$ in the right side of the inequality is a random variable (and thus we need some additional assumptions if we want its expectation to be finite for instance).
	\item We assumed that the matrices $Z^N$ were uniformly bounded over $N$. This is a technical assumption which is necessary to make sure that the coefficients $\alpha_i^P$ are well-defined. However, as we can see in Theorem \ref{3TTheo}, one can relax this assumption. That being said, in order for Equation \eqref{3mainresu0} to be meaningful one has to be careful that the term ${K_N^n}^4$ does not overwhelm the term $N^{-2}$.
	\item The exponent $12$ in the term $k^{12k}$ is very suboptimal and could easily be optimized a bit more in the proof of Theorem \ref{3lessopti}. For a better bound we refer to Theorem \ref{3TTheo}, where the term $k^{12k}$ is replaced by $k^{3k}$. However, in order to work with the $\mathcal{C}^k$-norm instead of the moments of the Fourier transform, we were forced to increase this term.
	\item Although we cannot take $k=\infty$, and hence we only get a finite Taylor expansion, we can still take $k$ which depends on $N$. However, to keep the last term under control we need to estimate the $k$-th derivative of $f$.
	\item Since the probability that there is an eigenvalue of $P(X^N,Z^N)$ outside of a neighborhood of $P(x,Z^N)$ is exponentially small as $N$ goes to infinity, the smoothness assumption on $f$ only needs to be verified on a neighborhood of $P(X^N,Z^N)$ for such an asymptotic expansion to exist.
\end{itemize}

As we said earlier in the introduction, by studying the trace of a smooth function evaluated in $P(X_1^N,\dots,X_d^N)$, Haagerup and Thorbj\o rnsen were able to show in \cite{HT} that the spectrum of $P(X_1^N,\dots,$ $ X_d^N)$ converges for the Hausdorff distance towards an explicit subset of $\R$. We summarized this result in Equation \eqref{3spec}. With the full finite order Taylor expansion, by taking $f: x\to g(N^{\alpha}x)$ where $g$ is a well-chosen smooth function, one can show the following proposition.

\begin{cor}
	\label{3voisinage}
	Let $X^N$ be independent $GUE$ matrices of size $N$, $A^N=(A_1^N,\dots,A_r^N, {A_1^N}^*,\dots,{A_r^N}^*)$ a family of deterministic matrices whose norm is uniformly bounded over $N$, $x$ be a free semicircular system and $P$ a self-adjoint polynomial. Given $\alpha< 1/2$, almost surely for $N$ large enough, 
	$$ \sigma\left( P(X^N,A^N) \right) \subset \sigma\left( P(x,A^N) \right) + (-N^{-\alpha},N^{-\alpha}) ,$$
	where $\sigma(X)$ is the spectrum of $X$, and $x$ is free from $\M_N(\C)$.
\end{cor}

\noindent In the case of a single GUE matrix, much more precise results were obtained by Tracy and Widom in \cite{TW1}. They proved the existence of a continuous decreasing function $F_2$ from $\R$ to $[0,1]$ such that if $\lambda_1(X^N)$ denotes the largest eigenvalue of $X^N$,
$$ \lim_{N\to \infty} P\big(N^{2/3} (\lambda_1(X^N) - 2) \geq s\big) = F_2(s).$$ 

\noindent This was generalized to $\beta$-matrix models in \cite{figalli1} and to polynomials in independent GUE matrices which are close to the identity in \cite{figalli2}. But there is no such result for general polynomials in independent GUE matrices. However, with Theorem \ref{3lessopti} we managed to get an estimate on the tail of the distribution of $\sqrt{N} \norm{P(X^N,A^N)}$.

\begin{cor}
	\label{3boundednormrenm}
	Let $X^N$ be a family of independent $GUE$ matrices of size $N$, $A^N = (A_1^N,\dots,A_r^N, {A_1^N}^*,$ $\dots, {A_r^N}^*)$ a family of deterministic matrices whose norm is uniformly bounded over $N$, $x$ a free semicircular system and $P$ a polynomial. Then there exists a constant $C$ such that for any $\delta>0$ and $N$ large enough,
	$$ \P\left( \frac{\sqrt{N}}{\ln^4 N} \left(\norm{P(X^N,A^N)} - \norm{P(x,A^N)}\right) \geq C\frac{\delta + 1}{\norm{P(x,A^N)}} \right) \leq e^{-N} + e^{-\delta^2 \ln^8N} . $$
\end{cor}

This corollary is similar to Theorem 1.5 obtained in \cite{un}, but with a substantial improvement on the exponent since in this theorem, instead of $N^{1/2}$, we only had $N^{1/4}$. Theorem 1.5 of \cite{un} also gave a similar bound on the probability that $\norm{P(X^N)}$ be smaller than its deterministic limit, but Theorem \ref{3TTheo} does not yield any improvement on this inequality. The proof of this corollary can be summarized in two steps: first use measure concentration to get an estimate on the probability that $\norm{P(X^N,A^N)}$ is far from its expectation, and secondly use Theorem \ref{3lessopti} to estimate the difference between the expectation and the deterministic limit. Finally it is worth noting that the exponent $1/2$ comes from the fact that for every $N^2$ that we gain in Equation \eqref{3mainresu0}, we also have to differentiate our function $f$ four more times. Thus, if we take $f:x\to g(N^{\alpha}x)$ where $g$ is smooth, then in order for $N^{-2}$ to compensate the differential, we have to take $\alpha = 1/2$. If we only had to differentiate our function three more times, then we could take $\alpha = 2/3$ which is the same exponent as in Tracy-Widom.

\section{Framework and standard properties}
\label{3definit}

\subsection{Usual definitions in free probability}
\label{3deffree}

In order to be self-contained, we begin by recalling the following definitions from free probability.

\begin{defi}~
	\label{3freeprob}
	\begin{itemize}
		\item A \textbf{$\mathcal{C}^*$-probability space} $(\A,*,\tau,\norm{.}) $ is a unital $\mathcal{C}^*$-algebra $(\A,*,\norm{.})$ endowed with a \textbf{state} $\tau$, i.e. a linear map $\tau : \A \to \C$ satisfying $\tau(1_{\A})=1$ and $\tau(a^*a)\geq 0$ for all $a\in \A$. In this paper we always assume that $\tau$ is a \textbf{trace}, i.e. that it satisfies $\tau(ab) = \tau(ba) $ for any $a,b\in\A$. An element of $\A$ is called a 
		\textbf{noncommutative random variable}. We will always work with a faithful trace, namely, for $a\in\A$, $\tau(a^*a)=0$ if and only if $a=0$.
		
		\item Let $\A_1,\dots,\A_n$ be $*$-subalgebras of $\A$, having the same unit as $\A$. They are said to be \textbf{free} if for all $k$, for all $a_i\in\A_{j_i}$ such that $j_1\neq j_2$, $j_2\neq j_3$, \dots , $j_{k-1}\neq j_k$:
		\begin{equation}
			\label{kddkdxkfl}
			\tau\Big( (a_1-\tau(a_1))(a_2-\tau(a_2))\dots (a_k-\tau(a_k)) \Big) = 0.
		\end{equation}
		Families of noncommutative random variables are said to be free if the $*$-subalgebras they generate are free.
		
		\item Let $ A= (a_1,\ldots ,a_k)$ be a $k$-tuple of random variables. The \textbf{joint $*$-distribution} of the family $A$ is the linear form $\mu_A : P \mapsto \tau\big[ P(A, A^*) \big]$ on the set of polynomials in $2k$ noncommutative variables. By \textbf{convergence in distribution}, for a sequence of families of variables $(A_N)_{N\geq 1} = (a_{1}^{N},\ldots ,a_{k}^{N})_{N\geq 1}$ 
		in $\mathcal C^*$-algebras $\big( \mathcal A_N, ^*, \tau_N, \norm{.} \big)$,
		we mean the pointwise convergence of
		the map 
		$$ \mu_{A_N} : P \mapsto \tau_N \big[ P(A_N, A_N^*) \big],$$
		and by \textbf{strong convergence in distribution}, we mean convergence in distribution, and pointwise convergence
		of the map
		$$P \mapsto \big\| P(A_N, A_N^*) \big\|.$$
		
		\item A family of noncommutative random variables $ x=(x_1,\dots ,x_d)$ is called a \textbf{free semicircular system} when the noncommutative random variables are free, self-adjoint ($x_i=x_i^*$), and for all $k$ in $\N$ and $i$, one has
		\begin{equation*}
		\tau( x_i^k) =  \int_{\R} t^k d\sigma(t),
		\end{equation*}
		with $d\sigma(t) = \frac 1 {2\pi} \sqrt{4-t^2} \ \mathbf 1_{|t|\leq2} \ dt$ the semicircle distribution.
		
	\end{itemize}
	
\end{defi}

It is important to note that thanks to \cite[Theorem 7.9]{nica_speicher_2006}, which we recall below, one can consider free copies of any noncommutative random variable.

\begin{theorem}
	\label{3freesum}
	
	Let $(\A_i,\phi_i)_{i\in I}$ be a family of $\mathcal{C}^*$-probability spaces such that the functionals $\phi_i : \A_i\to\C$, $i\in I$, are faithful traces. Then there exist a $\mathcal{C}^*$-probability space $(\A,\phi)$ with $\phi$ a faithful trace, and a family of norm-preserving unital $*$-homomorphism $W_i: \A_i\to\A$, $i\in I$, such that:
	
	\begin{itemize}
		\item $\phi \circ W_i = \phi_i$, $\forall i \in I$.
		\item The unital $\mathcal{C}^*$-subalgebras $W_i(\A_i)$, $i\in I$, form a free family in $(\A,\phi)$.
	\end{itemize}
\end{theorem}

Let us finally fix a few notations concerning the spaces and traces that we use in this paper.

\begin{defi}
	\label{3tra}
	\begin{itemize}
		\item $(\A_N,\tau_N)$ is the free product $\M_N(\C) * \mathcal{C}_d$ of $\M_N(\C)$ with $\mathcal{C}_d$ the $\CC^*$-algebra generated by a system of $d$ free semicircular variables, that is the $\mathcal{C}^*$-probability space built in Theorem \ref{3freesum}. Note that when restricted to $\M_N(\C)$, $\tau_N$ is just the renormalized trace on matrices. The restriction of $\tau_{N}$ to the $\mathcal{C}^*$-algebra generated by the free semicircular system $x$ is denoted by $\tau$. Note that one can view this space as the limit of a matrix space, we refer to \cite[Proposition 3.5]{un}. 
		\item $\tr_N$ is the non-renormalized trace on $\M_N(\C)$.
		\item $\ts_N$ is the renormalized trace on $\M_N(\C)$.
		\item We denote $E_{r,s}$ the matrix with $1$ in the $(r,s)$ entry and zeros in all the other entries. 
		\item We regularly identify $\M_N(\C)\otimes \M_k(\C)$ with $\M_{kN}(\C)$ through the isomorphism $E_{i,j}\otimes E_{r,s} \mapsto E_{i+rN,j+sN} $, similarly we identify $\tr_N\otimes\tr_k$ with $\tr_{kN}$.
		\item $\id_N\otimes \ts_k$ is the conditional expectation from $\M_N(\C)\otimes \M_k(\C)$  to $\M_N(\C)$. Basically it is the tensor product of the identity map $\id_N:\M_N(\C)\to\M_N(\C)$ and the renormalized trace on $\M_k(\C)$.
		\item If $A^N=(A_1^N,\dots,A_d^N)$ and $B^k=(B_1^k,\dots,B_d^k)$ are two families of random matrices, then we denote $A^N\otimes B^k=(A_1^N\otimes B^k_1,\dots,A_d^N\otimes B^k_d)$. We typically use the notation $X^N\otimes I_k$ for the family $(X^N_1\otimes I_k,\dots,X^N_1\otimes I_k)$.
	\end{itemize}
\end{defi}

\subsection{Noncommutative polynomials and derivatives}
\label{3poly}

Let $\PP_{d,2r}=\C\langle X_1,\dots,X_d,Y_1,\dots,Y_{2r}\rangle$ be the set of noncommutative polynomials in $d+2r$ variables. We set $q=2r$ to simplify notations. For any fixed $A\in\R_+^*$, one defines

\begin{equation}
\label{3normA}
\norm{P}_A = \sum_{M \text{ monomial}} |c_M(P)| A^{\deg M} \,
\end{equation}

\noindent where $c_M(P)$ is the coefficient of $P$ for the monomial $M$ and $\deg M$ the total degree of $M$ (that is the sum of its degree in each letter $X_1,\dots,X_d,Y_1,\dots,Y_{2r}$). Let us define several maps which we use frequently in the sequel. First, for $A,B,C\in \PP_{d,q}$, let

\begin{equation}
\label{3defperdu}
A\otimes B \# C = ACB,\ A\otimes B \widetilde{\#} C = BCA,\ m(A\otimes B) = BA.
\end{equation}

\noindent We define an involution $*$ on $\PP_{d,q}$ by $X_i^* = X_i$, $Y_i^* = Y_{i+r}$ if $1\leq i\leq r$, $Y_i^* = Y_{i-r}$ else, and then we extend it to $\PP_{d,q}$ by linearity and the formula $(\alpha P Q)^* = \overline{\alpha} Q^* P^*$. $P\in \PP_{d,q}$ is said to be self-adjoint if $P^* = P$. Self-adjoint polynomials have the property that if $x_1,\dots,x_d,z_1,\dots,z_r$ are elements of a $\mathcal{C}^*$-algebra such that $x_1,\dots,x_d$ are self-adjoint, then so is $P(x_1,\dots,x_d,z_1,\dots,z_r,z_1^*,\dots,z_r^*)$.

Finally let us define the noncommutative derivative, it is a widely used tool in the field of probability, see for example the work of Voiculescu, \cite{refdif} and \cite{refdif2}.

\begin{defi}
	\label{3application}
	
	If $1\leq i\leq d$, one defines the \textbf{noncommutative derivative} $\partial_i: \PP_{d,q} \longrightarrow \PP_{d,q} \otimes \PP_{d,q}$  by its value on a monomial  $M\in \PP_{d,q}$  given by
	$$ \partial_i M = \sum_{M=AX_iB} A\otimes B \,$$
	and then extend it by linearity to all polynomials. We can also define $\partial_i$ by induction with the formulas,
	\begin{equation}
	\label{3leibniz}
	\begin{array}{ccc}
	&\forall P,Q\in \mathcal{A}_{d,q},\quad \partial_i (PQ) = \partial_i P \times \left(1\otimes Q\right)	 + \left(P\otimes 1\right) \times \partial_i Q , \\
	& \\
	&\forall i,j,\quad \partial_i X_j = \delta_{i,j} 1\otimes 1,\quad \partial_i Y_j =0. \end{array}
	\end{equation}
	Similarly, with $m$ as in \eqref{3defperdu}, one defines the \textbf{cyclic derivative}  $D_i: \PP_{d,q} \longrightarrow \PP_{d,q}$ for $P\in \PP_{d,q}$ by
	$$ D_i P = m\circ \partial_i P \ . $$
	
\end{defi}

In this paper however, we will need to work not only with polynomials but also with more general functions, since we will work with the Fourier transform we introduce the following space.

\begin{defi}
	We define $\mathcal{S} = \left\{ R\in \mathcal{A}_{d,q}\ |\ R^*=R \right\}$, then one set
	$$\mathcal{F}_{d,q} = \C\big\langle (E_R)_{R\in\S}, X_1,\dots,X_d,Y_1,\dots,Y_{2r}\big\rangle.$$
	Then given  $ z = (x_1,\dots,x_d,y_1,\dots,y_r,y_1^*,\dots,y_r^*)$ elements of a $\CC^*$-algebra, one can define by induction the evaluation of an element of $\F_{d,q}$ in $z$ by following the following rules:
	\begin{itemize}
		\item $\forall Q\in\PP_{d,q}$, $Q(z)$ is defined as usual,
		\item $\forall Q_1,Q_2\in \F_{d,q}$, $(Q_1+Q_2)(z)= Q_1(z)+Q_2(z)$, $(Q_1Q_2)(z)= Q_1(z)Q_2(z)$,
		\item $\forall R\in\S$, $E_R(z) = e^{\i R(z)}$.
	\end{itemize}
	One can extend the involution $*$ from $\PP_{d,q}$ to $\mathcal{F}_{d,q}$ by setting $(E_R)^* = E_{(-R)} $, and then again we have that if $Q\in\F_{d,q}$ is self-adjoint, then so is $Q(z)$. Finally in order to make notations more transparent, we will usually write $e^{\i R}$ instead of $E_R$.
\end{defi}
Note that for technical reasons that we explain in Remark \ref{3quotient}, one cannot view $\mathcal{F}_{d,q}$ as a subalgebra of the set of formal power series in $X_1,\dots,X_d,Y_1,\dots,Y_{2r}$. This is why we need to introduce the notation $E_R$.

Now, as we will see in Proposition \ref{3duhamel}, a natural way to extend the definition of $\partial_i$ (and $D_i$) to $\F_{d,q}$ is by setting
\begin{equation}
\label{3ext}
\partial_i e^{\i Q} = \i \int_0^1 \big(e^{\i \alpha Q}\otimes 1\big)\ \partial_i Q\ \big(1\otimes e^{\i (1-\alpha) Q}\big) d\alpha .
\end{equation}

\noindent However, we cannot define the integral properly on $\mathcal{F}_{d,q}\otimes \mathcal{F}_{d,q}$. After evaluating our polynomials in $\CC^*$-algebras, the integral will be well-defined as we will see. Firstly, we need to define properly the operator norm of tensor of $\CC^*$-algebras. We work with the minimal tensor product also named the spatial tensor product. For more information we refer to \cite[Chapter 6]{murphy}.

\begin{defi}
	\label{3mini}
	Let $\A$ and $\B$ be $\CC^*$-algebra with faithful representations $(H_{\A},\phi_{\A})$ and $(H_{\B},\phi_{\B})$, then if $\otimes_2$ is the tensor product of Hilbert spaces, $\A\otimes_{\min}\B$ is the completion of the image of $\phi_{\A}\otimes\phi_{\B}$ in $B(H_{\A}\otimes_2 H_{\B})$ for the operator norm in this space. This definition is independent of the representations that we fixed.
\end{defi}

In particular, it is important to note that if $\A = \M_N(\C)$, then up to isomorphism $\A \otimes_{\min} \A$ is simply $\M_{N^2}(\C)$ with the usual operator norm. The main reason we pick this topology is for the following lemma. It is mainly a consequence of \cite[Lemma 4.1.8]{ozabr}.
\begin{lemma}
	\label{1faith}
	Let $(\A,\tau_{\A})$ and $(\B,\tau_{\B})$ be $\CC^*$-algebra with faithful traces, then $\tau_{\A}\otimes\tau_{\B}$ extends uniquely to a faithful trace $\tau_{\A}\otimes_{\min}\tau_{\B}$ on $\A\otimes_{\min}\B$. 
\end{lemma}

It is not necessary to understand in depth the minimal tensor product to read the rest of the paper. Indeed, we won't directly make use of this property in this paper, however it is necessary to introduce it to justify that every object in this paper is well-defined. Thus, we define the noncommutative differential on $\F_{d,q}$ as follows.

\begin{defi}
	\label{3technicality}
	For $\alpha\in [0,1]$, let $\partial_{\alpha,i}: \F_{d,q}\to \F_{d,q}\otimes \F_{d,q}$ which satisfies \eqref{3leibniz} and such that for any $P\in \mathcal{A}_{d,q}$ self-adjoint,
	$$ \partial_{\alpha,i} e^{\i P} = \i \big(e^{\i \alpha P}\otimes 1\big)\ \partial_i P\ \big(1\otimes e^{\i (1-\alpha) P}\big),\quad D_{\alpha,i} = m \circ \partial_{\alpha,i} . $$
	Then, given $ z = (z_1,\dots, z_{d+q})$ elements of a $\CC^*$-algebra, we define for any $Q\in \F_{d,q}$,
	$$ \partial_i Q(z) = \int_{0}^1 \partial_{\alpha,i} Q(z)\ d\alpha,\quad D_i Q(z) = \int_{0}^1 D_{\alpha,i} Q(z)\ d\alpha . $$
\end{defi}

\noindent Note that for any $P\in\PP_{d,q}$, since $\int_0^11d\alpha = 1$, we do also have that with $\partial_i Q$ defined as in Definition \ref{3application}, 
$$ \partial_i Q(z) = \int_{0}^1 \partial_{\alpha,i} Q(z)\ d\alpha .$$
Thus, Definition \ref{3technicality} indeed extends the definition of $\partial_i$ from $\PP_{d,q}$ to $\F_{d,q}$. Besides, it also means that we can define rigorously the composition of those maps. Since the map $\partial_{\alpha,i}$ goes from $\F_{d,q}$ to $\F_{d,q}\otimes \F_{d,q}$ it is very easy to do so. For example one can define the following operator. We will use a similar one later on, see also Example \ref{3exe33}.

\begin{defi}
	\label{3operatordef}
	Let $Q\in \F_{d,q}$, given $ z = (z_1,\dots, z_{d+q})$ elements of a $\CC^*$-algebra, let $i,j \in [1,d]$, with $\circ$ the composition of operators we define
	$$ (\partial_j\otimes \partial_j)\circ \partial_i\circ D_i Q(z) = \int_{[0,1]^4} (\partial_{\alpha_4,j}\otimes \partial_{\alpha_3,j})\circ\partial_{\alpha_2,i}\circ D_{\alpha_1,i} Q(z)\ d\alpha_1 d\alpha_2 d\alpha_3 d\alpha_4 .$$
\end{defi}

\noindent The definition \ref{3technicality} is the reason why one cannot view $\F_{d,q}$ as a subalgebra of the set of formal power series. More precisely one have the following remark.
\begin{rem}
\label{3quotient}
	In the set of formal power series, we have for example that $e^{X_1} e^{X_1} = e^{2X_1}$. However, when one defines the noncommutative differential $\partial_1$, then one must first define $\partial_{1,\alpha}$ such that
	$$ \partial_{1,\alpha} e^{X_1}e^{X_1} = e^{\alpha X_1}\otimes e^{(1-\alpha)X_1} e^{X_1} + e^{X_1} e^{\alpha X_1}\otimes e^{(1-\alpha)X_1},$$
	$$ \partial_{1,\alpha} e^{2X_1} = 2e^{2\alpha X_1}\otimes e^{2(1-\alpha)X_1} .$$
	And then we set for some element $x$ of a $\CC^*$-algebra,
	$$ \partial_{1} e^{X_1}e^{X_1}(x) = \int_{0}^{1}e^{\alpha X_1}\otimes e^{(1-\alpha)X_1} e^{x} + e^{x} e^{\alpha x}\otimes e^{(1-\alpha)x}\ d\alpha,$$
	$$ \partial_{1} e^{2X_1}(x) = 2\int_{0}^{1}e^{2\alpha x}\otimes e^{2(1-\alpha)x}\ d\alpha.$$
	And while with this construction we do have that $ \partial_{1} e^{X_1}e^{X_1}(x) = \partial_{1} e^{2X_1}(x)$, trying to define the noncommutative derivative on the set of power series forces us to consider this problem in all generality which we woud rather avoid. Besides, in any case, we do not have that $\partial_{1,\alpha} e^{X_1}e^{X_1}(x) = \partial_{1,\alpha} e^{2X_1}(x)$.
\end{rem}

If $P\in \PP_{d,q}$, $ z = (z_1,\dots, z_{d+q})$ belongs to a $\CC^*$-algebra $\A$, then we naturally have that
$$(\partial_i P^k) (z) = \sum_{l=1}^k \left(P^{l-1}(z)\otimes 1\right) \partial_iP(z) \left(1\otimes P^{k-l}(z)\right), $$
which is an element of $\A \otimes_{\min} \A$. Besides, there exists a constant $C_P(z)$ independent of $k$ such that $\norm{(\partial_i P^k) (z)} \leq C_P(z) k \norm{P(z)}^{k-1}$. Thus, it would seem natural to define
\begin{equation}
\label{3extension}
(\partial_i e^{P}) (z) = \lim\limits_{n\to\infty} \partial_i\left(\sum_{ 1\leq k \leq n} \frac{P^k}{k!} \right)(z) = \sum_{k\in \N} \frac{1}{k!} (\partial_i P^k) (z) ,
\end{equation}

\noindent as an element of $\A \otimes_{\min} \A$. It turns out that this definition is compatible with Definition \ref{3technicality} thanks to the following proposition (see \cite[Proposition 2.2]{deux} for the proof).

\begin{prop}
	\label{3duhamel}
	Let $P\in \PP_{d,q}$, $ z = (z_1,\dots, z_{d+q})$ elements of a $\CC^*$-algebra $\A$, then with $(\partial_i e^{P}) (z)$ defined as in \eqref{3extension},
	$$ \left(\partial_i e^P\right)(z) = \int_0^1 \left(e^{\alpha P(z)}\otimes 1\right)\ \partial_i P(z)\ \left(1\otimes e^{(1-\alpha)P(z)}\right) \ d\alpha . $$
	
\end{prop}

\vspace*{1cm}

For the sake of clarity, we introduce the following notation which is close to Sweedler's convention. Its interest will be clear in Section \ref{3mainsec}.
\begin{defi}
	\label{3sweedler}
	Let $Q\in \F_{d,q}$, $\mathcal{C}$ be a $\mathcal{C}^*$-algebra, $\alpha : \F_{d,q}\to \mathcal{C}$ and $\beta : \F_{d,q}\to \mathcal{C}$ be morphisms. We also set $\mathfrak{m} : A\otimes B \in\mathcal{C}\otimes\mathcal{C}\mapsto AB \in \mathcal{C}$. Then we use the following notation,
	$$ \alpha(\partial_i^1 P) \boxtimes \beta(\partial_i^2 P) = \mathfrak{m}\circ((\alpha\otimes\beta)(\partial_i P)) . $$
\end{defi}

\noindent Heuristically, if $\partial_i P$ was a simple tensor, then $\partial_i^1 P$ would represent the left tensorand while  $\partial_i^2 P$ would represent the right one. However $\partial_i P$ usually is not a simple tensor and one cannot extend this definition by linearity. This notation is especially useful when our maps $\alpha$ and $\beta$ are simply evaluation of $P$ as it is the case in Section \ref{3mainsec}. Indeed, we will typically write $\partial_i^1P (X) \boxtimes \partial_i^2P (Y)$ rather than first defining $h_X: P\to P(X)$ and using the more cumbersome and abstract notation, $ \mathfrak{m}\circ(h_X\otimes h_Y)(\partial_i P)$. 
\\

The map $\partial_i$ is related to the so-called Schwinger-Dyson equations on semicircular variable thanks to the following proposition. One can find a proof for polynomials in \cite[Lemma 5.4.7]{alice}, and then extend it to $\F_{d,q}$ thanks to Proposition \eqref{3duhamel} and Lemma \ref{1faith}.

\begin{prop}
	\label{3SDE}
	Let $ x=(x_1,\dots ,x_d)$ be a free semicircular system, $y = (y_1,\dots,y_r)$ be noncommutative random variables free from $x$, if the 
	family $(x,y)$ belongs to the $\mathcal{C}^*$-probability space $(\A,*,\tau,\norm{.}) $, then for any $Q\in \F_{d,q}$,
	$$ \tau(Q(x,y,y^*)\ x_i) = \tau\otimes\tau(\partial_i Q(x,y,y^*))\ .$$
	
\end{prop}

\noindent We conclude this subsection by a few concrete examples. Although one could have picked easier one, since we will typically be working with this kind of operators, it is important to understand it.

\begin{exe}
	\label{3exe11}
	If $Q\in\PP_{d,q}$ is a monomial, then one can try to compute 
	$$ R = \Big(\partial_{j}^2\left( \partial_i^1 Q\right) \boxtimes \partial_{j}^1 \left( \partial_i^1 Q\right) \Big) \boxtimes \Big(\partial_{j}^2\left( \partial_i^2 Q\right) \boxtimes \partial_{j}^1 \left( \partial_i^2 Q\right) \Big). $$
	To begin with one have that
	$$ \partial_i Q = \sum_{M=AX_iB} A\otimes B.$$
	Then naturally we also have that
	$$ \partial_j A = \sum_{M=A_1X_jA_2} A_1\otimes A_2,\quad \partial_j B = \sum_{B=B_1X_jB_2} B_1\otimes B_2.$$
	Consequently we get that
	$$ \partial_{j}^2 A \boxtimes \partial_{j}^1  A = \sum_{M=A_1X_jA_2} A_2 A_1,\quad \partial_{j}^2 B \boxtimes \partial_{j}^1  B = \sum_{M=B_1X_jB_2} B_2 B_1.$$
	Thus, we finally have that
	\begin{align*}
		R &= \sum_{M=AX_iB} \Big( \partial_{j}^2 A \boxtimes \partial_{j}^1  A \Big) \Big(\partial_{j}^2 B \boxtimes \partial_{j}^1  B\Big) \\
		&= \sum_{M=AX_iB} \left(\sum_{M=A_1X_jA_2} A_2 A_1\right)  \left(\sum_{M=B_1X_jB_2} B_2 B_1\right) \\
		&= \sum_{M=A_1X_jA_2X_iB_1X_jB_2} A_2A_1B_2B_1.
	\end{align*}
\end{exe}

\begin{exe}
	\label{3exe22}
	We set $Q = X_2X_1X_2^2X_1^2\in\PP_{2,0}$, let $w_1,x_1,y_1,z_1,w_2,x_2,y_2,z_2$ elements of a $\CC^*$-algebra $\A$. Let us compute the following quantity:
	$$ R = \Big(\partial^2_2\left( \partial^1_1 D_1 Q \right)(w) \boxtimes \partial^1_2 \left( \partial^1_1 D_1 Q \right)(x) \Big)  \boxtimes \Big(\partial^2_2\left( \partial^2_1 D_1 Q \right)(y) \boxtimes \partial^1_2 \left( \partial^2_1 D_1 Q \right)(z) \Big) .$$
	First and foremost, we have that 
	$$ \partial_1 Q = X_2\otimes X_2^2X_1^2 + X_2X_1X_2^2\otimes X_1 + X_2X_1X_2^2X_1\otimes 1 ,$$
	hence
	$$ D_1 Q =  X_2^2X_1^2X_2 + X_1X_2X_1X_2^2 + X_2X_1X_2^2X_1,$$
	which means that
	$$ \partial_1D_1 Q =  X_2^2\otimes X_1X_2 + X_2^2X_1\otimes X_2 + 1\otimes X_2X_1X_2^2 + X_1X_2\otimes X_2^2 + X_2\otimes X_2^2X_1 + X_2X_1X_2^2\otimes 1.$$
	Consequently we have that
	\begin{align*}
		R =&\ \Big(\partial^2_2\left( X_2^2 \right)(w) \boxtimes \partial^1_2 \left( X_2^2 \right)(x) \Big)  \Big(\partial^2_2\left( X_1X_2 \right)(y) \boxtimes \partial^1_2 \left( X_1X_2 \right)(z) \Big) \\
		&+ \Big(\partial^2_2\left( X_2^2X_1 \right)(w) \boxtimes \partial^1_2 \left( X_2^2X_1 \right)(x) \Big)  \Big(\partial^2_2\left( X_2 \right)(y) \boxtimes \partial^1_2 \left( X_2 \right)(z) \Big) \\
		&+ \Big(\partial^2_2\left( 1 \right)(w) \boxtimes \partial^1_2 \left( 1 \right)(x) \Big)  \Big(\partial^2_2\left(  X_2X_1X_2^2 \right)(y) \boxtimes \partial^1_2 \left(  X_2X_1X_2^2 \right)(z) \Big) \\
		&+ \Big(\partial^2_2\left( X_1X_2 \right)(w) \boxtimes \partial^1_2 \left( X_1X_2 \right)(x) \Big)  \Big(\partial^2_2\left( X_2^2 \right)(y) \boxtimes \partial^1_2 \left( X_2^2 \right)(z) \Big) \\
		&+ \Big(\partial^2_2\left( X_2 \right)(w) \boxtimes \partial^1_2 \left( X_2 \right)(x) \Big)  \Big(\partial^2_2\left( X_2^2X_1 \right)(y) \boxtimes \partial^1_2 \left( X_2^2X_1 \right)(z) \Big) \\
		&+ \Big(\partial^2_2\left( X_2X_1X_2^2 \right)(w) \boxtimes \partial^1_2 \left( X_2X_1X_2^2 \right)(x) \Big)  \Big(\partial^2_2\left( 1 \right)(y) \boxtimes \partial^1_2 \left( 1 \right)(z) \Big) .
	\end{align*}
	But since the non commutative differential of $1$ is always $0$, we get that 
	\begin{align*}
		R =&\ \Big(\partial^2_2\left( X_2^2 \right)(w) \boxtimes \partial^1_2 \left( X_2^2 \right)(x) \Big)  \Big(\partial^2_2\left( X_1X_2 \right)(y) \boxtimes \partial^1_2 \left( X_1X_2 \right)(z) \Big) \\
		&+ \Big(\partial^2_2\left( X_2^2X_1 \right)(w) \boxtimes \partial^1_2 \left( X_2^2X_1 \right)(x) \Big)  \Big(\partial^2_2\left( X_2 \right)(y) \boxtimes \partial^1_2 \left( X_2 \right)(z) \Big) \\
		&+ \Big(\partial^2_2\left( X_1X_2 \right)(w) \boxtimes \partial^1_2 \left( X_1X_2 \right)(x) \Big)  \Big(\partial^2_2\left( X_2^2 \right)(y) \boxtimes \partial^1_2 \left( X_2^2 \right)(z) \Big) \\
		&+ \Big(\partial^2_2\left( X_2 \right)(w) \boxtimes \partial^1_2 \left( X_2 \right)(x) \Big)  \Big(\partial^2_2\left( X_2^2X_1 \right)(y) \boxtimes \partial^1_2 \left( X_2^2X_1 \right)(z) \Big) .
	\end{align*}
	Now by using the fact that
	$$ \partial_2 X_2 = 1\otimes 1,\quad \partial_2 X_2^2 = 1\otimes X_2 + X_2\otimes 1,\quad \partial_2 X_1X_2 = X_1\otimes 1,\quad \partial_2 X_2^2X_1 = 1\otimes X_2X_1 + X_2\otimes X_1, $$ 
	we have that
	\begin{align*}
		R =&\ \Big( X_2(x) + X_2(w) \Big)  \Big( X_1(z) \Big) + \Big((X_2X_1)(w) + X_1(w) X_2(x) \Big) \\
		&+ \Big(X_1(x) \Big)  \Big(X_2(y) + X_2(z) \Big) + \Big((X_2X_1)(y) + X_1(y)X_2(z) \Big) \\
		=&\ x_2z_1 + w_2z_1 + w_2w_1 + w_1x_2 + x_1y_2 + x_1 z_2 + y_2y_1 + y_1z_2.
	\end{align*}

\end{exe}

\begin{exe}
	\label{3exe33}
	We set $Q = e^{\i \lambda X_1}\in\F_{1,0}$, let $w,x,y,z$ elements of a $\CC^*$-algebra $\A$. Given that we only have one variable, we set $X=X_1$ and $\partial=\partial_1$. Let us compute the following quantity:
	$$ R = \Big(\partial^2\left( \partial^1 D Q \right)(w) \boxtimes \partial^1 \left( \partial^1 D Q \right)(x) \Big)  \boxtimes \Big(\partial^2\left( \partial^2 D Q \right)(y) \boxtimes \partial^1 \left( \partial^2 D Q \right)(z) \Big) .$$
	First and foremost we need to use Definition \ref{3technicality} since $e^P$ is not a polynomial. Thus, we have that
	$$ R = \int_{[0,1]^4}\Big(\partial^2_{\epsilon}\left( \partial^1_{\beta} D_{\alpha} Q\right)(w) \boxtimes \partial^1_{\epsilon} \left( \partial^1_{\beta} D_{\alpha} Q\right)(x) \Big)  \boxtimes \Big(\partial^2_{\gamma}\left( \partial^2_{\beta} D_{\alpha} Q \right)(y) \boxtimes \partial^1_{\gamma} \left( \partial^2_{\beta} D_{\alpha} Q\right)(z) \Big) d\alpha\ d\beta\ d\epsilon\ d\gamma.$$
	Thus, our first step is to compute $\partial_{\beta} D_{\alpha} Q$, since $\partial X = 1 \otimes 1 $, we have that
	\begin{align*}
		\partial_{\alpha} e^P &= \i \lambda\ e^{\i \alpha \lambda X}\otimes e^{\i (1-\alpha) \lambda X}.
	\end{align*}
	Although in practice it should be the case after we evaluate our polynomials, we do not have that $e^{\i (1-\alpha) \lambda X} e^{\i \alpha \lambda X} = e^{\i\lambda X} $ in $\F_{1,0}$. Consequently one have that
	$$ D_{\alpha} Q = \i \lambda e^{\i (1-\alpha) \lambda X} e^{\i \alpha \lambda X}. $$
	Similarly we have that
	$$ \partial_{\beta} D_{\alpha} Q = -\lambda^2\ \left( (1-\alpha) e^{\i (1-\alpha)\beta \lambda X}\otimes e^{\i (1-\alpha)(1-\beta) \lambda X} e^{\i \alpha\lambda X} + \alpha e^{\i (1-\alpha) \lambda X} e^{\i \alpha \beta\lambda X} \otimes e^{\i \alpha(1-\beta)\lambda X} \right).$$
	And since  
	$$  \partial_{\varepsilon} e^{\i (1-\alpha)\beta \lambda X} = \i (1-\alpha)\beta \lambda\ e^{\i (1-\alpha)\beta\varepsilon \lambda X}\otimes e^{\i (1-\alpha)\beta(1-\varepsilon) \lambda X},  $$
	\begin{align*}
		\partial_{\varepsilon} \left(e^{\i (1-\alpha) \lambda X} e^{\i \alpha \beta\lambda X}\right) = \i\lambda\ \Big( &(1-\alpha) e^{\i (1-\alpha)\varepsilon \lambda X}\otimes e^{\i (1-\alpha)(1-\varepsilon) \lambda X} e^{\i \alpha\beta\lambda X} \\
		&+ \alpha\beta\ e^{\i (1-\alpha) \lambda X} e^{\i \alpha \beta\varepsilon\lambda X} \otimes e^{\i \alpha\beta(1-\varepsilon)\lambda X} \Big).
	\end{align*}
	\begin{align*}
		\partial_{\gamma} \left(e^{\i (1-\alpha)(1-\beta) \lambda X} e^{\i \alpha\lambda X} \right) = \i\lambda\ \Big( &(1-\alpha)(1-\beta) e^{\i (1-\alpha)(1-\beta) \gamma \lambda X} \otimes e^{\i (1-\alpha)(1-\beta)(1-\gamma) \lambda X} e^{\i \alpha\lambda X} \\
		&+ \alpha e^{\i (1-\alpha)(1-\beta) \lambda X} e^{\i \alpha\gamma\lambda X}\otimes e^{\i \alpha(1-\gamma)\lambda X} \Big).
	\end{align*}
	$$ \partial_{\gamma} e^{\i \alpha(1-\beta)\lambda X} =  \i\lambda \alpha(1-\beta)  e^{\i \alpha(1-\beta)\gamma\lambda X}\otimes  e^{\i \alpha(1-\beta)(1-\gamma)\lambda X}  .$$
	
	\noindent We finally get that 
	\begin{align*}
		R = \lambda^4 \int_{[0,1]^4} (1-\alpha&)^2\beta\ e^{\i (1-\alpha)\beta(1-\varepsilon) \lambda w} e^{\i (1-\alpha)\beta\varepsilon \lambda x} \Big( \alpha e^{\i \alpha(1-\gamma)\lambda y} e^{\i (1-\alpha)(1-\beta) \lambda z} e^{\i \alpha\gamma\lambda z} \\
		&+ (1-\alpha)(1-\beta) e^{\i (1-\alpha)(1-\beta)(1-\gamma) \lambda y} e^{\i \alpha\lambda y} e^{\i (1-\alpha)(1-\beta) \gamma \lambda z} \Big) \\
		+\Big( (1-\alpha&)\ e^{\i (1-\alpha)(1-\varepsilon) \lambda w} e^{\i \alpha\beta\lambda w} e^{\i (1-\alpha)\varepsilon \lambda x} + \alpha\beta\ e^{\i \alpha\beta(1-\varepsilon)\lambda w} e^{\i (1-\alpha) \lambda x} e^{\i \alpha \beta\varepsilon\lambda x} \Big) \\
		&\times \alpha^2(1-\beta) e^{\i \alpha(1-\beta)(1-\gamma)\lambda y}  e^{\i \alpha(1-\beta)\gamma\lambda z} \\
		d\alpha\ d\beta\ d\varepsilon\ & d\gamma.
	\end{align*}

\end{exe}

\subsection{Combinatorics and noncommutative derivatives}

Now that we have defined the usual noncommutative polynomial spaces, we build a very specific one which we need to define properly the coefficients of the topological expansion. 

\begin{defi}
	\label{3biz}
	Let $(c_n)_n$ be the sequence such that $ c_0 = 0$, $c_{n+1} = 6 c_n +6$. We define by induction, $J_0 = \{\emptyset\}$ and for $n\geq 0$, $j\in[1,2n]$,
	\begin{align*}
	J_{n+1}^{j,1} &= \Big\{ \{I_1+c_n,\dots,I_{j-1}+c_n,I_j+c_n ,I_{j},\dots,I_{2n},3c_n+1\}   \ \Big|\ I=\{I_1,\dots,I_{2n}\}\in J_n \Big\}, \\
	J_{n+1}^{2n+1,1} &= \Big\{ \{I_1+c_n,\dots,I_{2n}+c_n,3c_n+2,3c_n+1\}   \ \Big|\ I=\{I_1,\dots,I_{2n}\}\in J_n \Big\}, \\
	J_{n+1}^{j,2} &= \Big\{ \{I_1+2c_n,\dots,I_{j-1}+2c_n,I_j+2c_n,I_{j},\dots,I_{2n},3c_n+1\}   \ \Big|\ I=\{I_1,\dots,I_{2n}\}\in J_n \Big\}, \\
	J_{n+1}^{2n+1,2} &= \Big\{ \{I_1+2c_n,\dots,I_{2n}+2c_n,3c_n+3,3c_n+1\}   \ \Big|\ I=\{I_1,\dots,I_{2n}\}\in J_n \Big\}. \\
	\end{align*}
	\noindent We similarly define $\widetilde{J}_{n+1}^{j,1}$ and $\widetilde{J}_{n+1}^{j,2}$ by adding $3c_n+3$ to every integer in the corresponding sets ${J}_{n+1}^{j,1}$ and ${J}_{n+1}^{j,2}$. Finally we fix
	$$ J_{n+1} = \bigcup_{1\leq j\leq 2n+1} J_{n+1}^{j,1}\cup J_{n+1}^{j,2}\cup \widetilde{J}_{n+1}^{j,1}\cup \widetilde{J}_{n+1}^{j,2}.$$	
	
\end{defi}

\begin{exe}
	For example one have that
	\begin{align*}
		J_1 &= J_1^{1,1}\cup J_1^{1,2}\cup \widetilde{J}_1^{1,1}\cup \widetilde{J}_1^{1,2} \\
		&= \big\{ \{2,1\}, \{3,1\}, \{5,4\}, \{6,4\} \big\}.
	\end{align*}
	We also have that
	\begin{align*}
		J_2 &= J_2^{1,1}\cup J_2^{2,1}\cup J_2^{3,1} \cup J_2^{1,2} \cup J_2^{2,2} \cup J_2^{3,2} \cup \widetilde{J}_2^{1,1}\cup \widetilde{J}_2^{2,1}\cup \widetilde{J}_2^{3,1} \cup \widetilde{J}_2^{1,2} \cup \widetilde{J}_2^{2,2} \cup \widetilde{J}_2^{3,2}.
	\end{align*}
	It would be a bit too long to list every element in $J_2$. However here are a few subsets:
	$$ J_2^{1,1} = \big\{ \{8,2,1,19\}, \{9,3,1,19\}, \{11,5,4,19\}, \{12,6,4,19\} \big\}, $$
	$$ \widetilde{J}_2^{3,2} = \big\{ \{35,34,42,40\}, \{36,34,42,40\}, \{38,37,42,40\}, \{39,37,42,40\} \big\}. $$
\end{exe}

The previous definition is not exactly intuitive, however this construction will appear naturally in the rest of the paper. The following remark gives some insight on why and how.

\begin{rem}
	To better understand the heuristics behind the set $J_n$, one can view its construction the following way. Let us assume that we have a family of $c_n$ random variables $(X_i)_{i\in[1,c_n]}$ which are all independent copies of a random variable $X$, then for every element $I=\{I_1,\dots,I_{2n}\}\in J_n$, one can associate an ordered sequence of length $2n$ of those random variables, i.e. $(X_{I_1},\dots,X_{I_{2n}})$. Then the construction of the family $J_{n+1}$ is associated to the following process:
	\begin{itemize}
		\item In the case of $J_{n+1}^{j,1}$, given a list $\{I_1,\dots,I_{2n}\}\in J_n$ and its associated sequence of random variable $(X_{I_1},\dots,X_{I_{2n}})$, first we add at the end of the sequence a new independent copy of $X$, we chose to denote it as $X_{3c_n+1}$. Secondly we insert another independent copy in the $j$-th position which we denote $X_{I_{j}+c_n}$. Finally we replace all of the variables whose position is strictly smaller than $j$ (i.e. $X_{I_1}$ through $X_{I_{j-1}}$) by independent copies of $X$, which we in turn denote by $X_{I_{1}+c_n}$ though $X_{I_{j-1}+c_n}$. Following this process we get the sequence $X_{I_{1}+c_n}\dots X_{I_j+c_n} X_{I_j}\dots X_{I_{2n}} X_{3c_n+1}$. Thus this new sequence of random variable is indexed by the set of integer that we built out of $I$ while defining $J_{n+1}^{j,1}$ in Definition \ref{3biz}.
		\item In the case of $J_{n+1}^{2n+1,1}$, we proceed very similarly except that we want to insert an independent copy of $X$ in the $2n+1$-th position, since a list of $J_n$ only has $2n$ elements one cannot simply number it $I_{2n+1}+c_n$, this is why we number this element $3c_n+2$.
		\item $J_{n+1}^{j,2}$ and $J_{n+1}^{2n+1,2}$ are built very similarly with the difference that, except for $X_{3c_n+1}$, every random variable that we replaced or inserted are once again replaced by a third independent copy. Which is why instead of numbering them $I_i+c_n$ and $3c_n+2$ as we did in the first two bullet points, we number them $I_i+2c_n$ and $3c_n+3$.
		\item Finally as we will see later in this paper we also need to consider independent copies of all of the sequence of random variables that we created with  ${J}_{n+1}^{j,1}$ and ${J}_{n+1}^{j,2}$ , this is why we introduce $\widetilde{J}_{n+1}^{j,1}$ and $\widetilde{J}_{n+1}^{j,2}$.
	\end{itemize}
	This analogy with sequences of random variables comes from Lemma \ref{3imp2}. Indeed, if we replace $n$ by $2n$ in this lemma, then we have a given sequence of noncommutative random variables $(y_1,\dots,y_{2n})$. We are then led to introduce the sequences of noncommutative random variables $z_r^{1,j}$, $z_r^{2,j}$, $z_r^{1}$, $z_r^{2}$, $\widetilde{z}_r^{1,j}$, $\widetilde{z}_r^{2,j}$, $\widetilde{z}_r^{1}$, $\widetilde{z}_r^{2}$ and the process to build those is exactly the one that we just described to build elements of, respectively, $J_{n+1}^{j,1}$, $J_{n+1}^{j,2}$, $J_{n+1}^{2n+1,1}$, $J_{n+1}^{2n+1,2}$, $\widetilde{J}_{n+1}^{j,1}$, $\widetilde{J}_{n+1}^{j,2}$, $\widetilde{J}_{n+1}^{2n+1,1}$, $\widetilde{J}_{n+1}^{2n+1,2}$. Finally the fact that we define $J_n$ by induction comes from the fact that we have to use Lemma \ref{3imp2} repeatedly. See Lemma \ref{3apparition} and Proposition \ref{3intercoef}.
	
\end{rem}

\begin{defi}
	\label{3biz2}
	We define $\PP_{d,q}^n = \C\langle X_{i,I},\ 1\leq i\leq d, I\in J_n;\ Y_1,\dots,Y_{2r} \rangle$. We also define $\F_{d,q}^n$ as the $*$-algebra generated by $\PP_{d,q}^n$ and the family $\left\{e^{\i Q}\ |\ Q\in \mathcal{A}_{d,q}^n \text{ self-adjoint}\right\}$. Besides similarly to Definition \ref{3technicality}, we define $\partial_{i}$ and $\partial_{i,I}$ on $\F_{d,q}^n$ which satisfies \eqref{3leibniz} and \eqref{3ext} and
	$$ \forall i,j\in [1,p],\ I,K\in J_n,\quad \partial_{i,I} X_{j,K} = \delta_{i,j}\delta_{I,K} 1\otimes 1,\quad \partial_{i} X_{j,K} = \delta_{i,j} 1\otimes 1 .$$
	We then define $D_{i} = m \circ \partial_{i}$ and $D_{i,I} = m \circ \partial_{i,I}$ on $\F_{d,q}^n$.
\end{defi}

In particular, $\F_{d,q}^0 = \F_{d,q}$ and the two definitions of $\partial_i$ coincide. The following lemma will be important for a better estimation of the remainder term in the expansion.

\begin{lemma}
	\label{3detail}
	Given $s\in [1,c_n]$, there exists a unique $l\in [1,n]$ such that for any $ I =\{I_1,\dots,I_{2n}\}\in J_n$, either $I_l = s$ or $s\notin I$. We refer to $l$ as the depth of $s$ in $J_n$, and will denote it $\dep^n(s)$.
\end{lemma}

\begin{proof}
	Let us proceed by induction. If this lemma is true for a given $n$, then let $s\in [1,c_{n+1}]$. By definition if $s\leq 3c_n+3$ then it will only appears in the elements of ${J}_{n+1}^{j,1}$ and ${J}_{n+1}^{j,2}$ for $j$ from $1$ to $n+1$. On the contrary if $s> 3c_n+3$ then it will only appears in the elements of $\widetilde{J}_{n+1}^{j,1}$ and $\widetilde{J}_{n+1}^{j,2}$ for $j$ from $1$ to $n+1$. Let us first consider the case where $s\leq 3c_n+3$, to begin with:
	\begin{itemize}
		\item If $s=3c_n+1$, then by definition of ${J}_{n+1}^{j,1}$ and ${J}_{n+1}^{j,2}$, coupled with the fact that if $I\in J_n$ then for any $l$, $I_l\leq c_n$, we have that $\dep^{n+1}(s)=2n+2$.
		\item If $s=3c_n+2\text{ or }3c_n+3$, then similarly we have that $\dep^{n+1}(s)=2n+1$. 
	\end{itemize}
	
	\noindent Thus, there remains three possibilities:
	\begin{itemize}
		\item If $s\in[1,c_n]$, then let $l$ be the depth of $s$ in $J_n$, by construction the depth of $s$ in $J_{n+1}$ will be $l+1$.
		\item If $s\in[c_n+1,2c_n]$, then with $l$ the depth of $s-c_n$ in $J_n$, the depth of $s$ in $J_{n+1}$ will also be $l$.
		\item If $s\in[2c_n+1,3c_n]$, then with $l-2c_n$ the depth of $s$ in $J_n$, the depth of $s$ in $J_{n+1}$ will also be $l$.
	\end{itemize}
Since to define $\widetilde{J}_{n+1}^{j,1}$ and $\widetilde{J}_{n+1}^{j,2}$ we simply add $3c_n+3$ to every integer contained in the elements of ${J}_{n+1}^{j,1}$ and ${J}_{n+1}^{j,2}$, if $s>3c_n+3$, then with $l=\dep^{n+1}(s-3c_n+3)$, one has that $\dep^{n+1}(s)= l$.
\end{proof}



\subsection{GUE random matrices}

We conclude this section by recalling the definition of Gaussian random matrices and stating a few useful properties about them. 

\begin{defi}
	\label{3GUEdef}
	A GUE random matrix $X^N$ of size $N$ is a self-adjoint matrix whose coefficients are random variables with the following laws:
	\begin{itemize}
		\item For $1\leq i\leq N$, the random variables $\sqrt{N} X^N_{i,i}$ are independent centered Gaussian random variables of 
		variance $1$.
		\item For $1\leq i<j\leq N$, the random variables $\sqrt{2N}\ \Re{X^N_{i,j}}$ and $\sqrt{2N}\ \Im{X^N_{i,j}}$ are independent 
		centered Gaussian random variables of variance $1$, independent of  $\left(X^N_{i,i}\right)_i$.
	\end{itemize}
\end{defi}

When doing computations with Gaussian variables, the main tool that we use is Gaussian integration by parts. It can be summarized into the following formula, if $Z$ is a centered Gaussian variable with variance $1$ and $f$ a $\mathcal{C}^1$ function, then
\begin{equation}
\label{3IPPG}
\E[Z f(Z)] = \E[\partial_Z f(Z)] \ .
\end{equation}

\noindent A direct consequence of this, is that if $x$ and $y$ are independent centered Gaussian random variables with variance $1$, and 
$Z = \frac{x+\i y}{\sqrt{2}}$, then
\begin{equation}
\label{3IPPG2}
\E[Z f(x,y)] = \E[\partial_Z f(x,y)]\quad \text{ and  }\quad \E[\overline{Z} f(x,y)] = \E[\partial_{\overline{Z}} f(x,y)] \ ,
\end{equation}

\noindent where $\partial_Z = \frac{1}{2} (\partial_x + \i \partial_y)$ and $\partial_{\overline{Z}} = \frac{1}{2} (\partial_x - \i \partial_y)$. When working with GUE matrices, an important consequence of this are the so-called Schwinger-Dyson equations, which we summarize in the following proposition. For more information about these equations and their applications, we refer to \cite[Lemma 5.4.7]{alice}.

\begin{prop}
	\label{3SD}
	Let $X^N$ be GUE matrices of size $N$, $Q\in \F_{d,q}$, then for any $i$,
	
	$$ \E\left[ \frac{1}{N}\tr_N(X^N_i\ Q(X^N)) \right] = \E\left[ \left(\frac{1}{N}\tr_N\right)^{\otimes 2} (\partial_i Q(X^N)) \right] . $$
	
\end{prop}

\begin{proof}
	Let us first assume that $Q\in\PP_{d,q}$. One can write $X^N_i = \frac{1}{\sqrt{N}} (x_{r,s}^i)_{1\leq r,s\leq N}$ and thus
	
	\begin{align*}
	\E\left[ \frac{1}{N}\tr_N(X^N_i\ Q(X^N)) \right] &= \frac{1}{N^{3/2}} \sum_{r,s} \E\left[ x_{r,s}^i\ \tr_N(E_{r,s}\ Q(X^N)) \right] \nonumber \\
	&= \frac{1}{N^{3/2}} \sum_{r,s} \E\left[ \tr_N(E_{r,s}\ \partial_{x_{r,s}^i} Q(X^N)) \right] \\
	&= \frac{1}{N^{2}} \sum_{r,s} \E\left[ \tr_N(E_{r,s}\ \partial_i Q(X^N) \# E_{s,r}) \right] \nonumber \\
	&= \E\left[ \left(\frac{1}{N}\tr_N\right)^{\otimes 2} (\partial_i Q(X^N)) \right] . \nonumber
	\end{align*}
	
	\noindent If $Q\in\F_{d,q}$, then the proof is pretty much the same but we need to use Duhamel's formula (for a very similar proof see \cite[Proposition 2.2]{deux}) which states that for any matrices $A$ and $B$, 
	\begin{equation}
	\label{3duha}
	e^B - e^A = \int_{0}^1 e^{\alpha B} (B-A) e^{(1-\alpha)A}\ d\alpha .
	\end{equation}
	Thus, this let us prove that for any self-adjoint polynomials $P\in\PP_{d,q}$,
	$$ \partial_{x_{r,s}^i} e^{\i P(X^N)} = \i \int_{0}^1 e^{\i \alpha P(X^N)}\ \partial_i P(X^N) \# E_{s,r}\ e^{\i (1-\alpha) P(X^N)}\ d\alpha .$$
	And the conclusion follows.
	
\end{proof}

Now to finish this section we state a property that we use several times in this paper. For the proof we refer to \cite[Proposition 2.11]{un}.

\begin{prop}
	\label{3bornenorme}
	There exist constants $C,D$ and $\alpha$ such that for any $N\in\N$, if $X^N$ is a GUE random matrix of size $N$, then for any 
	$u\geq 0$,
	
	$$ \P\left(\norm{X^N}\geq u+D \right) \leq e^{-\alpha u N} . $$
	
	\noindent Consequently, for any $k\leq \alpha N /2 $,
	
	$$ \E\left[\norm{X^N}^k\right] \leq C^k .$$
	
\end{prop}

\section{Proof of Theorem \ref{3lessopti}}
\label{3mainsec}

\subsection{A Poincar\'e type equality}

One of the main tools when dealing with GUE random matrices is the Poincar\'e inequality (see \cite[Definition 4.4.2]{alice}), which gives us a sharp upper bound of the variance of a function in these matrices. Typically this inequality shows that the variance of a trace of a polynomial in GUE random matrices, which a priori is of order $\mathcal{O}(1)$, is of order $\mathcal{O}(N^{-2})$. In this paper we use the same kind of argument which are used to prove the Poincar\'e inequality to get an exact formula for the variances we are interested in.

\begin{prop}
	\label{3concentration}
	
	Let $P,Q\in\F_{d,q}$, $R^N$, $S^N$, $T^N$ be independent families of $d$ independent GUE matrices of size $N$. Let $A^N$ be a family of deterministic matrices and their adjoints. With convention $\cov(X,Y) = \E[XY]-\E[X]\E[Y]$, for any $t\geq 0$, we have:
	
	\begin{align*}
	&\cov\Big(\tr_N\left(P\left( (1-e^{-t})^{1/2}R^N ,A^N \right)\right), \tr_N\left(Q\left( (1-e^{-t})^{1/2}R^N ,A^N \right)\right) \Big) \\
	&= \frac{1}{N} \sum_{i} \int_{0}^{t} e^{-s}\ \E\Big[ \tr_N\Big( D_iP\left((e^{-s}-e^{-t})^{1/2}R^N + (1-e^{-s})^{1/2} S^N, A^N\right) \\
	&\quad\quad\quad\quad\quad\quad\quad\quad\quad\quad\quad \times D_iQ\left((e^{-s}-e^{-t})^{1/2}R^N + (1-e^{-s})^{1/2} T^N, A^N\right) \Big) \Big] ds .
	\end{align*}
	
\end{prop}

\begin{proof}
	
	We define the following function,
	\begin{align*}
	h(s) = \E\Big[ &\tr_{N}\left( P\left( (e^{-s}-e^{-t})^{1/2}R^N + (1-e^{-s})^{1/2}S^N ,A^N \right) \right) \\ 
	&\tr_{N}\left( Q\left( (e^{-s}-e^{-t})^{1/2}R^N + (1-e^{-s})^{1/2}T^N ,A^N \right) \right) \Big] .
	\end{align*}
	
	\noindent To simplify notations, we set 
	$$ S^{N,s} = \left( (e^{-s}-e^{-t})^{1/2}R^N + (1-e^{-s})^{1/2}S^N ,A^N \right),$$
	$$ T^{N,s} = \left( (e^{-s}-e^{-t})^{1/2}R^N + (1-e^{-s})^{1/2}T^N ,A^N \right). $$
	
	\noindent Then we have,
	$$ \cov\Big(\tr_N\left(P\left( (1-e^{-t})^{1/2}R^N ,A^N \right)\right), \tr_N\left(Q\left( (1-e^{-t})^{1/2}R^N ,A^N \right)\right) \Big) = -\int_{0}^t \frac{dh}{ds}(s) \ ds .$$
	
	\noindent Thanks to Duhamel's formula (see \eqref{3duha}) we find 
	$$ \frac{d P\left( S^{N,s} \right) }{ds} = -\frac{e^{-s}}{2} \sum_{i=1}^d \partial_i P\left( S^{N,s} \right) \# \left( \frac{R^N_i}{(e^{-s}-e^{-t})^{1/2}} - \frac{S^N_i}{(1-e^{-s})^{1/2}} \right) .$$
	
	\noindent Since $\tr_N(\partial_i P \# B) = \tr_N(D_iP\times B)$, we compute,
	\begin{align*}
	\frac{dh}{ds}(s) = -\frac{e^{-s}}{2} \sum_i \E\Bigg[& \tr_{N}\left( D_i P\left( S^{N,s} \right) \left( \frac{R^N_i}{(e^{-s}-e^{-t})^{1/2}} - \frac{S^N_i}{(1-e^{-s})^{1/2}} \right) \right) \tr_{N}\left( Q\left( T^{N,s} \right) \right) \\
	&+ \tr_{N}\left( P\left( S^{N,s} \right) \right) \tr_{N}\left( D_i Q\left( T^{N,s} \right) \left( \frac{R^N_i}{(e^{-s}-e^{-t})^{1/2}} - \frac{T^N_i}{(1-e^{-s})^{1/2}} \right) \right)  \Bigg] .
	\end{align*}
	
	\noindent But by using integration by parts formula \eqref{3IPPG2}, we get that 
	\begin{align*}
	&\E\Bigg[ \tr_{N}\left( D_i P\left( S^{N,s} \right) \frac{R^N_i}{(e^{-s}-e^{-t})^{1/2}} \right) \tr_{N}\left( Q\left( T^{N,s} \right) \right) \Bigg] \\
	&= \frac{1}{N} \sum_{ 1\leq a,b \leq N} \E\Bigg[ \tr_{N}\left( E_{a,b}\ \partial_i D_i P\left( S^{N,s} \right) \# E_{b,a} \right)\times  \tr_{N}\left( Q\left( T^{N,s} \right) \right) \\
	& \quad\quad\quad\quad\quad\quad + \tr_{N}\left( D_i P\left( S^{N,s} \right) E_{a,b} \right)\times \tr_{N}\left( D_i Q\left( T^{N,s} \right) E_{b,a}\right) \Bigg] .
	\end{align*}
	
	\noindent And similarly
	\begin{align*}
	&\E\Bigg[ \tr_{N}\left( D_i P\left( S^{N,s} \right) \frac{S^N_i}{(1-e^{-s})^{1/2}} \right) \tr_{N}\left( Q\left( T^{N,s} \right) \right) \Bigg] \\
	&= \frac{1}{N} \sum_{ 1\leq a,b \leq N} \E\Bigg[ \tr_{N}\left( E_{a,b}\ \partial_i D_i P\left( S^{N,s} \right) \# E_{b,a} \right)\times  \tr_{N}\left( Q\left( T^{N,s} \right) \right) \Bigg].
	\end{align*}
	
	\noindent Hence
	\begin{align*}
		&\E\Bigg[ \tr_{N}\left( D_i P\left( S^{N,s} \right) \left( \frac{R^N_i}{(e^{-s}-e^{-t})^{1/2}} - \frac{S^N_i}{(1-e^{-s})^{1/2}} \right) \right) \tr_{N}\left( Q\left( T^{N,s} \right) \right) \Bigg] \\
		&= \frac{1}{N^2} \E\Bigg[ \tr_{N}\Big( D_i P\left( S^{N,s} \right) D_i Q\left( T^{N,s} \right) \Big) \Bigg] .
	\end{align*}

	\noindent Therefore with similar computations we conclude,
	\begin{align*}
	\frac{dh}{ds}(s) = -\frac{1}{N} e^{-s} \sum_i \E\Bigg[ &\tr_{N}\Big( D_i P\left( S^{N,s} \right) D_i Q\left( T^{N,s} \right) \Big)  \Bigg] .
	\end{align*}
	
	\noindent Hence the conclusion.
	
\end{proof}

\subsection{A first rough formulation of the coefficients}

\label{3mauvaisepres}

In this subsection we prove the following lemma which will be the backbone of the proof of the topological expansion. The heuristics behind this lemma is that if $Q\in \F_{d(c_n+1),q}$, $X^N$ and $Z^N$ are matrices as in Lemma \ref{3imp2}, $x$ and $(y_i)_{i\geq 1}$ are systems of $d$ free semicircular variables free between each other, then we can find $R\in \F_{d(c_{n+1}+1),q}$ such that  
\begin{align*}
&\E\left[\tau_N\Big(Q\left(X^N,(y_i)_{1\leq i\leq c_n},Z^N\right)\Big)\right] - \tau_N\Big(Q\left(x,(y_i)_{1\leq i\leq c_n},Z^N\right)\Big) \\
&= \frac{1}{N^2} \E\left[\tau_N\Big(R\left(X^N,(y_i)_{1\leq i\leq c_{n+1}},Z^N\right)\Big)\right].
\end{align*}

\noindent Then we will only need to apply this lemma recursively to build the topological expansion. Note that thanks to the definition of $\A_N$ in Definition \ref{3tra}, it makes sense to consider matrices and free semicircular variables in the same space. One can also assume that those matrices are random thanks to \cite[Proposition 2.7]{un}. Finally to better understand Equation \eqref{3nececpour}, you can check Examples \ref{3exe11}, \ref{3exe22} and \ref{3exe33}.

\begin{lemma}
	\label{3imp2}
	Let the following objects be given,
	\begin{itemize}
		\item $X^N = (X_1^N,\dots,X_d^N)$ independent $GUE$ matrices of size $N$,
		\item $x,z^1,z^2$ free families of $d$ free semi-circular variables,
		\item $y_s = (y_{s,1},\dots,y_{s,d_s})$ for $s$ from $1$ to $n$, systems of free semicircular variables, free between each other and from $x$,
		\item $v_s,w_s$ free copies of $y_s$, free between each other,
		\item $Z^{N} = (Z_1^{N},\dots,Z_q^{N})$ deterministic matrices and their adjoints,
		\item $ Y^N = \left((1-e^{-t_1})^{1/2}y_1,\dots,(1-e^{-t_n})^{1/2}y_n,X^N,Z^{N}\right)$ where $t_i\in\R^+$ is fixed,
		\item $ Y = \left((1-e^{-t_1})^{1/2}y_1,\dots,(1-e^{-t_n})^{1/2}y_n,x,Z^{N}\right)$,
		\item $ z_{r}^{1} = \Big( (1-e^{-t_1})^{1/2}v_1,\dots,(1-e^{-t_n})^{1/2}v_n, (1-e^{-r})^{1/2} z^1 + (e^{-r}-e^{-t})^{1/2} x + e^{-t/2} X^N ,Z^{N}\Big) ,$
		\item for $s$ from $1$ to $n$,
		\begin{align*}
		z_{r}^{1,s} = \Big(& (1-e^{-t_1})^{1/2}v_1,\dots, (1-e^{-r})^{1/2} v_s + (e^{-r}-e^{-t_s})^{1/2}y_s, \\ &(1-e^{-t_{s+1}})^{1/2}y_{s+1}, \dots, (1-e^{-t_n})^{1/2}y_n, (1-e^{-t})^{1/2} x + e^{-t/2} X^N ,Z^{N}\Big) ,
		\end{align*}
		\item $z_{r}^{2}$ and $z_{r}^{2,s}$, defined similarly but with $w$ and $z^2$ instead of $v$ and $z^1$,
		\item $\widetilde{z}_{r}^{1},\widetilde{z}_{r}^{2},\widetilde{z}_{r}^{1,s}$ and $\widetilde{z}_{r}^{2,s}$ defined similarly but where we replaced $v_s,w_s,y_s,z^1,z^2,x$ by free copies,
		\item $Q\in \F_{d_1+\dots+d_n+d,q}$.
	\end{itemize}
	
	\noindent Then, for any $N$, with $\partial_{s,j}$ defined similarly to the noncommutative differential introduced in Definition \ref{3application} but with respect to $(1-e^{-t_s})^{1/2}y_{s,j}$ instead of $X_i$.
	\begin{align}
	\label{3nececpour}
	&\E\left[\tau_N\Big(Q\left(Y^N\right)\Big)\right] - \tau_N\Big(Q\left(Y\right)\Big) \nonumber\\
	&=  \frac{1}{2N^2} \int_0^{\infty} e^{-t} \sum_{\substack{1\leq i\leq d\\1\leq s \leq n,\\ 1\leq j\leq d_s} }\ \int_0^{t_s} e^{-r}\ \E\Big[ \tau_{N}\Big( \Big(\partial_{s,j}^2\left( \partial_i^1 D_i Q\right)(z_{r}^{1,s}) \boxtimes \partial_{s,j}^1 \left( \partial_i^1 D_i Q\right)(\widetilde{z}_{r}^{1,s}) \Big) \nonumber\\
	&\quad\quad\quad\quad\quad\quad\quad\quad\quad\quad\quad\quad\quad\quad\quad\quad\quad \boxtimes \Big(\partial_{s,j}^2\left( \partial_i^2 D_i Q\right)(\widetilde{z}^{2,s}_{r}) \boxtimes \partial_{s,j}^1 \left( \partial_i^2 D_i Q\right)(z^{2,s}_{r}) \Big) \Big) \Big] dr\ dt \\
	&\quad +  \frac{1}{2N^2} \int_0^{\infty} e^{-t} \sum_{1\leq i,j\leq d }\ \int_0^{t} e^{-r}\ \E\Big[ \tau_{N}\Big( \Big(\partial_{j}^2\left( \partial_i^1 D_i Q\right)(z_{r}^{1}) \boxtimes \partial_{j}^1 \left( \partial_i^1 D_i Q\right)(\widetilde{z}_{r}^{1}) \Big) \nonumber\\
	&\quad\quad\quad\quad\quad\quad\quad\quad\quad\quad\quad\quad\quad\quad\quad\quad\quad\quad \boxtimes \Big(\partial_{j}^2\left( \partial_i^2 D_i Q\right)(\widetilde{z}^{2}_{r}) \boxtimes \partial_{j}^1 \left( \partial_i^2 D_i Q\right)(z^{2}_{r}) \Big) \Big) \Big] dr\ dt. \nonumber
	\end{align}
	
\end{lemma}

Note that we used the definition of $\partial_i$ from Definition \ref{3technicality}. Thus, when we need to compose those operators, we do as in Definition \ref{3operatordef}. For a more concrete example don't hesitate to also check Example \ref{3exe33}. Now we need to prove the following technical lemma.

\begin{lemma}
	\label{3coeffnndiag}
	If $Y^{kN}$ is a family of $l$ independent GUE matrices of size $kN$, $K^N$ a family of $q$ deterministic matrices, then let
	$$ 	S_k = \left( Y^{kN}, A^N\otimes I_k \right) .$$
	With $P_{1,2} = I_N\otimes E_{1,2}$, $\E_k$ the expectation with respect to $ Y^{kN} $, given $Q\in \mathcal{F}_{l,q}$, we have that 
	$$ \lim_{k\to \infty} k^{3/2} \E_k\left[ \ts_{kN}\left( Q(S_k) P_{1,2} \right)\right] =0$$
\end{lemma}

\begin{proof}
	Given $A_1,\dots,A_r, B_1,\dots, B_r\in \PP_{l,q}$, assuming that the $B_i$'s are self-adjoints, we define the following quantities,
	$$ f_{A}(\alpha) = \E_k\left[ \ts_{kN}\left( (A_1 e^{\i \alpha B_1 } \dots A_r e^{\i \alpha B_r }) (S_k) P_{1,2} \right)\right] ,$$
	$$ d_n(\alpha) = \max_{\sum_i \deg A_i \leq n,\ A_i\text{ monomials}} \left| f_{A}(\alpha) \right| .$$
	
	\noindent Thanks to Proposition \ref{3bornenorme}, we know that there exists constants $\gamma$ and $D$ (depending on $N$,$\norm{X^N}$ and $\norm{K^N}$) such that for any $n\leq \gamma k$  and $\alpha\in[0,1]$, $|d_n(\alpha)| \leq D^n$. Consequently we define 
	$$ g(a,\alpha) = \sum_{n\leq \gamma k /2} d_n(\alpha) a^n .$$
	
	\noindent Let $m = \sup_i \deg B_i$ and $A$ be such that $\sum_i \deg A_i \leq n$, there exists a constant $C_B$ which only depends on the coefficients of the $B_i$'s such that
	$$ \left| \frac{d f_{A}(\alpha)}{d\alpha} \right| \leq C_B\ d_{n+m}(\alpha).$$
	
	\noindent By integrating this inequality, we get that for any $\alpha\in [0,1]$
	$$ \left| f_{A}(\alpha) \right| \leq \left| f_{A}(0) \right| + C_B  \int_{0}^{\alpha} d_{n+m}(\beta) d\beta.$$
	
	\noindent And by taking the supremum over $A$, we get that 
	$$ d_n(\alpha) \leq d_n(0) + C_B  \int_{0}^{\alpha} d_{n+m}(\beta) d\beta.$$
	
	\noindent Hence by summing over $n$, we have for $a<1/D$,
	\begin{align*}
		g(a,\alpha) &= \sum_{n\leq \gamma k /2} d_n(\alpha) a^n \\
		&\leq \sum_{n\leq \gamma k /2} d_n(0)a^n + C_B  \int_{0}^{\alpha} \sum_{n\leq \gamma k /2}	 d_{n+m}(\beta) a^n d\beta \\
		&\leq g(a,0) + C_B  \int_{0}^{\alpha} \sum_{m\leq n\leq \gamma k /2} d_{n}(\beta) a^{n-m} d\beta + C_B \sum_{\gamma k /2 < n\leq m+ \gamma k /2} D^n a^{n-m} \\
		&\leq g(a,0) + m C_B a^{-m} (aD)^{\gamma k / 2} +  C_B a^{-m}\int_{0}^{\alpha} g(a,\beta) d\beta .
	\end{align*}
	
	\noindent Thanks to Gr\"onwall's inequality (see \cite[Lemma 8.4]{legallfr}), we get that for any $\alpha\in [0,1]$,
	$$ g(a,\alpha) \leq \left( g(a,0) +  m C_B a^{-m} (aD)^{\gamma k / 2} \right) e^{\alpha C_B a^{-m}} .$$
	
	\noindent Thus, for $a<1/D$, we have 
	$$ \limsup_{k\to \infty} k^{3/2}g(a,\alpha) \leq e^{\alpha C_B a^{-m}} \limsup_{k\to \infty} k^{3/2} g(a,0) .$$
	
	\noindent However, we have
	$$ g(a,0) = \sum_{n\leq \gamma k /2} a^n \max_{A \text{ monomial, } \deg A \leq n} \left| \E_k\left[ \ts_{kN}\left( A(S_k) P_{1,2} \right)\right] \right| .$$
	
	\noindent We refer to the proof of \cite[Lemma 3.7]{un} to prove that $\limsup_{k\to \infty} k^{3/2} g(a,0) = 0$ (with the notations of \cite{un}, it is the same thing as to show that $k^{3/2} f_{\gamma k /2}(a)$ converges towards $0$). Hence for any $A,B$, 	
	\begin{align*}
		\limsup_{k\to \infty} k^{3/2} \left| \E_k\left[ \ts_{kN}\left( (A_1 e^{\i B_1 } \dots A_r e^{\i B_r }) (S_k) P_{1,2} \right)\right] \right| &\leq a^{-\sum_i \deg A_i} \limsup_{k\to \infty} k^{3/2}g(a,1)  \\
		&\leq a^{-\sum_i \deg A_i} e^{C_B a^{-m}} \limsup_{k\to \infty} k^{3/2}g(a,0) \\
		&=0
	\end{align*}
	
	\noindent Hence the conclusion.

\end{proof}

\begin{proof}[Proof of Lemma \ref{3imp2}]
	The proof being one of the longest of the paper, we divide it in four different steps. In the first one we interpolate between the GUE random matrices and the free semicircular variables and exhibit the term $\Lambda_{N,t}$ that we study in the next following steps. In the second one, with the help of Lemma \ref{3coeffnndiag}, we reformulate the term $\Lambda_{N,t}$ in order to express it as a covariance. In the third step, we use Proposition \ref{3concentration} to express the covariance as an integral. And finally we finish the computations in the fourth and last step. \\
	
	\textbf{Step 1:} With 
	$$ Y_{t}^N = \left((1-e^{-t_1})^{1/2}y_1, \dots,(1-e^{-t_n})^{1/2} y_n,(1-e^{-t} )^{1/2} x + e^{-t/2}X^N,Z^{N}\right),$$
	we have, 
	$$ \E\left[\tau_N\Big(Q\left(Y^N\right)\Big)\right] - \tau_N\Big(Q\left(Y\right)\Big) = -\int_{0}^{\infty} \E\left[ \frac{d}{dt} \tau_N\Big( Q\left(Y_{t}^N\right)\Big) \right] dt . $$
	
	\noindent We can compute 
	$$ \frac{d}{dt}\tau_N\Big(Q\left(Y_{t}^N\right)\Big) = \frac{e^{-t}}{2} \sum_i \tau_N\left(D_i Q\left(Y_{t}^N\right) \left( \frac{x_i}{(1-e^{-t})^{1/2}} - e^{t/2} X_i^N\right) \right) $$
	
	\noindent Thus, thanks to Gaussian integration by parts (see \eqref{3IPPG2}) and Schwinger-Dyson equations (see Proposition \ref{3SDE}), we get that
	\begin{align}
	\label{3vantlmbd}
	&\E\left[\frac{d}{dt}\tau_N\Big(Q\left(Y_{t}^N\right)\Big)\right] \nonumber \\
	&= \E\Bigg[ \frac{e^{-t}}{2} \sum_i \left( \tau_N\otimes\tau_N \Big(\partial_i D_i Q\left(Y_{t}^N\right)\Big)  - \frac{1}{N}\sum_{u,v} \tau_N\Big( E_{u,v}\ \partial_i D_i Q\left(Y_{t}^N \right)\# E_{v,u} \Big)\right) \Bigg].
	\end{align}
	
	\noindent Let
	\begin{align*}
	\Lambda_{N,t} = \tau_N\otimes\tau_N \Big(\partial_i D_i Q\left(Y_{t}^N\right)\Big)  - \frac{1}{N}\sum_{u,v} \tau_N\Big( E_{u,v}\ \partial_i D_i Q\left(Y_{t}^N\right)\# E_{v,u} \Big).
	\end{align*}
	
	\noindent Thanks to \cite[Theorem 5.4.5]{alice}, we have that if 
	$$ Z_k = \left((1-e^{-t_1})^{1/2}Y^{kN}_1,\dots,(1-e^{-t_n})^{1/2}Y^{kN}_n,(1-e^{-t})^{1/2}Y^{kN}_{n+1} + e^{-t/2} X^N\otimes I_k,Z^{N}\otimes I_k\right), $$
	with $Y^{kN}_s$ being independent families of $d_s$ independent GUE matrices (with $d_{n+1}=d$), independent from $X^N$, then with $\E_k$ the expectation with respect to $Y^{kN}_s$ for every $s$,
	\begin{align}
	\label{3limitnec1}
	\Lambda_{N,t} = \lim\limits_{k\to\infty} &\quad \E_k[\ts_{kN}]\otimes\E_k[\ts_{kN}] \Big(\partial_i D_i Q\left(Z_k\right)\Big) \\
	&- \E_k\Bigg[ \frac{1}{N} \sum_{1\leq u,v\leq N} \ts_{kN}\Big( E_{u,v}\otimes I_{k}\ \partial_i D_i Q\left(Z_k\right)\# E_{v,u}\otimes I_{k} \Big)\Bigg] . \nonumber
	\end{align}
	
	\noindent For more information, we refer to \cite[Proposition 3.5]{un}. See also the definition of $\A_N$ in Definition \ref{3tra}. \\
	
	\textbf{Step 2:} Let $A,B$ be matrices of $\M_N(\C) \otimes \M_k(\C)$, with $(g_a)_{1\leq a\leq N}$ the canonical basis of $\C^N$ and $(f_b)_{1\leq b\leq k}$ the one of $\C^k$, since $I_k = \sum_l E_{l,l}$ and $\forall M\in \M_N(\C) \otimes \M_k(\C)$, $\ts_{kN}(M) = \sum_{a,b} g_a^*\otimes f_b^* M g_a\otimes f_b$, 
	\begin{align*}
	& \frac{1}{N} \sum_{1\leq u,v\leq N} \ts_{kN}\Big( E_{u,v}\otimes I_{k}\ A\ E_{v,u}\otimes I_{k}\ B \Big) \\
	&= \frac{1}{N} \sum_{1\leq u,v\leq N} \sum_{1\leq l,l'\leq k} \ts_{kN}\left( E_{u,v}\otimes E_{l,l}\ A\ E_{v,u}\otimes E_{l',l'}\ B\right)  \nonumber \\
	&= \frac{1}{N^2k} \sum_{1\leq l,l'\leq k} \sum_{1\leq v\leq N} g_v^*\otimes f_l^*\ A\ g_v\otimes f_{l'} \sum_{1\leq u\leq N} g_u^*\otimes f_{l'}^*\ B\ g_u\otimes f_l \nonumber\\
	&= \frac{1}{k} \sum_{1\leq l,l'\leq k} \ts_N( I_N\otimes f_l^*\ A\ I_N\otimes f_{l'})\ \ts_N(I_N\otimes f_{l'}^*\ B\ I_N\otimes f_l) \nonumber\\
	&= k \sum_{1\leq l,l'\leq k} \ts_{kN}\big( A\ I_N\otimes E_{l',l}\big)\ \ts_{kN}\big( B\ I_N\otimes E_{l,l'}\big). \nonumber
	\end{align*}
	
	\noindent Hence with convention $ P_{l,l'} = I_N\otimes E_{l,l'}$, we have
	\begin{equation}
	\label{3limitnec10}
	\frac{1}{N} \sum_{1\leq u,v\leq N} \ts_{kN}\Big( E_{u,v}\otimes I_{k}\ \partial_i D_i Q\left(Z_k\right)\# E_{v,u}\otimes I_{k} \Big) = k \sum_{1\leq l,l'\leq k} \ts_{kN}\otimes \ts_{kN}\Big( \partial_i D_i Q\left(Z_k\right) \times P_{l',l}\otimes P_{l,l'} \Big) 
	\end{equation}

\noindent Consequently, we get that

\begin{align}
	\label{3limitnec11}
	\Lambda_{N,t} = \lim\limits_{k\to\infty} \quad &\E_k[\ts_{kN}]\otimes\E_k[\ts_{kN}] \Big(\partial_i D_i Q\left(Z_k\right)\Big) \\
	& - \E_k\Bigg[k \sum_{1\leq l,l'\leq k} \ts_{kN}\otimes \ts_{kN}\Big( \partial_i D_i Q\left(Z_k\right) \times P_{l',l}\otimes P_{l,l'} \Big) \Bigg] . \nonumber
	\end{align}
	
	\noindent Let $U\in\M_k(\C)$ be a unitary matrix, then since for any $i$, 
	$$ I_N\otimes U\ X_i^N\otimes I_{k}\ I_N\otimes U^* = X_i^N\otimes I_{k} ,$$
	$$ I_N\otimes U\ Z_i^{N}\otimes I_k\ I_N\otimes U^* = Z_i^{N}\otimes I_k ,$$
	and that the law of $Y^{kN}_{s,j}$ is invariant by conjugation by a unitary matrix, we get that for any unitary matrices $U$ and $V$,
	\begin{align*}
		&\E_k\left[  \ts_{kN} \right] \otimes \E_k\left[ \ts_{kN} \right]\Big( \partial_i D_i Q\left(Z_k\right) \times P_{l',l}\otimes P_{l,l'} \Big)  \\
		&= \E_k\left[  \ts_{kN} \right] \otimes \E_k\left[ \ts_{kN} \right] \Big( \partial_i D_i Q\left(Z_k\right) \times (I_N\otimes U^*P_{l',l} I_N \otimes U) \otimes (I_N\otimes V^*P_{l,l'} I_N \otimes V) \Big) .
	\end{align*}
			
	\noindent Thus, if $l=l'$, we can pick $U$ such that $U^*E_{l',l} U = E_{1,1}$, and if $l\neq l'$, we can pick $U$ such that $U^*E_{l',l} U = E_{1,2}$. By doing the same for $V$, we have
	\begin{align}
	\label{3limitnec2}
	&k \sum_{1\leq l,l'\leq k} \E_k[\ts_{kN}]\otimes\E_k[\ts_{kN}] \Big( \partial_i D_i Q(Z_k) P_{l',l}\otimes P_{l,l'} \Big) \nonumber \\
	=&\ k^2 \E_k[\ts_{kN}]\otimes\E_k[\ts_{kN}] \Big( \partial_i D_i Q(Z_k) P_{1,1}\otimes P_{1,1} \Big) \\
	&+ k^2(k-1) \E_k[\ts_{kN}]\otimes\E_k[\ts_{kN}] \Big( \partial_i D_i Q(Z_k) P_{1,2}\otimes P_{1,2} \Big) .\nonumber
	\end{align}
	
	\noindent Similarly we also have,
	\begin{align}
	\label{3limitnec3}
	&\E_k[\ts_{kN}]\otimes\E_k[\ts_{kN}] \Big( \partial_i D_i Q(Z_k) \Big) \nonumber \\
	=&\ \sum_{1\leq l,l'\leq k} \E_k[\ts_{kN}]\otimes\E_k[\ts_{kN}] \Big( \partial_i D_i Q(Z_k)\ P_{l,l}\otimes P_{l',l'} \Big) \\
	=&\ k^2\ \E_k[\ts_{kN}]\otimes\E_k[\ts_{kN}] \Big( \partial_i D_i Q(Z_k)\ P_{1,1}\otimes P_{1,1} \Big) .\nonumber
	\end{align}
	
	\noindent By combining equations \eqref{3limitnec11}, \eqref{3limitnec2} and \eqref{3limitnec3}, we get that
	\begin{align*}
	\Lambda_{N,t} = \lim\limits_{k\to\infty} & \quad - \Bigg\{ k \Bigg(\sum_{1\leq l,l'\leq k} \E_k\left[ \ts_{kN}\otimes\ts_{kN} \Big( \partial_i D_i Q(Z_k) P_{l',l}\otimes P_{l,l'} \Big) \right] \\
	& \quad\quad\quad\quad\quad\quad - \E_k[\ts_{kN}] \otimes \E_k[\ts_{kN}] \Big( \partial_i D_i Q(Z_k) P_{l',l}\otimes P_{l,l'} \Big)\Bigg) \\
	& \quad\quad\ + k^2(k-1) \E_k[\ts_{kN}] \otimes \E_k[\ts_{kN}] \Big( \partial_i D_i Q(Z_k) P_{1,2}\otimes P_{1,2} \Big) \Bigg\}. 
	\end{align*}
	
	\noindent Thanks to Lemma \ref{3coeffnndiag} (with $K^N=(X^N,Z^N)$), the last term converges towards $0$. Consequently,
	\begin{align}
	\label{3limitnec5}
	\Lambda_{N,t} = \lim\limits_{k\to\infty} & \quad - k\ \Bigg\{ \sum_{1\leq l,l'\leq k} \E_k\left[ \ts_{kN}\otimes\ts_{kN} \Big( \partial_i D_i Q(Z_k) P_{l',l}\otimes P_{l,l'} \Big) \right]  \\
	& \quad\quad\quad\quad\quad\quad - \E_k[\ts_{kN}] \otimes \E_k[\ts_{kN}] \Big( \partial_i D_i Q(Z_k) P_{l',l}\otimes P_{l,l'} \Big) \Bigg\} . \nonumber
	\end{align}
	
	\textbf{Step 3:} Let $R_s^{kN},S_s^{kN}$ for $s$ from $1$ to $n+1$ be independent families of $d_s$ independent GUE random matrices. We also assume that those families are independent from $X^N$ and $Y^{kN}$. We set the following notations,
	\begin{align*}
	Z_{k,r}^{1} = \Big(&(1-e^{-t_1})^{1/2}S^{kN}_1,\dots,(1-e^{-t_n})^{1/2}S^{kN}_n, \\
	& (1-e^{-r})^{1/2} S^{kN}_{n+1} + (e^{-r}-e^{-t})^{1/2}Y_{n+1}^{kN} + e^{-t/2} X^N\otimes I_k,Z^{N}\otimes I_k\Big) , \\
	Z_{k,r}^{1,s} = \Big(& (1-e^{-t_1})^{1/2}S^{kN}_1,\dots, (1-e^{-r})^{1/2} S_s^{kN} + (e^{-r}-e^{-t_s})^{1/2}Y^{kN}_s , (1-e^{-t_{s+1}})^{1/2}Y^{kN}_{s+1}, \\
	&\dots,(1-e^{-t_n})^{1/2}Y^{kN}_n,(1-e^{-t})^{1/2} Y^{kN}_{n+1} + e^{-t/2} X^N\otimes I_k ,Z^{N}\otimes I_k\Big) ,
	\end{align*}
	
	\noindent And similarly we define $Z_{k,r}^{2}$ and $Z_{k,r}^{2,s}$ but with $R_1^{kN}, \dots,R_s^{kN}$ instead of $S_1^{kN}, \dots,S_s^{kN}$. Thanks to Proposition \ref{3concentration}, and the fact that $Z^{1,s}_{k,t_s} = Z^{1,s+1}_{k,0}$ as well as $Z^{1,n}_{k,t_n} = Z^{1}_{k,0}$, we get that
	\begin{align*}
	&k\  \sum_{1\leq l,l'\leq k} \E_k\left[ \ts_{kN}\otimes\ts_{kN} \Big( \partial_i D_i Q(Z_k) P_{l',l}\otimes P_{l,l'} \Big) \right]  - \E_k[\ts_{kN}] \otimes \E_k[\ts_{kN}] \Big( \partial_i D_i Q(Z_k) P_{l',l}\otimes P_{l,l'} \Big) \\
	=&\ k \sum_{1\leq l,l'\leq k} \sum_{1\leq s \leq n} \E_k\left[ \ts_{kN}\Big( \partial_i^1 D_i Q(Z^{1,s}_{k,0}) P_{l',l}\Big)\boxtimes  \ts_{kN}\Big(\partial_i^2 D_i Q(Z^{2,s}_{k,0})P_{l,l'} \Big) \right] \\
	&\quad\quad\quad\quad\quad\quad\quad - \E_k\left[ \ts_{kN}\Big( \partial_i^1 D_i Q(Z^{1,s}_{k,t_s}) P_{l',l}\Big) \boxtimes \ts_{kN}\Big( \partial_i^2 D_i Q(Z^{2,s}_{k,t_s})P_{l,l'} \Big) \right] \\
	&\quad\quad\quad\quad + \E_k\left[ \ts_{kN}\Big( \partial_i^1 D_i Q(Z^{1}_{k,0}) P_{l',l}\Big) \boxtimes \ts_{kN}\Big( \partial_i^2 D_i Q(Z^{2}_{k,0})P_{l,l'} \Big) \right] \\
	&\quad\quad\quad\quad\quad - \E_k\left[ \ts_{kN}\Big( \partial_i^1 D_i Q(Z^{1}_{k,t}) P_{l',l}\Big) \boxtimes \ts_{kN}\Big( \partial_i^2 D_i Q(Z^{2}_{k,t})P_{l,l'} \Big) \right] \\
	=&\ \frac{1}{k^2 N^3}\sum_{1\leq l,l'\leq k} \sum_{\substack{1\leq s \leq n,\\ 1\leq j\leq d_s}} \int_0^{t_s} e^{-r}\ \E_k\Big[ \tr_{kN}\Big( \left(\partial_{s,j}\left( \partial_i^1 D_i Q\right)(Z_{k,r}^{1,s})\widetilde{\#} P_{l',l}\right) \boxtimes \left(\partial_{s,j}\left(\partial_i^2 D_i Q\right)(Z^{2,s}_{k,r})\widetilde{\#} P_{l,l'}\right) \Big) \Big] dr \\
	&\quad\quad\quad\quad\quad + \sum_{1\leq j\leq d} \int_0^{t} e^{-r}\ \E_k\Big[ \tr_{kN}\Big(  \left(\partial_{j}\left( \partial_i^1 D_i Q\right)(Z_{k,r}^{1})\widetilde{\#} P_{l',l}\right) \boxtimes \left(\partial_{j}\left(\partial_i^2 D_i Q\right)(Z^{2}_{k,r})\widetilde{\#} P_{l,l'}\right) \Big) \Big] dr,
	\end{align*}
	
	\noindent where $\widetilde{\#}$ is as in \eqref{3defperdu} and $\partial_{s,j}$ is defined similarly to the noncommutative differential introduced in Definition \ref{3application} but with respect to $(1-e^{-t_s})^{1/2}y_{s,j}$ instead of $X_i$. In order to illustrate the previous formula, we give the following example. For $Q\in\PP_{d_1+\dots+d_n+d,q}$ a polynomial, one can write $ \partial_i D_i Q = \sum_m A_m\otimes B_m$ for some monomials $A_m,B_m$, then with $X_{s,j}\in \PP_{d_1+\dots+d_n+d,q}$ such that $X_{s,j}(Z_k)= (1-e^{-t_s})^{1/2}y_{s,j}$, we have that 
	\begin{align*}
		 &\left(\partial_{s,j}\left( \partial_i^1 D_i Q\right)(Z_{k,r}^{1,s})\widetilde{\#} P_{l',l}\right) \boxtimes \left(\partial_{s,j}\left(\partial_i^2 D_i Q\right)(Z^{2,s}_{k,r})\widetilde{\#} P_{l,l'}\right) \\
		 &= \sum_{A=A_1 X_{s,j} A_2,\ B=B_1 X_{s,j} B_2} A_2(Z_{k,r}^{1,s})  P_{l',l} A_1(Z_{k,r}^{1,s}) B_2(Z^{2,s}_{k,r}) P_{l,l'} B_1(Z^{2,s}_{k,r}).
	\end{align*} \\

	\textbf{Step 4:} Besides if $U,V$ are matrices of $\M_N(\C) \otimes \M_k(\C)$, then 
	$$\sum_{1\leq l,l'\leq k} \tr_{kN}(U P_{l',l} V P_{l,l'}) = \tr_{N}( \id_N\otimes \tr_{k}(U) \id_N\otimes\tr_{k}(V)). $$
	Hence given $A,B,C,D\in\F_{d_1+\dots+d_{n}+d,q}$,
	\begin{align*}
	&\frac{1}{k^2N} \sum_{1\leq l,l'\leq k} \E_k\left[ \tr_{kN}\Big( A(Z_{k,r}^{1,s}) P_{l',l} B(Z_{k,r}^{1,s}) C(Z^{2,s}_{k,r}) P_{l,l'} D(Z^{2,s}_{k,r}) \Big) \right] \\
	=&\ \E_k\left[ \ts_N\left( \id_N\otimes\ts_k\left( D(Z^{2,s}_{k,r}) A(Z_{k,r}^{1,s}) \right) \id_N\otimes\ts_k \left( B(Z_{k,r}^{1,s}) C(Z^{2,s}_{k,r}) \right) \right) \right].
	\end{align*}
	
	\noindent If $Q = D(Z^{2,s}_{k,r}) A(Z_{k,r}^{1,s})$ and $T = B(Z_{k,r}^{1,s}) C(Z^{2,s}_{k,r})$, then thanks to Proposition \ref{3concentration},
	\begin{align*}
	&\ts_N\left(\E_k\left[ \id_N\otimes\ts_k(Q) \id_N\otimes\ts_k(T)\right]\right) \\
	&= \frac{1}{Nk^2}\sum_{\substack{1\leq i,j\leq N\\ 1\leq l,l'\leq k}} \E_k\left[ g_i^*\otimes f_l^* Q g_j\otimes f_l \times  g_j^*\otimes f_{l'}^* T g_i^*\otimes f_{l'}^* )\right]\\
	&= \frac{1}{Nk^2}\sum_{\substack{1\leq i,j,m,m'\leq N\\ 1\leq l,l'\leq k}} \E_k\Big[ g_m^*\otimes f_l^*\ Q\times  E_{j,i}\otimes I_k\ g_m\otimes f_l \times  g_{m'}^*\otimes f_{l'}^*\ T\times E_{i,j}\otimes I_k\ g_{m'}\otimes f_{l'} \Big]\\
	&= \frac{1}{Nk^2}\sum_{1\leq i,j\leq N} \E_k\left[ \tr_{kN}(Q E_{j,i}\otimes I_k)\ \tr_{kN}(T E_{i,j}\otimes I_k) \right] \\
	&= \mathcal{O}(k^{-2}) + \frac{1}{Nk^2}\sum_{1\leq i,j\leq N} \E_k\left[ \tr_{kN}(Q E_{j,i}\otimes I_k)\right] \left[ \tr_{kN}(T E_{i,j}\otimes I_k) \right]  \\
	&= \mathcal{O}(k^{-2}) +  \ts_N( \E_k\left[ \id_N\otimes\ts_k(Q)\right] \E_k\left[\id_N\otimes\ts_k(T)\right]).
	\end{align*}
	
	\noindent We can view a GUE matrix of size $kN$ as a matrix of size $N$ with matrix coefficients. The diagonal coefficients are independent GUE matrices of size $k$ multiplied by $N^{-1/2}$. The upper non-diagonal coefficients are independent random matrices of size $k$ which have the same law as $(2N)^{-1/2}(X+\i Y)$ where $X$ and $Y$ are independent GUE matrices of size $k$, and the lower non-diagonal coefficients are the adjoints of the upper coefficients. Thus, if $U^{kN}$ is a family of $l$ independent GUE matrices of size $kN$, $u$ a family of $l$ free semicircular variables, we then define $\mathbf{u}^{N}$ as a family of $l$ matrices of size $N$ whose diagonal coefficients are free semicirculars multiplied by $N^{-1/2}$, and the upper non-diagonal coefficients are free between each other, free from the diagonal one, and they are of the form $(2N)^{-1/2}(a+\i b)$ where $a$ and $b$ are free semicirculars. Finally the lower non-diagonal coefficients are the adjoints of the upper coefficients. We also assume that semicirculars from different matrices are free and that all of those semicirculars live in a $\CC^*$-algebra endowed with a trace $\tau$. Then with $\widetilde{U}^{kN}$ an independent copy of $U^{kN}$, $\widetilde{u}$ a free copy of $u$ and $\mathbf{\widetilde{u}}^{N}$ a free copy of $\mathbf{u}^{N}$, for $L,K\in\F_{l,q}$,
	\begin{align*}
	&\lim\limits_{k\to \infty} \ts_{N}\left(\E_k\left[ \id_N\otimes\ts_k(L(U^{kN},Z^N\otimes I_k))\right] \E_k\left[\id_N\otimes\ts_k(K(U^{kN},Z^N\otimes I_k))\right]\right) \\
	&= \ts_{N}\left( \id_N\otimes \tau(L(\mathbf{u}^{N},Z^N)) \times \id_N\otimes \tau(K(\mathbf{u}^{N},Z^N))\right) \\
	&= \ts_{N}\otimes \tau\left( L(\mathbf{u}^{N},Z^N) \times K(\mathbf{\widetilde{u}}^{N},Z^N)\right) \\
	&= \lim\limits_{k\to \infty} \E_k\left[\ts_{kN}\left( L(U^{kN},Z^N\otimes I_k) \times K(\widetilde{U}^{kN},Z^N\otimes I_k) \right)\right] \\
	&= \tau_{N}\left( L(u,Z^N) \times K(\widetilde{u},Z^N) \right)
	\end{align*}	
	
	\noindent Consequently, we have that 
	\begin{align*}
	&\lim_{k\to \infty} \frac{1}{k^2N} \sum_{1\leq l,l'\leq k} \E_k\left[ \tr_{kN}\Big( A(Z_{k,r}^{1,s}) P_{l',l} B(Z_{k,r}^{1,s}) C(Z^{2,s}_{k,r}) P_{l,l'} D(Z^{2,s}_{k,r}) \Big) \right] \\
	&= \tau_N\left( A(z_{r}^{1,s}) B(\widetilde{z}_{r}^{1,s}) C(\widetilde{z}^{2,s}_{r}) D(z^{2,s}_{r}) \right).
	\end{align*}
	
	\noindent In particular this implies
	\begin{align*}
	&\lim_{k\to \infty} \frac{1}{k^2 N}\sum_{1\leq l,l'\leq k} \E_k\Big[ \tr_{kN}\Big( \left(\partial_{s,j}\left( \partial_i^1 D_i Q\right)(Z_{k,r}^{1,s})\widetilde{\#} P_{l',l}\right) \boxtimes \left(\partial_{s,j}\left(\partial_i^2 D_i Q\right)(Z^{2,s}_{k,r})\widetilde{\#} P_{l,l'}\right) \Big) \Big] \\
	&= \tau_{N}\Big( \Big(\partial_{s,j}^2\left( \partial_i^1 D_i Q\right)(z_{r}^{1,s}) \boxtimes \partial_{s,j}^1 \left( \partial_i^1 D_i Q\right)(\widetilde{z}_{r}^{1,s}) \Big) \\
	&\quad\quad\quad \boxtimes \Big(\partial_{s,j}^2\left( \partial_i^2 D_i Q\right)(\widetilde{z}^{2,s}_{r}) \boxtimes \partial_{s,j}^1 \left( \partial_i^2 D_i Q\right)(z^{2,s}_{r}) \Big) \Big).
	\end{align*}
	
	\noindent Which in turn means that $\Lambda_{N,t}$ is equal to
	\begin{align*}
	-\frac{1}{N^2}\sum_{\substack{1\leq s \leq n,\\ 1\leq j\leq d_s}} \int_0^{t_s} e^{-r}\ \tau_{N}\Big(& \Big(\partial_{s,j}^2\left( \partial_i^1 D_i Q\right)(z_{r}^{1,s}) \boxtimes \partial_{s,j}^1 \left( \partial_i^1 D_i Q\right)(\widetilde{z}_{r}^{1,s}) \Big) \\
	& \boxtimes \Big(\partial_{s,j}^2\left( \partial_i^2 D_i Q\right)(\widetilde{z}^{2,s}_{r}) \boxtimes \partial_{s,j}^1 \left( \partial_i^2 D_i Q\right)(z^{2,s}_{r}) \Big) \Big) dr \\
	-\frac{1}{N^2}\sum_{1\leq j\leq d} \int_0^{t} e^{-r}\ \tau_{N}\Big(& \Big(\partial_{j}^2\left( \partial_i^1 D_i Q\right)(z_{r}^{1}) \boxtimes \partial_{j}^1 \left( \partial_i^1 D_i Q\right)(\widetilde{z}_{r}^{1}) \Big) \\
	& \boxtimes \Big(\partial_{j}^2\left( \partial_i^2 D_i Q\right)(\widetilde{z}^{2}_{r}) \boxtimes \partial_{j}^1 \left( \partial_i^2 D_i Q\right)(z^{2}_{r}) \Big) \Big) dr.
	\end{align*}
	
	\noindent Thus, by using this result in Equation \eqref{3vantlmbd}, we have in conclusion
	\begin{align*}
	&\E\Bigg[ \frac{d}{dt}\tau_N\Big(Q\left(Y_{t}^N\right)\Big) \Bigg] \\
	&= - \frac{e^{-t}}{2N^2} \sum_{\substack{1\leq i\leq d\\1\leq s \leq n,\\ 1\leq j\leq d_s} }\ \int_0^{t_s} e^{-r}\ \E\Big[ \tau_{N}\Big( \Big(\partial_{s,j}^2\left( \partial_i^1 D_i Q\right)(z_{r}^{1,s}) \boxtimes \partial_{s,j}^1 \left( \partial_i^1 D_i Q\right)(\widetilde{z}_{r}^{1,s}) \Big) \\
	&\quad\quad\quad\quad\quad\quad\quad\quad\quad\quad\quad\quad \boxtimes \Big(\partial_{s,j}^2\left( \partial_i^2 D_i Q\right)(\widetilde{z}^{2,s}_{r}) \boxtimes \partial_{s,j}^1 \left( \partial_i^2 D_i Q\right)(z^{2,s}_{r}) \Big) \Big) \Big] dr\\
	&\quad- \frac{e^{-t}}{2N^2} \sum_{1\leq i,j\leq d} \int_0^{t} e^{-r}\ \E\Big[ \tau_{N}\Big( \Big(\partial_{j}^2\left( \partial_i^1 D_i Q\right)(z_{r}^{1}) \boxtimes \partial_{j}^1 \left( \partial_i^1 D_i Q\right)(\widetilde{z}_{r}^{1}) \Big) \\
	&\quad\quad\quad\quad\quad\quad\quad\quad\quad\quad\quad\quad \boxtimes \Big(\partial_{j}^2\left( \partial_i^2 D_i Q\right)(\widetilde{z}^{2}_{r}) \boxtimes \partial_{j}^1 \left( \partial_i^2 D_i Q\right)(z^{2}_{r}) \Big) \Big) \Big] dr.
	\end{align*}
	
\end{proof}

\subsection{Proof of Theorem \ref{3lessopti}}

\label{3technical}

In this section we focus on proving Theorem \ref{3lessopti} from which we deduce all of the important corollaries. It will mainly be a corollary of the following theorem, which is slightly stronger but less explicit. We refer to  Lemma \ref{3apparition} for the definition of $L^{T_{i}}$ and $x^{T_i}$, and to Proposition \ref{3intercoef} for the one of $A_i$. To fully understand how the coefficients $\alpha_i^P(f,Z^N)$ are built we also refer to those Propositions. Some explicit computations on how to compute $\alpha_1^P(f,Z^N)$ for $P=X_1$ can be found in Example \ref{3exe33}.

\begin{theorem}
	\label{3TTheo}
	Let the following objects be given,
	\begin{itemize}
		\item $X^N = (X_1^N,\dots,X_d^N)$ independent $GUE$ matrices of size $N$,
		\item $Z^N = (Z_1^N,\dots,Z_q^N)$ deterministic matrices and their adjoints,
		\item $P\in \PP_{d,q}$ a polynomial that we assume to be self-adjoint,
		\item $f:\R\mapsto\R$ such that there exists a complex-valued measure on the real line $\mu$ with $$\int (1+y^{4(k+1)})\  d|\mu|(y)\ < +\infty$$ and for any $x\in\R$,
		\begin{equation}
		\label{3hypoth}
		f(x) = \int_{\R} e^{\i x y}\ d\mu(y) .
		\end{equation}
		
	\end{itemize}
	
	\noindent Then with notations as in Lemma \ref{3apparition} and Proposition \ref{3intercoef} if we set,
	\begin{equation}
		\label{exprdescoeff}
		\alpha_i^P(f,Z^N) = \int_{\R} \int_{A_i} \int_{[0,1]^{4i}} \tau_N\Big( \left(L^{{T}_i}_{\rho_i,\beta_i,\gamma_i,\delta_i} \dots L^{{T}_1}_{\rho_1,\beta_1,\gamma_1,\delta_1}\right)(e^{\i y P}) (x^{{T}_i},Z^N) \Big)\  d\rho\, d\beta\, d\gamma\, d\delta\ dt\ d\mu(y), 
	\end{equation} 
	
	\noindent and that we write $P = \sum_{1\leq i\leq Nb(P)} c_i M_i$ where the $M_i$ are monomials and $c_i\in\C$ (i.e. $P$ is a a sum of at most $Nb(P)$ monomials), if we set $C_{\max}(P) = \max \{1, \max_i |c_i|\}$, then there exist constants $C,K$ and $c$ independent of $P$ such that with $K_N = \max \{ \norm{Z^N_1}$ $,\dots, \norm{Z^N_q}, K\}$, for any $N$ and $k\leq cN (\deg P)^{-1}$,
	\begin{align}
	\label{3mainresu}
	&\left| \E\left[ \ts_{N}\Big(f(P(X^N,Z^N))\Big)\right] - \sum_{0\leq i\leq k} \frac{1}{N^{2i}} \alpha_i^P(f,Z^N) \right| \\
	&\leq \frac{1}{N^{2k+2}} \int_{\R} (|y|+ y^{4(k+1)}) d|\mu|(y) \times \Big(C\times K_N^{\deg P} C_{\max}(P) Nb(P)(\deg P)^2\Big)^{4(k+1)}\times k^{3k} . \nonumber
	\end{align}
	
	\noindent Besides, if we define $\widehat{K}_N$ like $K_N$ but with $2$ instead of $K$, then we have that for any $j\in\N^*$,
	\begin{equation}
	\label{3mainresu2}
	\left| \alpha_j^P(f,Z^N) \right| \leq \int_{\R} (|y|+ y^{4j}) d|\mu|(y) \times \Big(C\times \widehat{K}_N^{\deg P} C_{\max}(P) Nb(P)(\deg P)^2\Big)^{4j}\times j^{3j} .
	\end{equation}
	Finally if $f$ and $g$ both satisfy \eqref{3hypoth} for some complex measures $\mu_f$ and $\mu_g$, then if they are bounded functions equal on a neighborhood of the spectrum of $P(x,Z^N)$, where $x$ is a free semicircular system free from $\M_N(\C)$, then for any $i$, $\alpha_i^P(f,Z^N) = \alpha_i^P(g,Z^N)$. In particular if $f$ is a bounded function such that its support and the spectrum of $P(x,Z^N)$ are disjoint, then for any $i$, $\alpha_i^P(f,Z^N)=0$.
	
\end{theorem}

Note that it is quite important to be able to assume that $\mu$ can be a complex-valued measure. Indeed, this means that one can use the Fourier inversion formula and thus consider pretty much any functions smooth enough, as we sill see in the proof of Theorem \ref{3lessopti}.

\begin{rem}
\label{produf}
	It is worth noting that if one wanted, one could consider a product of functions $f_i$ evaluated in self-adjoint polynomials $P_i\in\PP_{d,q}$ instead of a single function $f$ evaluated in $P$. Indeed the proof of Theorem \ref{3TTheo} consists in first using Proposition \ref{3intercoef} and then estimating the remainder term. However Proposition \ref{3intercoef} can be used in more general situations. If we asume that for any $i$ and $x\in\R$,
	\begin{equation*}
		f_i(x) = \int_{\R} e^{\i x y}\ d\mu_i(y),
	\end{equation*}
	for some complex-valued measure $\mu_i$. Then given $R_i\in\PP_{d,q}, y_i\in\R$, $Q= e^{\i y_1 P_1}R_1\dots e^{\i y_k P_k}R_k$ belongs to $\F_{d,q}$. Consequently, one can apply Proposition \ref{3intercoef} to $Q$ and since 
	\begin{align*}
		&\E\left[ \ts_N\Big(f_1(P_1(X^N,Z^N))R_1(X^N,Z^N)\dots f_k(P_k(X^N,Z^N))R_k(X^N,Z^N)\Big) \right] \\
		&= \int_{\R^k} \E\left[ \ts_N\Big(Q(X^N,Z^N)\Big) \right] d\mu_1(y_1)\dots d\mu_k(y_k),
	\end{align*}
	one can obtain an asympotic expansion for any products of smooth functions.
	
	One can also study the case where we have a product of traces, to do so we use the Schwinger-Dyson equations to reduce the problem to the case of a single trace. Given matrices $A,B\in\M_N(\C)$, one has thanks to Proposition \ref{3SD}, that with $Y$ a GUE random matrix of size $N$,
	$$ \ts_{N}(A) \ts_{N}(B) = \E\left[ \ts_{N}\left(YAYB\right) \right]. $$
	Consequently, given $Q_1,\dots,Q_k\in\F_{d,q}$, $Y_1^N,\dots,Y_{k-1}^N$ independent GUE random matrices, independent from $X^N$, one has that
	\begin{align*}
		&\E\left[ \ts_N\Big(Q_1(X^N,Z^N))\Big)\dots \ts_N\Big(Q_k(X^N,Z^N)\Big) \right] \\
		&= \E\left[ \ts_N\Big(Y_{k-1}^N\dots Y_1^N Q_1(X^N,Z^N))Y^N_1Q_2(X^N,Z^N)) \dots Y_{k-1}^N Q_k(X^N,Z^N)\Big) \right].
	\end{align*}
	Hence once again one can use Proposition \ref{3intercoef} to get an asymptotic expansion.
\end{rem}

The following lemma allows us to define the coefficients of the topological expansion by induction. It is basically a reformulation of Lemma \ref{3imp2} with the notations of Definitions \ref{3biz} and \ref{3biz2}. Although the notations in this formula are a bit heavy, they are necessary in order to get a better upper bound on the remainder term. It is the first step of the proof of Theorem \ref{3TTheo}.

\begin{lemma}
	\label{3apparition}
	Let $x,x^1,\dots,x^{c_{n}}$ be free semicircular systems of $d$ variables. Then with $T_n = \{t_1,\dots,t_{2n}\}$ a sequence of non-negative number,  $T_n = \{\widetilde{t}_1,\dots,\widetilde{t}_{2n}\}$ the same set but ordered by increasing orders, and $I = \{I_1,\dots,I_{2n}\}\in J_n$, with $t_0=0$, we set
	$$ X_{i,I}^{N,T_{n}} = \sum_{l=1}^{2n} (e^{-\widetilde{t}_{l-1}} -e^{-\widetilde{t}_l})^{1/2} x^{I_{l}}_i + e^{-\widetilde{t}_{2n} /2} X_i^N , $$
	$$ x_{i,I}^{T_{n}} = \sum_{l=1}^{2n} (e^{-\widetilde{t}_{l-1}} -e^{-\widetilde{t}_l})^{1/2} x^{I_{l}}_i + e^{- \widetilde{t}_{2n}/2} x_i . $$
	
	\noindent We define the following subfamily of $(X_{i,I})_{i\in [1,d], I\in J_{n+1}}$,
	$$X_{l,1} = \left(X_{i,I}\right)_{i\in [1,d], I\in J_{n+1}^{l,1}}, X_{l,2} = \left(X_{i,I}\right)_{i\in [1,d],I\in J_{n+1}^{l,2}}, $$
	$$\widetilde{X}_{l,1} = \left(X_{i,I}\right)_{i\in [1,d],I\in \widetilde{J}_{n+1}^{l,1}}, \widetilde{X}_{l,2} = \left(\widetilde{X}_{i,I}\right)_{i\in [1,d],I\in \widetilde{J}_{n+1}^{l,2}}.$$
	
	\noindent Since there is a natural bijection between $J_n$ and $J_{n+1}^{l,1}$, one can evaluate an element of $\F_{d,q}^{n}$ in $(X_{l,1},Z)$ where $Z=(Y_1,\dots,Y_{2r})$ as in Definition \ref{3biz}, and similarly for $X_{l,2},\widetilde{X}_{l,1}$ and $\widetilde{X}_{l,2}$. 
	Then we define the following operators from $\F_{d,q}^{n}$ to $\F_{d,q}^{n+1}$, for $l$ from $1$ to $2n$,
	\begin{align*}
	&L_l^{n,\rho_{n+1},\beta_{n+1},\gamma_{n+1},\delta_{n+1}}(Q) \\
	&:= \frac{1}{2} \sum_{\substack{1\leq i,j\leq d \\ I,J\in J_n\\ \text{such that } I_l=J_l}} \Big(\partial_{\delta_{n+1},j,I}^2\left( \partial_{\beta_{n+1},i}^1 D_{\rho_{n+1},i} Q\right)\left(X_{l,1},Z\right) \boxtimes \partial_{\delta_{n+1},j,I}^1 \left( \partial_{\beta_{n+1},i}^1 D_{\rho_{n+1},i} Q\right)\left(\widetilde{X}_{l,1},Z\right) \Big) \\
	&\quad\quad\quad\quad\quad\quad\quad \boxtimes \Big(\partial_{\gamma_{n+1},j,J}^2\left( \partial_{\beta_{n+1},i}^2 D_{\rho_{n+1},i} Q\right)\left(\widetilde{X}_{l,2},Z\right) \boxtimes \partial_{\gamma_{n+1},j,J}^1 \left( \partial_{\beta_{n+1},i}^2 D_{\rho_{n+1},i} Q\right)\left(X_{l,2},Z\right) \Big). \\
	\end{align*} 
	
	\noindent We also define
	\begin{align*}
	&L_{2n+1}^{n,\rho_{n+1},\beta_{n+1},\gamma_{n+1},\delta_{n+1}}(Q) \\
	&:= \frac{1}{2} \sum_{1\leq i,j\leq d} \Big(\partial_{\delta_{n+1},j}^2\left( \partial_{\beta_{n+1},i}^1 D_{\rho_{n+1},i} Q\right)\left(X_{2n+1,1},Z\right) \boxtimes \partial_{\delta_{n+1},j}^1 \left( \partial_{\beta_{n+1},i}^1 D_{\rho_{n+1},i} Q\right)\left(\widetilde{X}_{2n+1,1},Z\right) \Big) \\
	&\quad\quad\quad\quad\quad \boxtimes \Big(\partial_{\gamma_{n+1},j}^2\left( \partial_{\beta_{n+1},i}^2 D_{\rho_{n+1},i} Q\right)\left(\widetilde{X}_{2n+1,2},Z\right) \boxtimes \partial_{\gamma_{n+1},j}^1 \left( \partial_{\beta_{n+1},i}^2 D_{\rho_{n+1},i} Q\right)\left(X_{2n+1,2},Z\right) \Big). \\
	\end{align*} 
	
	\noindent And finally, if $T_{n+1}$ a set of $2n+2$ numbers, with $\widetilde{T}_{n} = \{\widetilde{t}_1,\dots,\widetilde{t}_{2n}\}$ the set which contains the first $2n$ elements of $T_{n+1}$ but  sorted by increasing order, we set 
	\begin{align}
	\label{3fullop}
	L^{T_{n+1}}_{\rho_{n+1},\beta_{n+1},\gamma_{n+1},\delta_{n+1}}(Q) :=  e^{-t_{2n+2}-t_{2n+1}} \Bigg(& \1_{[\widetilde{t}_{2n},t_{2n+2}]}(t_{2n+1}) L^{n,\rho_{n+1},\beta_{n+1},\gamma_{n+1},\delta_{n+1}}_{2n+1}(Q) \\
	& + \sum_{1\leq l\leq 2n} \1_{[\widetilde{t}_{l-1},\widetilde{t}_l]}(t_{2n+1}) L_l^{n,\rho_{n+1},\beta_{n+1},\gamma_{n+1},\delta_{n+1}}(Q) \Bigg). \nonumber
	\end{align}
	
	\noindent Then, given $Q\in \F_{d,q}^{n}$, 
	
	\begin{align*}
	&\E\left[\tau_N\Big(Q(X^{N,T_n},Z^N)\Big)\right] - \tau_N\Big(Q(x^{T_n},Z^N)\Big)  \\
	&= \int_{\widetilde{t}_{2n}}^{\infty} \int_0^{t_{2n+2}} \int_{[0,1]^4}\tau_N\left( L^{T_{n+1}}_{\rho_{n+1},\beta_{n+1},\gamma_{n+1},\delta_{n+1}}(Q)\left( X^{N,T_{n+1}},Z^N \right) \right)\ d\rho_{n+1} d\beta_{n+1} d\gamma_{n+1} d\delta_{n+1}\ dt_{2n+1} dt_{2n+2}.
	\end{align*}

\end{lemma}

\begin{proof}	
	\noindent With the notations of Lemma \ref{3imp2}, if we set $ y_l = (x^s)_{\{s | \dep^{n}(s)=l\}}$, then let $S\in \F_{d(c_{n}+1),q}$ be such that
	\begin{align*}
	&S\left( \left(1-e^{-\left(\widetilde{t}_1-\widetilde{t}_0\right)}\right)^{1/2}y_1,\dots,\left(1-e^{-\left(\widetilde{t}_{2n}-\widetilde{t}_{2n-1}\right)}\right)^{1/2}y_{2n},X^N, Z^N \right) \\
	&= Q\left( \left(\sum_{l=1}^{2n}e^{-\widetilde{t}_{l-1}/2} \left(1 -e^{-(\widetilde{t}_l-\widetilde{t}_{l-1})}\right)^{1/2} x^{I_{l}}_i + e^{-\widetilde{t}_{2n} /2} X_i^N\right)_{1\leq i\leq d, I\in J_n} , Z^N\right) \\
	&= Q\left(\left(X_{i,I}^{N,T_{n}}\right)_{1\leq i\leq d, I\in J_n},Z^N\right)
	\end{align*}
	
	\noindent Consequently 
	\begin{equation}
		\label{derivpolydif}
		\partial_i D_i S = e^{-\widetilde{t}_{2n}}\ \partial_i D_i Q,
	\end{equation}
	where on the left side, since $S\in\F_{d(c_{n}+1),q}$, we used the noncommutative differential defined in Lemma \ref{3imp2}, whereas on the right side, since $Q\in\F_{d,q}^{n}$, we used the noncommutative differential defined in Definition \ref{3biz2}. Thus, with the convention of Lemma \ref{3detail} , we set $d_l$ to be $d$ times the number of $s\in [1,c_n]$ which have depth $l$ in $J_{n}$, then
	\begin{align*}
	&\E\left[\tau_N\Big(S\left(Y^N\right)\Big)\right] - \tau_N\Big(S\left(Y\right)\Big) \\
	&=  \frac{1}{2N^2} \int_0^{\infty} e^{-t} \sum_{\substack{1\leq i\leq d\\1\leq l \leq 2n,\\ 1\leq g\leq d_l} }\ \int_0^{\widetilde{t}_l-\widetilde{t}_{l-1}} e^{-r}\ \E\Big[ \tau_{N}\Big( \Big(\partial_{l,g}^2\left( \partial_i^1 D_i S\right)(z_{r}^{1,l}) \boxtimes \partial_{l,g}^1 \left( \partial_i^1 D_i S\right)(\widetilde{z}_{r}^{1,l}) \Big) \\
	&\quad\quad\quad\quad\quad\quad\quad\quad\quad\quad\quad\quad\quad\quad\quad\quad\quad\quad\quad \boxtimes \Big(\partial_{l,g}^2\left( \partial_i^2 D_i S\right)(\widetilde{z}^{2,l}_{r}) \boxtimes \partial_{l,g}^1 \left( \partial_i^2 D_i S\right)(z^{2,l}_{r}) \Big) \Big) \Big] dr\ dt \\
	&\quad +  \frac{1}{2N^2} \int_0^{\infty} e^{-t} \sum_{1\leq i,j\leq d }\ \int_0^{t} e^{-r}\ \E\Big[ \tau_{N}\Big( \Big(\partial_{j}^2\left( \partial_i^1 D_i S\right)(z_{r}^{1}) \boxtimes \partial_{j}^1 \left( \partial_i^1 D_i S\right)(\widetilde{z}_{r}^{1}) \Big) \\
	&\quad\quad\quad\quad\quad\quad\quad\quad\quad\quad\quad\quad\quad\quad\quad\quad\quad\quad \boxtimes \Big(\partial_{j}^2\left( \partial_i^2 D_i S\right)(\widetilde{z}^{2}_{r}) \boxtimes \partial_{j}^1 \left( \partial_i^2 D_i S\right)(z^{2}_{r}) \Big) \Big) \Big] dr\ dt .
	\end{align*}

	\noindent By definition, for any $g\in[1,d_l]$, there exist a unique $j\in [1,d]$ and $s\in[1,c_n]$ with $\dep^n(s)=l$, such that $\partial_{l,g}$ is the differential with respect to $(1-e^{-(\widetilde{t}_l-\widetilde{t}_{l-1})})^{1/2}x^s_{j}$. Consequently for $T\in\F_{d(c_{n}+1),q}$ and $R\in\F_{d,q}^{n}$ such that
	\begin{align*}
		&T\left( \left(1-e^{-(\widetilde{t}_1-\widetilde{t}_0)}\right)^{1/2}y_1,\dots,\left(1-e^{-(\widetilde{t}_{2n}-\widetilde{t}_{2n-1})}\right)^{1/2}y_{2n},X^N, Z^N \right) = R\left(\left(X_{i,I}^{N,T_{n}}\right)_{1\leq i\leq d, I\in J_n},Z^N\right),
	\end{align*}
	thanks to Lemma \ref{3detail} which guaranty that for any $I\in J_n$, either $s\notin I$ or $I_l=s$, one have that
	\begin{equation}
		\partial_{l,g} T = e^{-\widetilde{t}_{l}/2} \sum_{I\in J_n \text{ such that } I_l=s} \partial_{j,I} R.
	\end{equation}
	
	\noindent Consequently, since $S(Y^N) = Q(X^{N,T_n},Z^N)$ and $S(Y) = Q(x^{T_n},Z^N)$, we set 
	$$X^{N,{T_{n+1}}}_{l,1} = \left(X^{N,T_{n+1}}_I\right)_{I\in J_{n+1}^{l,1}}, X^{N,{T_{n+1}}}_{l,2} = \left(X^{N,T_{n+1}}_I\right)_{I\in J_{n+1}^{l,2}}, $$
	$$\widetilde{X}^{N,{T_{n+1}}}_{l,1} = \left(X^{N,T_{n+1}}_I\right)_{I\in \widetilde{J}_{n+1}^{l,1}}, \widetilde{X}^{N,{T_{n+1}}}_{l,2} = \left(\widetilde{X}^{N,T_{n+1}}_I\right)_{I\in \widetilde{J}_{n+1}^{l,2}}.$$

	\noindent and finally we have that
	
	\begin{align*}
	&\E\left[\tau_N\Big(Q(X^{N,T_n},Z^N)\Big)\right] - \tau_N\Big(Q(x^{T_n},Z^N)\Big)  \\
	&=  \frac{1}{2N^2} \int_0^{\infty} e^{-t-\widetilde{t}_{2n}} \sum_{\substack{1\leq i,j\leq d\\1\leq l \leq 2n}}\ \sum_{\substack{ I,J\in J_n\\\text{such that } I_l=J_l}} \int_0^{\widetilde{t}_l-\widetilde{t}_{l-1}}  e^{-r-\widetilde{t}_{l-1}}\\\
	&\quad \E\Big[ \tau_{N}\Big( \Big(\partial_{j,I}^2\left( \partial_i^1 D_i Q\right)\left(X^{N,\{T_n,r+\widetilde{t}_{l-1},t+\widetilde{t}_{2n}\}}_{l,1},Z^N\right) \boxtimes \partial_{j,I}^1 \left( \partial_i^1 D_i Q\right)\left(\widetilde{X}^{N,\{T_n,r+\widetilde{t}_{l-1},t+\widetilde{t}_{2n}\}}_{l,1},Z^N\right) \Big) \\
	&\quad\quad\quad \boxtimes \Big(\partial_{j,J}^2\left( \partial_i^2 D_i Q\right)\left(\widetilde{X}^{N,\{T_n,r+\widetilde{t}_{l-1},t+\widetilde{t}_{2n}\}}_{l,2},Z^N\right) \boxtimes \partial_{j,J}^1 \left( \partial_i^2 D_i Q\right)\left(X^{N,\{T_n,r+\widetilde{t}_{l-1},t+\widetilde{t}_{2n}\}}_{l,2},Z^N\right) \Big) \Big) \Big] dr\ dt \\
	&\quad +  \frac{1}{2N^2} \int_0^{\infty} e^{-t-\widetilde{t}_{2n}} \sum_{1\leq i,j\leq d }\ \int_0^{t} e^{-r-\widetilde{t}_{2n}}\\
	&\quad\quad\quad \E\Big[ \tau_{N}\Big( \Big(\partial_{j}^2\left( \partial_i^1 D_i Q\right)\left(X^{N,\{T_n,r+\widetilde{t}_{2n},t+\widetilde{t}_{2n}\}}_{n+1,1},Z^N\right) \boxtimes \partial_{j}^1 \left( \partial_i^1 D_i Q\right)\left(\widetilde{X}^{N,\{T_n,r+\widetilde{t}_{2n},t+\widetilde{t}_{2n}\}}_{n+1,1},Z^N\right) \Big) \\
	&\quad\quad\quad\quad\quad \boxtimes \Big(\partial_{j}^2\left( \partial_i^2 D_i Q\right)\left(\widetilde{X}^{N,\{T_n,r+\widetilde{t}_{2n},t+\widetilde{t}_{2n}\}}_{n+1,2},Z^N\right) \boxtimes \partial_{j}^1 \left( \partial_i^2 D_i Q\right)\left(X^{N,\{T_n,r+\widetilde{t}_{2n},t+\widetilde{t}_{2n}\}}_{n+1,2},Z^N\right) \Big) \Big) \Big] dr\ dt. \\ 
	\end{align*}
	
	\noindent Thus, if we set $T_{n+1} = \{ T_n, r,t \}$, after a change of variables, and by Definition \ref{3technicality} and \ref{3operatordef}, we get that
	
	\begin{align*}
	&\E\left[\tau_N\Big(Q(X^{N,T_n},Z^N)\Big)\right] - \tau_N\Big(Q(x^{T_n},Z^N)\Big)  \\
	&=\  \frac{1}{N^2} \sum_{1\leq l \leq 2n} \int_{\widetilde{t}_{2n}}^{\infty} \int_{\widetilde{t}_{l-1}}^{\widetilde{t}_l} \int_{[0,1]^4} \tau_{N}\left(L_l^{n,\rho_{n+1},\beta_{n+1},\gamma_{n+1},\delta_{n+1}}(Q)\left(X^{N,T_{n+1}},Z^N\right)\right)\ d\rho_{n+1}\, d\beta_{n+1}\, d\gamma_{n+1}\, d\delta_{n+1}\ dr\ dt\\
	&\quad +  \frac{1}{N^2} \int_{\widetilde{t}_{2n}}^{\infty} \int_{\widetilde{t}_{2n}}^{t} \int_{[0,1]^4} \tau_N\left(L_{2n+1}^{n,\rho_{n+1},\beta_{n+1},\gamma_{n+1},\delta_{n+1}}(Q)\left(X^{N,T_{n+1}},Z^N\right)\right)\ d\rho_{n+1}\, d\beta_{n+1}\, d\gamma_{n+1}\, d\delta_{n+1}\ dr\ dt. \\ 
	\end{align*}
	
	\noindent Hence the conclusion by renaming $r$ to $t_{2n+1}$ and $t$ to $t_{2n+2}$.
	
\end{proof}

%

\noindent Thus, we get the following proposition by induction.

\begin{prop}
	\label{3intercoef}
	Let $x$ be a free semicircular system, $(x^i)_{i\geq 1 }$ be free semicircular systems free from $x$, and $X^N$ be independent GUE matrices. We define $X^{N,T_n}$ and $x^{T_n}$ as in Lemma \ref{3apparition}, $A_i = \{ t_{2i}\geq t_{2i-2}\geq \dots \geq t_2\geq 0 \}\cap\{\forall s\in [1,i], t_{2s} \geq t_{2s-1} \geq 0\} \subset \R^{2i}$, then for any $Q\in \F_{d,q}$,
	\begin{align*}
	&\E\left[ \ts_N\Big( Q(X^N,Z^N) \Big) \right] \\
	&=\sum_{0\leq i\leq k}\ \frac{1}{N^{2i}} \int_{A_i } \int_{[0,1]^{4i}} \tau_N\Big( \left(L^{{T}_i}_{\rho_i,\beta_i,\gamma_i,\delta_i} \dots L^{{T}_1}_{\rho_1,\beta_1,\gamma_1,\delta_1}\right)(Q) (x^{T_i},Z^N) \Big)\ d\rho\, d\beta\, d\gamma\, d\delta\  dt_1\dots dt_{2i}\ \\
	&\quad + \frac{1}{N^{2(k+1)}} \int_{A_{k+1}} \int_{[0,1]^{4(k+1)}} \\
	&\quad\quad\quad \E\left[\tau_N\Big( \left(L^{{T}_{k+1}}_{\rho_{k+1},\beta_{k+1},\gamma_{k+1},\delta_{k+1}} \dots L^{{T}_1}_{\rho_1,\beta_1,\gamma_1,\delta_1}\right)(Q) (X^{N,T_{k+1}},Z^N) \Big)\right]\  d\rho\, d\beta\, d\gamma\, d\delta\ dt_1\dots dt_{2(k+1)} .
	\end{align*}
	
\end{prop}

\begin{proof}
	For $k=0$, we only need to apply Lemma \ref{3apparition} with $n=0$. Then if the formula is true for $k-1$, then since $\left(L^{{T}_{k}}_{\rho_{k},\beta_{k},\gamma_{k},\delta_{k}} \dots L^{{T}_1}_{\rho_1,\beta_1,\gamma_1,\delta_1}\right)(Q)$ is an element of $\F_{d,q}^k$, one can use Lemma \ref{3apparition} with $n=k$. Besides, for any $T_k=\{t_1,\dots,t_{2k}\}\in A_k$, we have $\widetilde{t}_{2k} = t_{2k}$. Hence the conclusion with the fact that
	$$ A_{k+1} = \{A_k\times\R^2\} \cap \{t_{2k+2}\geq t_{2k}, t_{2k+2}\geq t_{2k+1}\} .$$
\end{proof}

Before giving the proof of Theorem \ref{3TTheo}, as mentioned in the introduction, the former proposition gives some insight in map enumeration.

\begin{rem}
	\label{3map}
	We say that a graph on a surface is a map if it is connected and its faces are homeomorphic to discs. It is of genus $g$ if it can be embedded in a surface of genus $g$ but not $g-1$. For an edge-colored graph on an orientated surface we say that a vertex is of type $q=X_{i_1}\dots X_{i_p}$ if it has degree $p$ and when we look at the half-edges going out of it, starting from a distinguished one and going in the clockwise order the first half-edge is of color $i_1$, the second $i_2$, and so on. If $\mathcal{M}_g(X_{i_1}\dots X_{i_p})$ is the number of such maps of genus $g$ with a single vertex of type $q$, then given $X_i^N$ independent GUE matrices
	$$ \E\left[ \frac{1}{N}\tr_{N}\left( X_{i_1}^N\dots X_{i_p}^N \right) \right] = \sum_{g\in \N} \frac{1}{N^{2g}} \mathcal{M}_g(X_{i_1}\dots X_{i_p}) .$$
	
	\noindent For a proof we refer to \cite{harerzag} for the one matrix case and \cite[Chapter 22]{nica_speicher_2006}, for the multimatrix case. Thanks to Proposition \ref{3intercoef}, we immediately get that
	$$ \mathcal{M}_g(X_{i_1}\dots X_{i_p}) = \int_{A_g} \int_{[0,1]^{4g}} \tau\Big( \Big(L^{T_g}_{\rho_g,\beta_g,\gamma_g,\delta_g}\dots L^{T_1}_{\rho_1,\beta_1,\gamma_1,\delta_1}\Big)\left(X_{i_1}\dots X_{i_p}\right) (x^{T_g}) \Big)\  d\rho\, d\beta\, d\gamma\, d\delta\ dt_1\dots dt_{2g} .$$
	Note in particular that $L^{T_g}_{\rho_g,\beta_g,\gamma_g,\delta_g}\dots L^{T_1}_{\rho_1,\beta_1,\gamma_1,\delta_1}(X_{i_1}\dots X_{i_p})$ does not depend on $\rho,\beta,\gamma,\delta$. Indeed, $q$ is a polynomial, and hence as we can see in Definition \ref{3technicality}, the noncommutative differentials do not depend on $\rho,\beta,\gamma,\delta$. Consequently if one defines $L^{T_l}$ similarly to $L^{T_l}_{\rho_l,\beta_l,\gamma_l,\delta_l}$ but with $	\partial_{j,I}, \partial_{j,J}, \partial_{i}, D_{i}$ instead of $\partial_{\delta_l,j,I}, \partial_{\gamma_l,j,J}, \partial_{\beta_l,i}, D_{\rho_l,i}$, then 
	$$ \mathcal{M}_g(X_{i_1}\dots X_{i_p}) = \int_{A_g} \tau\Big( \Big(L^{T_g}\dots L^{T_1}\Big)\left(X_{i_1}\dots X_{i_p}\right) (x^{T_g}) \Big)\ dt_1\dots dt_{2g} .$$
		
\end{rem}

\begin{proof}[Proof of Theorem \ref{3TTheo}]
	
	The proof will be divided in two parts, first we prove Equations \eqref{3mainresu}, then we will prove the properties of the coefficients $\alpha_i^P(f,Z^N)$ that we listed in Theorem \ref{3TTheo}. 
	
	\textbf{Part 1:} Thanks to Proposition \ref{3intercoef}, we immediately get that 
	\begin{align*}
	\E\Big[ \ts_{N}&\Big(f(P(X^N,Z^N))\Big)\Big] = \sum_{0\leq i\leq k} \frac{1}{N^{2i}} \alpha_i^P(f,Z^N) \\
	&+ \frac{1}{N^{2(k+1)}} \int_{\R} \int_{A_{k+1}} \int_{[0,1]^{4(k+1)}} \E\left[\tau_N\Big( \left(L^{{T}_{k+1}}_{\rho_{k+1},\beta_{k+1},\gamma_{k+1},\delta_{k+1}} \dots L^{{T}_1}_{\rho_1,\beta_1,\gamma_1,\delta_1}\right)(Q) (X^{N,T_{k+1}},Z^N) \Big)\right] \\
	&\quad\quad\quad\quad\quad\quad\quad\quad\quad\quad\quad\quad\quad\quad\quad\quad\quad\quad\quad\quad\quad\quad\quad\quad\quad\quad\quad d\rho\, d\beta\, d\gamma\, d\delta\  dt_1\dots dt_{2(k+1)}\  d\mu(y) .
	\end{align*}
	
	\noindent All we need to do from now on is to get an estimate on the last line. Let $Q\in \F_{d,q}^n$,  we say that $M\in\F^n_{d,q}$ is a monomial if it is a monomial in $X_{i,I},Y_j$ and $\left\{e^{\i R}\ |\ R\in \mathcal{A}_{d,q}^n \text{ self-adjoint}\right\}$, we denote $\deg M$ the length of $M$ as a word in $X_{i,I},Y_j$ and $e^{\i R}$. Then we can write 
	$$ Q = \sum_{1\leq i\leq Nb(Q)} c_i M_i $$
	where $c_i\in\C$ and $M_i\in \F_{d,q}^n$ are monomials (not necessarily distinct). We also define $C_{\max}(Q) = \max \{1, \sup_i |c_i|\}$. Since for any $I\in J_n$, $\norm{X^{N,T_n}_{i,I}} \leq 2 + \norm{X^N_i}$, given 
	$$\mathcal{B} = \left\{ 2 + \norm{X_i^N}\right\}_{1\leq i\leq d} \bigcup \left\{{\norm{Z_j^N}} \right\}_{1\leq j\leq q} ,$$ and $D_N$ the maximum of this family, we get that 
	\begin{equation}
	\label{3majorgross}
	\norm{Q(X^{N,T_{n}},Z^N)} \leq Nb(Q) \times  C_{\max}(Q) \times D_N^{\deg(Q)} .
	\end{equation}
	
	\noindent It is worth noting that this upper bound is not optimal at all and heavily dependent on the decomposition chosen. We also consider $\widetilde{\mathcal{F}}_{d,q}^n$ the subalgebra of ${\mathcal{F}}_{d,q}^n$ generated by $\mathcal{A}_{d,q}^n$ and the family 
	$$\left\{e^{\i \lambda y P((X_{i,I})_{1\leq i\leq d},Y)}\ |\ I\in J_n, \lambda\in[0,1]\right\} .$$
	
	\noindent Then ${L}^{T_{n+1}}_{\rho_n,\beta_n,\gamma_n,\delta_n}$ sends $\widetilde{\mathcal{F}}_{d,q}^n$ to $\widetilde{\mathcal{F}}_{d,q}^{n+1}$. Let $Q\in \widetilde{\mathcal{F}}_{d,q}^n$, then we get that 
	$$\deg\left({L}^{T_{n+1}}_{\rho_n,\beta_n,\gamma_n,\delta_n}(Q)\right) \leq \deg Q + 4\deg P, $$
	$$C_{\max}\left({L}^{T_{n+1}}_{\rho_n,\beta_n,\gamma_n,\delta_n}(Q)\right) \leq e^{-t_{2n+2}-t_{2n+1}} (1+|y|)^4\ C_{\max}(P)^4\ C_{\max}(Q), $$
	\begin{align*}
	Nb\left({L}^{T_{n+1}}_{\rho_n,\beta_n,\gamma_n,\delta_n}(Q)\right) \leq &\deg(Q) (\deg Q + \deg P)(\deg Q + 2\deg P) \\
	&\times (\deg Q + 3\deg P) \times (Nb(P)\deg P)^4 \times Nb(Q) .
	\end{align*}
	
	\noindent Thus, if we define by induction $Q_0 = e^{\i y P}$, and $Q_{n+1} = {L}^{{T}_{n+1}}_{\rho_n,\beta_n,\gamma_n,\delta_n} Q_n $, since $\deg Q_0 = C_{\max}(Q_0) = Nb(Q_0)=1$, by a straightforward induction we get that
	\begin{equation}
	\deg Q_n \leq 4n \deg P +1
	\end{equation}
	\begin{equation}
	\label{3trucamodif}
	C_{\max}(Q_n) \leq e^{- \sum_{r=1}^{2n} t_{r}}  (1+|y|)^{4n}\ C_{\max}(P)^{4n}
	\end{equation}
	\begin{equation}
	Nb(Q_n) \leq \Big(Nb(P)\deg P\Big)^{4n} \prod\limits_{j= 0}^{4n-1} (j\deg P +1) \leq \Big(Nb(P)(\deg P)^2\Big)^{4n} (4n)!
	\end{equation}	
	
	\noindent Actually since we have $D_{\delta_1,i} e^{\i y P} = \i y\ \partial_{\delta_1,i} P \widetilde{\#} e^{\i y P} $, one can replace $(1+|y|)^{4n}$ in Equation \eqref{3trucamodif} by $|y|(1+|y|)^{4n-1}$. Thus, thanks to \eqref{3majorgross}, we get that 
	\begin{align*}
	&\norm{{L}^{T_{k+1}}_{\rho_{k+1},\beta_{k+1},\gamma_{k+1},\delta_{k+1}}\dots {L}^{T_1}_{\rho_1,\beta_1,\gamma_1,\delta_1} Q(X^{N,T_{k+1}},Z^N)} \\
	&\leq e^{- \sum_{r=1}^{2(k+1)} t_{r}}\times \frac{|y|}{1+|y|} \\
	&\quad \times \Big((1+|y|) C_{\max}(P) Nb(P)(\deg P)^2\Big)^{4(k+1)} (4(k+1))! \times D_N^{4(k+1)\deg P\ +1}
	\end{align*}
	
	\noindent Consequently after integrating over $\rho_n,\beta_n,\gamma_n,\delta_n$, we get that
	\begin{align*}
	&\Bigg| \int_{\R} \int_{A_{k+1}} \int_{[0,1]^{4(k+1)}} \E\left[\tau_N\Big( \left(L^{{T}_{k+1}}_{\rho_{k+1},\beta_{k+1},\gamma_{k+1},\delta_{k+1}} \dots L^{{T}_1}_{\rho_1,\beta_1,\gamma_1,\delta_1}\right)(Q) (X^{N,T_{k+1}},Z^N) \Big)\right] \\
	&\quad\quad\quad\quad\quad\quad\quad\quad\quad\quad\quad\quad\quad\quad\quad\quad\quad\quad\quad\quad\quad\quad\quad\quad\quad\quad\quad\quad d\rho\, d\beta\, d\gamma\, d\delta\  dt_1\dots dt_{2(k+1)}\  d\mu(y) \Bigg| \\
	&\leq \int_{A_{k+1}} e^{- \sum_{r=1}^{2(k+1)} t_{r}} dt_1\dots dt_{2k+2} \times \int_{\R} |y|(1+|y|)^{4k+3} d|\mu|(y) \\
	&\quad\quad  \times \Big( C_{\max}(P) Nb(P)(\deg P)^2\Big)^{4(k+1)} (4(k+1))! \times \E\left[ D_N^{4(k+1)\deg P\ +1} \right] .
	\end{align*}
	
	\noindent Besides
	\begin{align*}
	 \int_{A_{k+1}} e^{- \sum_{r=1}^{2(k+1)} t_{r}} dt_1\dots dt_{2k+2} &\leq \int\limits_{0\leq t_2\leq t_4\leq \dots \leq t_{2(k+1)}} \prod_{r=1}^{k+1} e^{-t_{2r}}\quad dt_1\dots dt_{k+1} \\
	&= \frac{1}{(k+1)!}\ ,
	\end{align*}
	
	\noindent and
	\begin{align*}
	\int_{\R} |y|(1+|y|)^{4k+3} d|\mu|(y) \leq 2^{4k+3} \int_{\R} (|y|+ y^{4(k+1)}) d|\mu|(y) .
	\end{align*}
	
	\noindent Thanks to Proposition \ref{3bornenorme} we can find constants $K$ and $c$ such that with 
	$$K_N = \max \{K, \norm{Z_1^N},\dots , \norm{Z_q^N}\} , $$ then for any $k\leq c (\deg P)^{-1} N,$
	$$\E\left[ D_N^{4(k+1)\deg P\ +1} \right] \leq K_N^{4(k+1)\deg P} .$$
	
	\noindent Thus, thanks to Stirling's formula, there exists a constant $C$ such that 
	\begin{align*}
	&\Bigg| \int_{\R} \int_{A_{k+1}} \int_{[0,1]^{4(k+1)}} \E\left[\tau_N\Big( \left(L^{{T}_{k+1}}_{\rho_{k+1},\beta_{k+1},\gamma_{k+1},\delta_{k+1}} \dots L^{{T}_1}_{\rho_1,\beta_1,\gamma_1,\delta_1}\right)(Q) (X^{N,T_{k+1}},Z^N) \Big)\right] \\
	&\quad\quad\quad\quad\quad\quad\quad\quad\quad\quad\quad\quad\quad\quad\quad\quad\quad\quad\quad\quad\quad\quad\quad\quad\quad\quad\quad\quad d\rho\, d\beta\, d\gamma\, d\delta\  dt_1\dots dt_{2(k+1)}\  d\mu(y) \Bigg| \\
	&\leq \int_{\R} (|y|+ y^{4(k+1)}) d|\mu|(y) \times \Big(C\times K_N^{\deg P} C_{\max}(P) Nb(P)(\deg P)^2\Big)^{4(k+1)}\times k^{3k}  .
	\end{align*}
	
	\noindent Hence we get Equation \eqref{3mainresu}. We get Equation \eqref{3mainresu2} very similarly. \\
	
	\textbf{Part 2:} Let us now prove that if $f$ and $g$ both satisfy \eqref{3hypoth} for some complex measures $\mu_f$ and $\mu_g$, then if they are bounded functions equal on a neighborhood of the spectrum of $P(x,Z^N)$, where $x$ is a free semicircular system free from $\M_N(\C)$, then for any $i$, $\alpha_i^P(f,Z^N) = \alpha_i^P(g,Z^N)$. Since the coefficients $\alpha_i^P(f,Z^N)$ are linear with respect to the measure $\mu_f$, we only need to show that given a function which takes the value $0$ on a neighborhood of the spectrum of $P(x,Z^N)$, then for any $i$, $\alpha_i^P(f,Z^N) = 0$. Let $X^{lN}$ be independent GUE matrices of size $lN$, then we get that for any $k$ such that $f$ is smooth enough, thanks to Equation \eqref{3mainresu}, 
	$$ \E\left[ \ts_{lN}\Big(f(P(X^{lN},Z^N\otimes I_l))\Big)\right] = \sum_{0\leq i\leq k} \frac{1}{(lN)^{2i}} \alpha_i^P(f,Z^N\otimes I_l) + \mathcal{O}(l^{-2(k+1)}) .$$
	
	\noindent But in the sense of Definition \ref{3freeprob}, for any $i$, $(x^{T_i},Z^N\otimes I_l)$ and $(x^{T_i},Z^N)$ have the same joint $*$-distribution, hence
	$$ \E\left[ \ts_{lN}\Big(f(P(X^{lN},Z^N\otimes I_l))\Big)\right] = \sum_{0\leq i\leq k} \frac{1}{(lN)^{2i}} \alpha_i^P(f,Z^N) + \mathcal{O}(l^{-2(k+1)}) .$$
	
	\noindent Consequently, if there exists $i$ such that $\alpha_i^P(f,Z^N)\neq 0$, then we can find constants $c$ and $k$ (dependent on $N$) such that 
	\begin{equation}
	\label{3contradhypo}
	\E\left[ \ts_{lN}\Big(f(P(X^{lN},Z^N\otimes I_l))\Big)\right] \sim_{l\to\infty} c\times l^{-2k} . 
	\end{equation}
	
	\noindent We are going to show that the left hand side decays exponentially fast in $l$, hence proving a contradiction. Now if we set $E$ the support of $f$, then
	$$ \left| \E\left[ \ts_{lN}\Big(f(P(X^{lN},Z^N\otimes I_l))\Big)\right] \right| \leq \norm{f}_{\CC^0} \P\left( \sigma\left( P(X^{lN},Z^N\otimes I_l) \right) \cap E \neq \emptyset \right) .$$
	However, thanks to Proposition \ref{3bornenorme}, there exist constants $A$ and $B$ such that for any $l$,
	$$ \P\left( \norm{P(X^{lN},Z^N\otimes I_l)} \geq A \right) \leq e^{-B l} . $$
	\noindent Thus,
	$$ \left| \E\left[ \ts_{lN}\Big(f(P(X^{lN},Z^N\otimes I_l))\Big)\right] \right| \leq \norm{f}_{\CC^0} \Big(\P\left( \sigma\left( P(X^{lN},Z^N\otimes I_l) \right) \cap E \cap [-A,A] \neq \emptyset \right) + e^{-Bl}\Big) .$$
	Let $g$ be a $\mathcal{C}^{\infty}$-function which takes non-negative values, with compact support disjoint from the spectrum of $P(x,Z^N)$ and such that $g_{|E\cap[-A,A]}=1$. Then,
	$$ \left| \E\left[ \ts_{lN}\Big(f(P(X^{lN},Z^N\otimes I_l))\Big)\right] \right| \leq \norm{f}_{\CC^0} \P\left( \norm{g\left( P(X^{lN},Z^N\otimes I_l) \right)} \geq 1 \right) + e^{-Bl} .$$
	
	\noindent Since $g$ is $\mathcal{C}^{\infty}$ and has compact support, thanks to the Fourier inversion formula, we have with $ \hat{g}(y) = \frac{1}{2\pi} \int_{\R} g(x) e^{-\i xy} dx$, that 
	$$ g(x) = \int_{\R} e^{\i x y}\ \hat{g}(y)\  dy, $$
	and besides $\int |y\hat{g}(y)| dy<\infty$. Thus, for any self-adjoint matrices $U$ and $V$,
	\begin{align*}
	\norm{g(U)-g(V)} &= \norm{ \int y \int_0^1 e^{\i y U \alpha} (U-V) e^{\i y V (1-\alpha)} \hat{g}(y) d\alpha dy } \\
	&\leq \norm{U-V} \int  |y\hat{g}(y)| dy .
	\end{align*}
	
	\noindent Hence there is a constant $C_B$ such that for any self-adjoint matrices $X_i,Y_i\in\M_{lN}(\C)$ whose operator norm is bounded by $B$,
	\begin{align*}
	\norm{g(P(X,Z^N)) - g(P(Y,Z^N))} \leq C_B \sum_i \norm{X_i - Y_i} .
	\end{align*}
	
	\noindent Consequently, with a proof very similar to the one of \cite[Proposition 4.6]{un}, we get that there exist constant $D$ and $S$ such that for any $\delta>0$,
	$$ \P\left( \left| \norm{g\left( P(X^{lN},Z^N\otimes I_l)\right)} - \E\left[\norm{g\left( P(X^{lN},Z^N\otimes I_l)\right)}\right] \right| \geq \delta + D e^{-Nl} \right) \leq d e^{-2Nl} + e^{-S \delta^2 lN} .$$
	
	\noindent By using Equation \eqref{3mainresu} with $k=0$, we get that
	\begin{align*}
		&\E\left[\norm{g\left( P(X^{lN},Z^N\otimes I_l)\right)}\right] \\
		&\leq \E\left[\tr_{lN}\left(g\left( P\left(X^{lN},Z^N\otimes I_l\right)\right)\right)\right] \\
		&= lN\tau_{lN}\left(g\left( P\left(x,Z^N\otimes I_l\right)\right)\right) + \mathcal{O}(l^{-1})\\
		&= lN\tau_N\left(g\left( P\left(x,Z^N\right)\right)\right) + \mathcal{O}(l^{-1})\\
		&= \mathcal{O}(l^{-1}).
	\end{align*}
	Hence for $l$ large enough, there exists a constant $S$ such that
	$$ \P\left( \norm{g\left( P(X^{lN},Z^N\otimes I_l) \right)} \geq 1 \right) \leq e^{-S l} .$$
	
	\noindent Consequently, there exist constants $A$ and $B$ such that 
	$$ \left| \E\left[ \ts_{lN}\Big(f(P(X^{lN},Z^N\otimes I_l))\Big)\right] \right| \leq A e^{-B l} ,$$
	which is in contradiction with Equation \eqref{3contradhypo}. Hence the conclusion.
	
\end{proof}

We can now prove Theorem \ref{3lessopti}, the only difficulty of the proof is to use the hypothesis of smoothness to replace our function $f$ by a function which satisfies \eqref{3hypoth} without losing too much on the constants.

\begin{proof}[Proof of Theorem \ref{3lessopti}]
	
	To begin with, let
	\begin{equation}
	\label{3fctplateau}
		h:x\to \left\{
		\begin{array}{ll}
			e^{-x^{-4} - (1-x)^{-4}} & \mbox{if } x\in (0,1), \\
			0 & \mbox{else.}
		\end{array}
		\right.
	\end{equation}
	
	\noindent Let $H$ be the primitive of $h$ which takes the value $0$ on $\R^-$ and renormalized such that it takes the value $1$ for $x\geq 1$. Then given a constant $m$ one can define the function $g : x\to H(m+1-x)H(m+1+x)$ which takes the value $1$ on $[-m,m]$ and $0$ outside of $(-m-1,m+1)$. Let $B$ be the union over $i$ of the events $\{\norm{X_i^N} \geq D + \alpha^{-1}\}$ where $D$ and $\alpha$ where defined in Proposition \ref{3bornenorme}. Thus, $\P(B)\leq p e^{-N}$. By adjusting the constant $K$ defined in Theorem \ref{3TTheo} we can always assume that it is larger than $D + \alpha^{-1}$, thus if for any $i$, $\norm{X_i^N} \leq D + \alpha^{-1}$, $\norm{P(X^N,Z^N)} \leq \mathbf{m} C_{\max} K_N^n$. We fix $m = \sup_N \mathbf{m} C_{\max} K_N^n$, thus if $P(X^N,Z^N)$ has an eigenvalue outside of $[-m,m]$, necessarily $X^N\in B$. Thus
	\begin{equation}
	\label{3alinf}
	\E\left[ \ts_{N}\Big(f(1-g)(P(X^N,Z^N))\Big)\right] \leq \norm{f}_{\CC^0} \P(B) \leq \norm{f}_{\CC^0} p \times e^{-N} .
	\end{equation}
	
	\noindent Since $fg$ has compact support and is a function of class $\mathcal{C}^{4(k+1)+2}$, we can take its Fourier transform and then invert it so that with the convention $ \hat{h}(y) = \frac{1}{2\pi} \int_{\R} h(x) e^{-\i xy} dx$, we have
	
	$$ \forall x\in\R,\quad (fg)(x) = \int_{\R} e^{\i xy} \widehat{fg}(y)\ dy. $$
	
	\noindent Besides, since if $h$ has compact support bounded by $m+1$ then $\norm{\hat{h}}_{\CC^0} \leq \frac{1}{\pi}(m+1) \norm{h}_{\CC^0} $, we have
	
	\begin{align*}
	\int_{\R} (|y|+ y^{4(k+1)}) \left| \widehat{fg}(y) \right|\ dy &\leq \int_{\R} \frac{\sum_{i=0}^{4(k+1)+2} |y|^i }{1+y^2}\ \left| \widehat{fg}(y) \right|\ dy \\
	&\leq \bigintss_{\R} \frac{\sum_{i=0}^{4(k+1)+2} \left| \widehat{(fg)^{(i)}}(y) \right| }{1+y^2}\ dy \\
	&\leq \frac{1}{\pi} \left( m + 1\right) \norm{fg}_{\mathcal{C}^{4(k+1)+2}} \int_{\R} \frac{1}{1+y^2}\ dy \\
	&\leq \left( m + 1\right) \norm{fg}_{\mathcal{C}^{4(k+1)+2}} \,
	\end{align*}
	
	\noindent Hence $fg$ satisfies the hypothesis of Theorem \ref{3TTheo} with $\mu(dy) = \widehat{fg}(y) dy$. 
	Therefore,  combining with Equation \eqref{3alinf}, by adjusting the constant $C$, we get that
	\begin{align*}
	&\left| \E\left[ \ts_{N}\Big(f(P(X^N,Z^N))\Big)\right] - \sum_{0\leq i\leq k} \frac{1}{N^{2i}} \alpha_i^P(fg,Z^N) \right| \\
	&\leq \frac{1}{N^{2k+2}} \norm{fg}_{\mathcal{C}^{4(k+1)+2}} \times \Big(C\times K_N^{\deg P} C_{\max}(P) Nb(P)(\deg P)^2\Big)^{4(k+1)+1}\times k^{3k} . \nonumber
	\end{align*}

	\noindent Then one sets $\alpha_i^P(f,Z^N) = \alpha_i^P(fg,Z^N)$. Besides, if $f$ and $g$ are functions of class $\mathcal{C}^{4(k+1)+2}$ equal on a neighborhood of the spectrum of $P(x,Z^N)$, where $x$ is a free semicircular system free from $\M_N(\C)$, then with the same proof as in Theorem \ref{3TTheo}, one has that for any $i\leq k$, $\alpha_i^P(f,Z^N) = \alpha_i^P(g,Z^N)$.
	
	Finally, one can write the $j$-th derivative of $x\to e^{-x^{-4}} $ on $\R^+$ as $x\to Q_j(x^{-1})e^{-x^{-4}} $ for some polynomial $Q_j$. By studying $Nb(Q_j), C_{\max}(Q_j) $ and $\deg(Q_j)$, as in the proof of Theorem \ref{3TTheo}, we get that the infinity norm of the $j$-th derivative of this function is smaller than $20^j j! (5j/4)^{5j/4} $. Hence by adjusting $C$ and using Stirling's formula,
	\begin{align*}
	&\left| \E\left[ \ts_{N}\Big(f(P(X^N,Z^N))\Big)\right] - \sum_{0\leq i\leq k} \frac{1}{N^{2i}} \alpha_i^P(fg,Z^N) \right| \\
	&\leq \frac{1}{N^{2k+2}} \norm{f}_{\mathcal{C}^{4(k+1)+2}} \times \Big(C\times K_N^{\deg P} C_{\max}(P) Nb(P)(\deg P)^2\Big)^{4(k+1)+1}\times k^{12k} . \nonumber
	\end{align*}
	
	\noindent The other points of the theorem are a direct consequence of Theorem \ref{3TTheo}.
	
\end{proof}

\subsection{Continuity properties of the coefficients of the asymptotic expansion}

The aim of this subsection is to give some details on the continuity of the coefficients $\alpha_i^P(f,Z^N)$ with respect to their parameters. Indeed, one have the following corollary of Theorem \ref{3TTheo} and more specifically Formula \eqref{exprdescoeff}.

\begin{cor}
	\label{continucoeff}
	With notations and assumptions as in Theorem \ref{3TTheo}, given the following objects,
	\begin{itemize}
		\item $f,g:\R\to\R\in\CC^{4i+2}$,
		\item $P,Q\in \PP_{d,q}$ polynomials of degree at most $n$ and largest coefficient $c_{\max}$,
		\item $Z^N$ and $\widetilde{Z}^N$ tuples of matrices such that for every $i$, $\norm{Z_i^N} \leq K$ and $\norm{\widetilde{Z}_i^N} \leq K$,
	\end{itemize}
	Then there exist constants $C_i(n,c_{\max},K), C_i^1(n,c_{\max},K,\norm{f}_{4i+2}), C_i^2(n,c_{\max},K,\norm{f}_{4i+2})\geq 0$ such that with $c_M(\cdot)$ defined as in Equation \eqref{3normA},
	\begin{equation}
		\label{result1}
		\left| \alpha_i^P(f,Z^N) - \alpha_i^P(g,Z^N) \right| \leq C_i(n,c_{\max},K)\ \norm{f-g}_{\CC^i},
	\end{equation}
	\begin{equation}
		\label{result2}
		\left| \alpha_i^P(f,Z^N) - \alpha_i^Q(f,Z^N) \right| \leq C_i^1(n,c_{\max},K,\norm{f}_{4i+2})\ \sup_{M \text{ monomial}} |c_M(P)-c_M(Q)|,
	\end{equation}
	\begin{equation}
		\label{result3}
		\left| \alpha_i^P(f,Z^N) - \alpha_i^P(f,\widetilde{Z}^N) \right| \leq C_i^2(n,c_{\max},K,\norm{f}_{4i+2})\ \max_i \norm{Z_i^N-\widetilde{Z}_i^N}.
	\end{equation}
	Besides, if $Z^N$ converges in distribution (as defined in Definition \ref{3freeprob}) towards a family $z$, then $\alpha_i^P(f,Z^N)$ converges towards $\alpha_i^P(f,z)$.
\end{cor}

Note that one could estimate the constants $C_i^P(Z^N), C_i^{m,c_{\max}}(Z^N,f)$ and $C_i^P(K,f)$ with respect to their parameters, similarly to how we obtain Equations \eqref{3mainresu} and \eqref{3mainresu2}. However, we do not do it here in order to keep the computations short.

\begin{proof}
	The uniqueness of the coefficients $\alpha_i^P(f,Z^N)$ coupled with the linearity of the map $$f\mapsto \E\left[ \ts_{N}\Big(f(P(X^N,Z^N))\Big)\right],$$ implies the linearity of the map $f\mapsto \alpha_i^P(f,Z^N)$. Hence Equation \eqref{3mainresu02} implies Equation \eqref{result1}.
	
	Besides, with $P$ and $Q$ defined as previously, we have thanks to Equation \eqref{3duha} that
	\begin{align*}
		&\E\left[ \ts_N\left( e^{\i y P(X^{N},Z^N)} \right) \right] - \E\left[ \ts_N\left( e^{\i y Q(X^{N},Z^N)} \right) \right] \\
		&=\i y \int_{0}^1 \E\left[ \ts_N\left( e^{\i y u P(X^{N},Z^N)} \left(P(X^{N},Z^N) - Q(X^{N},Z^N)\right) e^{\i y (1-u) Q(X^{N},Z^N)} \right) \right]\ du.
	\end{align*}
	Hence thanks to Proposition \ref{3intercoef}, we get that
	\begin{align*}
		 &\int_{[0,1]^{4i}} \tau_N\Big( \left(L^{{T}_i}_{\rho_i,\beta_i,\gamma_i,\delta_i} \dots L^{{T}_1}_{\rho_1,\beta_1,\gamma_1,\delta_1}\right)(e^{\i y P}-e^{\i y Q}) (x^{{T}_i},Z^N) \Big)\  d\rho\, d\beta\, d\gamma\, d\delta \\
		 &= \i y \int_{0}^1  \int_{[0,1]^{4i}} \tau_N\Big( \left(L^{{T}_i}_{\rho_i,\beta_i,\gamma_i,\delta_i} \dots L^{{T}_1}_{\rho_1,\beta_1,\gamma_1,\delta_1}\right)(e^{\i y u P} \left(P - Q\right) e^{\i y (1-u) Q}) (x^{{T}_i},Z^N) \Big)\  d\rho\, d\beta\, d\gamma\, d\delta\ du,
	\end{align*}
	Consequently after integrating over $y$, we get that
	\begin{align*}
		&\alpha_i^P(f,Z^N) - \alpha_i^Q(f,Z^N) \\
		&= \i\int_{0}^1  \int_{\R} y \int_{[0,1]^{4i}} \tau_N\Big( \left(L^{{T}_i}_{\rho_i,\beta_i,\gamma_i,\delta_i} \dots L^{{T}_1}_{\rho_1,\beta_1,\gamma_1,\delta_1}\right)(e^{\i y u P} \left(P - Q\right) e^{\i y (1-u) Q}) (x^{{T}_i},Z^N) \Big)\\
		&\quad\quad\quad\quad\quad\quad\quad\quad\quad\quad\quad\quad\quad\quad\quad\quad\quad\quad\quad\quad\quad\quad\quad\quad\quad\quad\quad\quad\quad\quad\quad\quad\quad\quad  d\rho\, d\beta\, d\gamma\, d\delta\ d\mu(y)\ du
	\end{align*}
	Since one can write 
	$$ P-Q = \sum_{M\text{ monomial}} (c_M(P)-c_M(Q)) M,$$
	we get Equation \eqref{result2}. Similarly we have that
	\begin{align*}
		 P(X,Z) - P(X,\widetilde{Z}) &= \sum_{M\text{ monomial}} c_M(P) \left(M(X,Z) - M(X,\widetilde{Z})\right) \\
		 &= \sum_{M\text{ monomial}} c_M(P) \sum_{M=AZ_iB} A(X,Z)(Z_i-\widetilde{Z}_i)B(X,\widetilde{Z}),
	\end{align*}
	hence we get Equation \eqref{result3}. Finally if $Z^N$ converges in distribution towards a family $z$, then the family $(x^{{T}_i},Z^N)$ converges in distribution towards $(x^{{T}_i},z)$ where $z$ is free from $x^{{T}_i}$. Indeed, thanks to Equation \ref{kddkdxkfl}, the trace of a polynomial $L$ evaluated in $(x^{{T}_i},Z^N)$ can be expressed into a linear combination of product of traces of polynomials in either $x^{{T}_i}$ or $Z^N$. Then the convergence in distribution of the family $Z^N$ implies that this formula converges towards the same linear combination but whose polynomials are evaluated into $x^{{T}_i}$ or $z$ instead of $x^{{T}_i}$ or $Z^N$, that is the trace of the polynomial $L$ evaluated into $(x^{{T}_i},z)$ where the family $x^{{T}_i}$ and $z$ are free. Thus, thanks to the dominated convergence theorem, Formula \eqref{exprdescoeff} implies the convergence of $\alpha_i^P(f,Z^N)$ towards $\alpha_i^P(f,z)$.

\end{proof}

\section{Consequences of Theorem \ref{3TTheo}}

\subsection{Proof of corollary \ref{3voisinage}}

Let $g$ be a non-negative $\mathcal{C}^{\infty}$-function which takes the value $0$ on $(-\infty,1/2]$, $1$ on $[1,\infty)$ and values in $[0,1]$ elsewhere. For any $a,b\in \R\cup \{\infty, -\infty\}$, we define $h_{(a,b)}^{\varepsilon} : x\mapsto g(\varepsilon^{-1}(x-a)) g(-\varepsilon^{-1}(x-b) )$ with convention $g(\infty)=1$. Then let $\mathcal{I}_N$ be the collection of connected components of the complementary set of $\sigma(P(x,A^N))$. Then we define
$$ h^{\varepsilon}_N = \sum_{I\in \mathcal{I}_N} h_I^{\varepsilon} .$$
This function is well-defined since the spectrum of $P(x,A^N)$ is compact, hence its complementary set has a finite number of connected components of measure larger than $\varepsilon$, and since $h_{(a,b)}^{\varepsilon} = 0$ when $b-a\leq \varepsilon$, the sum over $I\in\mathcal{I}_N$ is actually a finite sum. Besides, we have that
\begin{align*}
\P\left( \sigma(P(X^N,A^N)) \not\subset \sigma(P(x,A^N)) + \varepsilon  \right)\ &\leq\ \P\left( \norm{h^{\varepsilon}(P(X^N,A^N))} \geq 1 \right) \\
&\leq\ \E\left[ \tr_N\left( h^{\varepsilon}(P(X^N,A^N)) \right) \right] .
\end{align*}

\noindent Besides, $\norm{h_I^{\varepsilon}}_{\mathcal{C}^{4(k+1) +2}}$ is bounded by $C_k\varepsilon ^{-4k-6}$ for $\varepsilon$ small enough where $C_k$ is a constant independent of $N$, and since the supports of functions $h_I^{\varepsilon}$ are disjoint for $I\in\mathcal{I}_N$, we have that $\norm{h_N^{\varepsilon}}_{\mathcal{C}^{4(k+1) +2}}$ is also bounded by $C_k\varepsilon ^{-4k-6}$. Then thanks to Theorem \ref{3lessopti} since the spectrum of $P(x,A^N)$ and the support of $h^{\varepsilon}_N$ are disjoint, in combination with the assumption that the operator norm of the matrices $A^N$ is uniformly bounded over $N$, for any $k\in\N$, we get that there is a constant $C_k$ such that for any $\varepsilon$ and for $N$ large enough, 
$$ \E\left[ \tr_N\left( h^{\varepsilon}(P(X^N,A^N)) \right) \right] \leq C_k \frac{\varepsilon ^{-4k-6}}{N^{2k+1}} .$$
Thus, if we set $\varepsilon = N^{-\alpha}$ with $\alpha<1/2$, then by fixing $k$ large enough we get that
$$ \P\left( \sigma(P(X^N,A^N)) \not\subset \sigma(P(x,A^N)) + (-N^{-\alpha},N^{-\alpha})  \right) = \mathcal{O}(N^{-2}) .$$
Hence the conclusion by Borel-Cantelli lemma.

\subsection{Proof of Corollary \ref{3boundednormrenm}}

Firstly, we need the following lemma.

\begin{lemma}
	\label{3meilleurestime}
	There exists a $\mathcal{C}^{\infty}$ function $g$ which takes the value $0$ on $(-\infty,1/2]$, $1$ on $[1,\infty)$, and in $[0,1]$ otherwise, such that if we set $f_{\varepsilon}:t\mapsto g(\varepsilon^{-1} (t - \alpha))$ with $\alpha = \norm{PP^*(x,A^N)}$, then there exist constants $C$ and $c$ such that for any $k\leq c N$, $\varepsilon>0$ and $N$,
	\begin{equation*}
	\E\left[\tr_{N}\Big(f_{\varepsilon}(PP^*(X^N,A^N))\Big)\right] \leq N\times C^k \frac{\max(\varepsilon^{-4k},\varepsilon^{-1})}{N^{2k}} k^{12k}.
	\end{equation*}
\end{lemma}

\begin{proof}
	
	To estimate the above  expectation we once again want to use the Fourier transform with a few refinements to have an optimal estimate with respect to $\varepsilon$. Let $g$ be a function which takes the value $0$ on $(-\infty,1/2]$, $1$ on $[1,\infty)$, and in $[0,1]$ otherwise. We set $f^{\kappa}_{\varepsilon}:t\mapsto g(\varepsilon^{-1} (t - \alpha)) g(\varepsilon^{-1} (\kappa -t)+1)$ with $\kappa>\alpha $. Since $f^{\kappa}_{\varepsilon}$ has compact support and is sufficiently smooth we can apply Theorem \ref{3TTheo}. Setting $h:t\mapsto g(t - \varepsilon^{-1} \alpha) g( \varepsilon^{-1}\kappa +1  -t)= f^{\kappa}_{\varepsilon}(\varepsilon t)$, we have for $k\in\N^*$,
	
	\begin{align*}
	\int y^{4k} |\hat{f^{\kappa}_{\varepsilon}}(y)|\ dy &= \frac{1}{2\pi} \int y^{4k} \left|\int f^{\kappa}_{\varepsilon}(t) e^{-\i y t}\ dt \right|\ dy \\
	&= \frac{1}{2\pi} \int y^{4k} \left|\int h(t) e^{-\i y \varepsilon t}\ \varepsilon dt \right|\ dy \\
	&= \frac{\varepsilon^{-4k}}{2\pi} \int y^{4k} \left|\int h(t) e^{-\i y t}\ dt \right|\ dy \\
	&= \frac{\varepsilon^{-4k}}{2\pi} \int \frac{1}{1+y^2} \left|\int (h^{(4k)}(t) + h^{(4k+2)}(t)) e^{-\i y t}\ dt \right|\ dy \\
	&\leq  \frac{\varepsilon^{-4k}}{2\pi} \int \frac{1}{1+y^2}\ dy \int ( |h^{(4k)}(t)| + |h^{(4k+2)}(t)|)\ dt \\
	&\leq \varepsilon^{-4k} \left( \norm{h^{(4k)}}_{\CC^0} + \norm{h^{(4k+2)}}_{\CC^0} \right).
	\end{align*}
	
	\noindent In the last line we used the fact the support of the derivatives of $h$ are included in $[\varepsilon^{-1}\alpha, \varepsilon^{-1}\alpha+1] \cup [\varepsilon^{-1}\kappa, \varepsilon^{-1}\kappa +1]$, and thus that their integral is bounded by twice their maximum. Similarly we have that
	\begin{align*}
		\int |y| |\hat{f^{\kappa}_{\varepsilon}}(y)|\ dy &\leq \varepsilon^{-1} \left( \norm{h^{(1)}}_{\CC^0} + \norm{h^{(3)}}_{\CC^0} \right)
	\end{align*}
	
	\noindent We set the function $g:t\to H(2t-1)$ where $H$ is the primitive of the function $h$ defined in \eqref{3fctplateau}, which takes the value $0$ on $\R^-$ and renormalized such that it takes the value $1$ for $x\geq 1$. Then thanks to Theorem \ref{3TTheo} as well as the analysis of the derivatives of $h$ made in the proof of Theorem \ref{3lessopti}, we get that there exist constants $C$ and $c$ such that for any $k\leq cN$, for any $\kappa>\alpha$,
	
	\begin{equation*}
	\E\left[\tr_{N}\Big(f_{\varepsilon}^{\kappa}(PP^*(X^N,A^N))\Big)\right] \leq N\times C^k  \frac{\max(\varepsilon^{-4k},\varepsilon^{-1})}{N^{2k}} k^{12k}.
	\end{equation*}
	
	\noindent Hence the conclusion by dominated convergence theorem.
	
\end{proof}

Consequently, with $x_+ = \max(x,0)$, for any $r>0$,
\begin{align*}
&\E\left[ \left(\norm{PP^*(X^N,A^N)} - \norm{PP^*(x,A^N)}\right)_+ \right] \\
&\leq \int_{0}^{1} \P\left( \norm{PP^*(X^N,A^N)} \geq \norm{PP^*(x,A^N)} + \varepsilon \right)\ d\varepsilon \\
&\quad 	+ \E\left[ \left(\norm{PP^*(X^N,A^N)} - \norm{PP^*(x,A^N)}\right) \1_{\norm{PP^*(X^N,A^N)} \geq \norm{PP^*(x,A^N)}+ 1} \right]\\
&\leq r + \int_{r}^{1} \P\left( \tr_{N}\Big(f_{\varepsilon}(PP^*(X^N,A^N))\Big) \geq 1 \right) d\varepsilon \\
&\quad 	+ \E\left[ \left(\norm{PP^*(X^N,A^N)} - \norm{PP^*(x,A^N)}\right)^2\right]^{1/2} \P\left(\norm{PP^*(X^N,A^N)} \geq \norm{PP^*(x,A^N)}+ 1\right)^{1/2}\\
&\leq r + \int_{r}^{1} \E\left[ \tr_{N}\Big(f_{\varepsilon}(PP^*(X^N,A^N))\Big) \right] d\varepsilon \\
&\quad + \E\left[ \left(\norm{PP^*(X^N,A^N)} - \norm{PP^*(x,A^N)}\right)^2\right]^{1/2} \E\left[\tr_{N}\Big(f_{1}(PP^*(X^N,A^N))\Big)\right]^{1/2}\\
&\leq r + r\times N\times C^k  \left(\frac{r^{-2}}{N}\right)^{2k} k^{12k} + \E\left[ \left(\norm{PP^*(X^N,A^N)} - \norm{PP^*(x,A^N)}\right)^2\right]^{1/2} \times C^{k/2}  N^{-k+1/2} k^{6k} .
\end{align*}

\noindent Besides, thanks to Proposition \ref{3bornenorme} as well as the assumption that the norm of the matrices $A_i^N$ are uniformly bounded, we get that $\E\left[ \left(\norm{PP^*(X^N,A^N)} - \norm{PP^*(x,A^N)}\right)^2\right]$ is bounded independently of $N$ by a constant $K$. Thus, by taking $r = N^{-a}$ for some $a>0$, we get
$$ \E\left[ \left(\norm{PP^*(X^N,A^N)} - \norm{PP^*(x,A^N)}\right)_+ \right] \leq N^{-a} \times \left(1+ N^{1 + 2k (2a-1)} \times C^k k^{12k} \right) + K\times C^{k/2}  N^{-k+1/2} k^{6k}  .$$

\noindent Now we want to pick $a$ and $k$ such that $N^{1 + 2k (2a-1)} \times C^k k^{12k}$ is bounded by $1$ uniformly over $N$ (while keeping in mind that $k$ has to be an integer). It is sufficient to pick $a$ and $k$ such that,
$$ \ln C + \frac{\ln N}{k} + 12 \ln k \leq 2\times(1-2a) \ln N. $$

\noindent We fix $k= \lceil\ln N \rceil$, then we need to pick $a$ such that 
$$ \ln C + 1 + 12 \ln \lceil\ln N\rceil \leq 2\times(1-2a) \ln N. $$

\noindent Which means that we can pick $a = \frac{1}{2} - 4\frac{\ln\ln N}{\ln N}$, and for $N$ large enough,
$$ \E\left[ \left(\norm{PP^*(X^N,A^N)} - \norm{PP^*(x,A^N)}\right)_+ \right] \leq 2 N^{-a} = \frac{2 \ln^4 N}{\sqrt{N}} .$$

Thanks to \cite[Proposition 4.6]{un} we have, with $d$ the number of GUE random matrices, that for $N\geq \ln(d)$, there exist constants $K$ and $D$ such that 
$$ \P\Big( \left|\ \norm{P^*P(X^N,A^N)} - \E\left[\norm{P^*P(X^N,A^N)}\right]\ \right| \geq \delta + K e^{-N} \Big) \leq e^{-N} + e^{-D\delta^2 N} . $$

\noindent Thus, we immediately get that
\begin{align*}
&\P\left( \left(\norm{P(X^N,A^N)} - \norm{P(x,A^N)}\right) \geq \frac{\delta + K e^{-N} + \frac{2 \ln^4 N}{\sqrt{N}}}{\norm{P(x,A^N)}} \right) \\
&\leq \P\left( \left(\norm{P(X^N,A^N)} - \norm{P(x,A^N)}\right)\times \left(\norm{P(X^N,A^N)} + \norm{P(x,A^N)}\right) \geq \delta + K e^{-N} + \frac{2 \ln^4 N}{\sqrt{N}} \right) \\
&\leq \P\left( \norm{P^*P(X^N,A^N)} - \norm{P^*P(x,A^N)} \geq \delta + K e^{-N} + \frac{2 \ln^4 N}{\sqrt{N}} \right) \\
&\leq \P\Big( \norm{P^*P(X^N,A^N)} - \norm{P^*P(x,A^N)} \geq \delta + K e^{-N} \\
&\quad\quad\quad\quad\quad\quad\quad\quad\quad\quad\quad\quad\quad\quad\quad\quad\quad\quad+ \E\left[\left(\norm{P^*P(X^N,A^N)} - \norm{P^*P(x,A^N)}\right)_+\right] \Big) \\
&\leq \P\Big( \norm{P^*P(X^N,A^N)} - \norm{P^*P(x,A^N)} \geq \delta + K e^{-N} + \E\left[\norm{P^*P(X^N,A^N)} - \norm{P^*P(x,A^N)}\right] \Big) \\
&\leq \P\Big( \norm{P^*P(X^N,A^N)} - \E\left[\norm{P^*P(X^N,A^N)}\right] \geq \delta + K e^{-N} \Big) \\
&\leq e^{-N} + e^{-D\delta^2 N} .
\end{align*}

\noindent Hence by replacing $\delta$ by $D^{-1/2} \delta$, we get that there is a constant $C$ such that 
$$ \P\left( \norm{P(X^N,A^N)} - \norm{P(x,A^N)} \geq C\norm{P(x,A^N)}^{-1}\left(\delta + \frac{\ln^4 N}{\sqrt{N}}\right) \right) \leq e^{-N} + e^{-\delta^2 N} . $$

\noindent Finally by replacing $\delta$ by $\frac{\ln^4 N}{\sqrt{N}} \delta$, we get that
$$ \P\left( \frac{\sqrt{N}}{\ln^4 N} \left(\norm{P(X^N,A^N)} - \norm{P(x,A^N)}\right) \geq C \frac{\delta + 1}{\norm{P(x,A^N)}} \right) \leq e^{-N} + e^{-\delta^2 \ln^8N} . $$

\section*{Acknowledgements}

The author would like to thanks his PhD supervisors Beno\^it Collins and Alice Guionnet for proofreading this paper and their continuous help, as well as Mikael de la Salle for helpful discussion. The author was partially supported by a MEXT JASSO fellowship and Labex Milyon (ANR-10-LABX-0070) of Universit\'e de Lyon.

The author would also like to strongly thank all of the reviewers for their numerous comments and suggestions which really improved the clarity and the quality of the paper and its results.

\bibliographystyle{abbrv}

\begin{thebibliography}{10}
	
	\bibitem{cherbin}
	Albeverio, S., Pastur, L., \& Shcherbina, M. (2001). On the 1/n expansion for some unitary invariant ensembles of random matrices. \emph{Communications in Mathematical Physics}, 224(1), 271-305.
	
	\bibitem{anci}
	Ambjørn, J., Chekhov, L., Kristjansen, C. F., \& Makeenko, Y. (1993). Matrix model calculations beyond the spherical limit. \emph{Nuclear Physics B}, 404(1-2), 127-172.
	
	\bibitem{alice}
	Anderson, G. W., Guionnet, A., \& Zeitouni, O. (2010). \emph{An introduction to random matrices} (No. 118). Cambridge university press.
	
	\bibitem{figalli1}
	Bekerman, F., Figalli, A., \& Guionnet, A. (2015). Transport Maps for ${\beta}$-Matrix Models and Universality. \emph{Communications in mathematical physics}, 338(2), 589-619.
	
	\bibitem{borot3}
	Borot, G., Eynard, B., \& Orantin, N. (2015). Abstract loop equations, topological recursion and new applications. \emph{Communications in Number Theory and Physics}, 9(1), 51-187.
	
	\bibitem{borot2}
	Borot, G., \& Guionnet, A. (2013). Asymptotic expansion of beta matrix models in the multi-cut regime. \emph{arXiv preprint arXiv:1303.1045}.
	
	\bibitem{borot1}
	Borot, G., \& Guionnet, A. (2013). Asymptotic expansion of $\beta$-matrix models in the one-cut regime. \emph{Communications in Mathematical Physics}, 317(2), 447-483.
	
	\bibitem{borot5}
	Borot, G., Guionnet, A., \& Kozlowski, K. K. (2016). \emph{Asymptotic expansion of a partition function related to the sinh-model}. Springer.
	
	\bibitem{borot4}
	Borot, G., Guionnet, A., \& Kozlowski, K. K. (2015). Large-$N$ asymptotic expansion for mean field models with Coulomb gas interaction. \emph{International Mathematics Research Notices}, 2015(20), 10451-10524.

	\bibitem{ozabr}
	Brown, N. P., \& Ozawa, N. (2008). \emph{ $\CC^*$-Algebras and Finite-Dimensional Approximations} (Vol. 88). American Mathematical Soc..
	
	\bibitem{further}
	Chekhov, L., \& Eynard, B. (2006). Hermitian matrix model free energy: Feynman graph technique for all genera. \emph{Journal of High Energy Physics}, 2006(03), 014.
	
	\bibitem{betaens2}
	Chekhov, L., \& Eynard, B. (2006). Matrix eigenvalue model: Feynman graph technique for all genera. Journal of High Energy Physics, 2006(12), 026.

	\bibitem{betaens1}
	Chekhov, L. O., Eynard, B., \& Marchal, O. (2011). Topological expansion of the $\beta$-ensemble model and quantum algebraic geometry in the sectorwise approach. \emph{Theoretical and Mathematical Physics}, 166(2), 141-185.
	
	\bibitem{segalaU1}
	Collins, B., Guionnet, A., \& Maurel-Segala, E. (2009). Asymptotics of unitary and orthogonal matrix integrals. \emph{Advances in Mathematics}, 222(1), 172-215.
	
	\bibitem{un}
	Collins, B., Guionnet, A., \& Parraud, F. (2022). On the operator norm of non-commutative polynomials in deterministic matrices and iid GUE matrices. \emph{Cambridge Journal of Mathematics}, 10(1), 195--260.  
	
	\bibitem{debut}
	David, F. (1993). Loop equations and non-perturbative effects in two-dimensional quantum gravity. In \emph{The Large $N$ Expansion In Quantum Field Theory And Statistical Physics: From Spin Systems to 2-Dimensional Gravity} (pp. 798-808).
	
	\bibitem{macl}
	Ercolani, N. M., \& McLaughlin, K. R. (2003). Asymptotics of the partition function for random matrices via Riemann-Hilbert techniques and applications to graphical enumeration. \emph{International Mathematics Research Notices}, 2003(14), 755-820.
	
	\bibitem{eynard1}
	Eynard, B., \& Orantin, N. (2007). Invariants of algebraic curves and topological expansion. \emph{Communications in Number Theory and Physics}, 1(2), 347-452.

	\bibitem{further2}
	Eynard, B., \& Orantin, N. (2005). Topological expansion of the 2-matrix model correlation functions: diagrammatic rules for a residue formula. \emph{Journal of High Energy Physics}, 2005(12), 034.
	
	\bibitem{eynard2}
	Eynard, B. (2005). Topological expansion for the 1-hermitian matrix model correlation functions. \emph{Journal of High Energy Physics}, 2004(11), 031.
	
	\bibitem{figalli2}
	Figalli, A., \& Guionnet, A. (2016). Universality in several-matrix models via approximate transport maps. \emph{Acta mathematica}, 217(1), 81-176.
	
	\bibitem{segala1}
	Guionnet, A., \& Maurel-Segala, E. (2006). Combinatorial aspects of matrix models. \emph{Alea}, 1, 241-279.
	
	\bibitem{segala2}
	Guionnet, A., \& Maurel-Segala, E. (2007). Second order asymptotics for matrix models. \emph{The Annals of Probability}, 35(6), 2160-2212.
		
	\bibitem{segalaU2}
	Guionnet, A., \& Novak, J. (2015). Asymptotics of unitary multimatrix models: the Schwinger–Dyson lattice and topological recursion. \emph{Journal of Functional Analysis}, 268(10), 2851-2905.
    
	\bibitem{HT}
	Haagerup, U., \& Thorbjørnsen, S. (2005). A new application of random matrices: ${\rm Ext}(C^*_{\rm red}(\mathbb F_2))$ is not a group. \emph{Annals of Mathematics}, 711-775.
	
	\bibitem{precurso}
	Haagerup, U., \& Thorbjørnsen, S. (2012). Asymptotic expansions for the Gaussian unitary ensemble. \emph{Infinite Dimensional Analysis, Quantum Probability and Related Topics}, 15(01), 1250003.
	
	\bibitem{harerzag}
	Harer, J., \& Zagier, D. (1986). The Euler characteristic of the moduli space of curves. \emph{Inventiones mathematicae}, 85, 457-485.
	
	\bibitem{debut2}
	Kazakov, V. A. (1989). The appearance of matter fields from quantum fluctuations of 2D-gravity. \emph{Modern Physics Letters A}, 4(22), 2125-2139.
	
	\bibitem{legallfr}
	Le Gall, J. F. (2016). \emph{Brownian motion, martingales, and stochastic calculus} (Vol. 274). New York: Springer.
	
	\bibitem{male}
	Male, C. (2012). The norm of polynomials in large random and deterministic matrices. \emph{Probability Theory and Related Fields}, 154(3), 477-532.
	
	\bibitem{segala3}
	Segala, E. M. (2006). High order asymptotics of matrix models and enumeration of maps. \emph{arXiv preprint math/0608192}.
	
	\bibitem{murphy}
	Murphy, G. J. (1990). \emph{$\CC^*$-Algebras and Operator Theory}. Academic press.
	
	\bibitem{nica_speicher_2006}
	Nica, A., \& Speicher, R. (2006). \emph{Lectures on the combinatorics of free probability} (Vol. 13). Cambridge University Press.
	
	\bibitem{parisi}
	Brézin, E., Itzykson, C., Parisi, G., \& Zuber, J. B. (1993). Planar diagrams. In \emph{The Large N Expansion In Quantum Field Theory And Statistical Physics: From Spin Systems to 2-Dimensional Gravity} (pp. 567-583).
		
	\bibitem{deux}
	Parraud, F. (2022). On the operator norm of non-commutative polynomials in deterministic matrices and iid Haar unitary matrices. \emph{Probability Theory and Related Fields}, 182(3), 751-806.
		
	\bibitem{cherbin2}
	Shcherbina, M. (2014). Asymptotic expansions for $\beta$-matrix models and their applications to the universality conjecture. \emph{Random Matrix Theory, Interacting Particle Systems and Integrable Systems}, 65, 463.
	
	\bibitem{T_ouf}
	T'Hooft, G. (1974). Magnetic monopoles in unified theories. \emph{Nuclear Physics B}, 79(CERN-TH-1876), 276-284.
	
	\bibitem{TW1} 
	Tracy, C. A., \& Widom, H. (1994). Level-spacing distributions and the Airy kernel. \emph{Communications in Mathematical Physics}, 159(1), 151-174.
	
	\bibitem{Vo91}
	Voiculescu, D. (1991). Limit laws for random matrices and free products. \emph{Inventiones mathematicae}, 104(1), 201-220.
	
	\bibitem{refdif}
	Voiculescu, D. (1993). The analogues of entropy and of Fisher's information measure in free probability theory, I. \emph{Communications in mathematical physics}, 155(1), 71-92.
	
	\bibitem{refdif2}
	Voiculescu, D. (2000). The coalgebra of the free difference quotient and free probability. \emph{International Mathematics Research Notices}, 2000(2), 79-106.
	
	\bibitem{zvonski}
	Zvonkin, A. (1997). Matrix integrals and map enumeration: an accessible introduction. \emph{Mathematical and Computer Modelling}, 26(8-10), 281-304.


\end{thebibliography}

\end{document}